\documentclass[reqno]{amsart}
\usepackage[inner=1.0in,outer=1.0in,bottom=1.0in, top=1.0in]{geometry}
\usepackage[utf8]{inputenc}
\usepackage{amsfonts}
\usepackage{amsmath}
\usepackage{amssymb}
\usepackage{amsthm}
\usepackage{hyperref}
\usepackage{ dsfont }
\usepackage{cleveref}
\usepackage{tikz}
\usepackage{tikz-cd}
\usepackage{verbatim}
\usepackage{pgfplots}
\usepackage{textcomp}
\pgfplotsset{compat=newest}
\usepgfplotslibrary{fillbetween}

\DeclareMathOperator{\inv}{^{-1}}

\DeclareMathOperator{\Ocal}{\mathcal{O}}

\DeclareMathOperator{\rr}{\mathbb{R}}
\DeclareMathOperator{\cc}{\mathbb{C}}
\DeclareMathOperator{\nn}{\mathbb{N}}
\DeclareMathOperator{\zz}{\mathbb{Z}}
\DeclareMathOperator{\qq}{\mathbb{Q}}

\DeclareMathOperator{\Sym}{Sym}

\DeclareMathOperator{\Acal}{\mathcal{A}}

\renewcommand{\aa}{\mathbb{A}}

\DeclareMathOperator{\Nred}{N_{red}}

\DeclareMathOperator{\GL}{GL}

\DeclareMathOperator{\Volume}{Volume}

\DeclareMathOperator{\supp}{supp}

\DeclareMathOperator{\SL}{SL}

\DeclareMathOperator{\hh}{\mathbb{H}}
\DeclareMathOperator{\Fcal}{\mathcal{F}}

\DeclareMathOperator{\tr}{Trace}

\DeclareMathOperator{\ad}{ad}
\DeclareMathOperator{\asai}{asai}
\DeclareMathOperator{\sym}{sym}
\DeclareMathOperator{\res}{res}
\DeclareMathOperator{\PGL}{PGL}
\DeclareMathOperator{\qterm}{(1+\vert D\vert^{3/2})}
\newcommand{\Mod}[1]{\ (\mathrm{mod}\ #1)}
\newcommand{\norm}[1]{\left\lVert#1\right\rVert}

\newtheorem{thm}{Theorem}[section]
\newtheorem{prop}[thm]{Proposition}
\newtheorem{lemma}[thm]{Lemma}
\newtheorem{cor}[thm]{Corollary}

\newtheorem{remark}[thm]{Remark}
\newtheorem{conj}[thm]{Conjecture}

\newcommand{\psum}{\sideset{}{'}\sum}
\newcommand{\hashsum}{\mathop{\sum}\limits^{\#}}
\newcommand{\tmod}[1]{\ \left(\text{mod }#1\right)}

\title[Cancellation in Sums of Hecke Eigenvalues over Quadratic Polynomials]{Cancellation in Sums of Hecke Eigenvalues over Quadratic Polynomials and Mass Equidistribution}
\author{Steven Creech}
\address{Department of Mathematics \\ Brown University \\ Providence, RI 02906 \\ USA}
\email{steven\_creech@brown.edu}

\date{\today}

\begin{document}

\begin{abstract}
    We study cancellation in sums of Hecke eigenvalues over irreducible quadratic polynomials over short intervals. In particular, we look at an average over bases of Hecke forms of weight $k$ in the range $\vert k-K\vert<K^\theta$ where $1/3<\theta<1$. We see that when averaged over this family such sums admit square root cancellation. The key new arithmetic input for such a result is a bound on sums of Kloosterman sums over irreducible quadratic polynomials. Then using work of Nelson, we relate such sums to the mass equidistribution conjecture for modular forms on compact arithmetic surfaces, and we show that almost all forms satisfy the mass equidistribution conjecture. Furthermore, such forms will satisfy the conjecture with an effective convergence rate of $k^{-\delta}$ for any $\delta<1/2$.  
\end{abstract}

\maketitle
\section{Introduction}

Suppose we have a Fuchsian group of the first kind, $\Gamma\subseteq \SL_2(\zz)$, $k$ an even positive integer, and a Hecke-(cusp) form $f\in S_k(\Gamma)$. We normalize $f$, so that $\norm{f}_{L^2(\Gamma\backslash\hh)}=\langle f, f\rangle=1$ where $\langle \cdot,\cdot\rangle$ is the usual Petersson inner product given by 
\begin{equation*}
    \langle f,g\rangle = \int_{\Gamma\backslash\hh} f(z)\overline{g(z)}\frac{dxdy}{y^2}.
\end{equation*}
By normalizing in this manner, we define a probability measure on $\Gamma\backslash\hh$ by 
\begin{equation*}\label{eq: Definition of Measure mu_f}
    \mu_f:=\vert f(z)\vert^2y^k\frac{dxdy}{y^2}.
\end{equation*}
The mass equidistribution conjecture of Rudnick and Sarnak \cite{rudnick1994behaviour} claims that as $k\rightarrow\infty$ this will converge (in the weak-$*$ sense) to the uniform measure on our surface. That is if we take a sequence of Hecke forms, $f_{k}$, indexed by the weight
, this will converge to the uniform measure $\Tilde{\nu}$ on $\Gamma\backslash\hh$.

\begin{conj} (Mass Equidistribution Conjecture)\label{conj: Mass Equidistribution}
    \begin{equation*}
        \lim_{k\rightarrow\infty}\mu_{f_k} = \frac{1}{\Volume(\Gamma\backslash\hh)}\frac{dxdy}{y^2}=:\Tilde{\nu}
    \end{equation*}
    where the convergence is in the weak-$*$ sense. That is for $g\in C_0^\infty(\Gamma\backslash\hh)$, we have that 
    \begin{equation*}
        \lim_{k\rightarrow\infty} \mu_{f_{k}} (g) = \lim_{k\rightarrow\infty}\int_{\Gamma\backslash\hh}g(z) \vert f(z)\vert^2y^k\frac{dxdy}{y^2} = \frac{1}{\Volume(\Gamma\backslash\hh)}\int_{\Gamma\backslash\hh}g(z) \frac{dxdy}{y^2} = \Tilde{\nu}(g)
    \end{equation*}
\end{conj}
There are several refinements to this conjecture that one can make such as looking at the rate of convergence and how that rate is affected by the choice of test function $g$. We call this the effective mass equidistribution conjecture which we state as follows:
\begin{conj}(Effective Mass Equidistribution Conjecture)\label{conj: Strong QQUE}
    There is a $\delta>0$ and $T<\infty$ such that for all $g\in C_0^\infty(\Gamma\backslash\hh)$
    \begin{equation*}
        \left\vert\mu_f(g)-\Tilde{\nu}(g)\right\vert\ll k^{-\delta}\norm{g}_{W^{T,2}}
    \end{equation*}
    where $\norm{\cdot}_{W^{T,2}}$ is the Sobolev norm on $\Gamma\backslash\hh$.
\end{conj}

The purpose of this paper is to study this conjecture in the case where $\Gamma$ is a cocompact group arising from a quaternion algebra. We do this by studying the so-called quantum variance. For a family $\Fcal$ of Hecke forms and a smooth compactly supported test function $g\in C_0^\infty(\Gamma\backslash\hh)$, we define the quantum variance over the family $\Fcal$ to be given by 
\begin{equation*}
    \sum_{f\in\Fcal} \vert \mu_f(g)-\Tilde{\nu}(g)\vert^2. 
\end{equation*}
We remark this a variance like sum as we are averaging over the family $\Fcal$ the square of the difference from the measured value of $\mu_f(g)$ and the expected outcome of $\Tilde{\nu}(g)$. In particular, in this paper we shall look at the family $\Fcal(K) = \{f\in H_k(\Gamma): \vert k-K\vert < K^\theta\}$ where we let $H_k(\Gamma)$ denoted a basis of Hecke forms of $S_k(\Gamma)$, $1/3<\theta<1$, and $K$ is a sufficiently large parameter. In particular, we will have that $\Fcal(K)$ is a family of size on the order of $K^{1+\theta}$. Looking at such sums corresponds to sums over short intervals of weights, as we can take the length of the interval to be as short as $K^{1/3+\epsilon}$. By proving such bounds on the quantum variance, \Cref{cor: Quantum Ergodicty}, we will get that effective mass equidistribution holds for almost all Hecke forms for any $\delta<1/2$, \Cref{cor: Quantitative QUE}. Such a rate of convergence is optimal, see for example \cite{LuoSarnakVariance}. We remark that even though we get an almost all statement for the effective version of mass equidistribution, the weaker statement is still open for all forms. 



The statement of mass equidistribution comes from an analogy to the quantum unique ergodicity (QUE) conjecture. The original statement of QUE was given for a general Riemannian manifold $M$. In particular, if we have that $\Psi$ is an eigenfunction of the Laplacian with eigenvalue $\lambda_\Psi$, and if we normalize $\Psi$ so that $\norm{\Psi}_{L^2(M)} = 1$. Then one defines a measure $\mu_\Psi$ on $M$ in the same way we had defined the measure $\mu_f$, then QUE states that as $\lambda\rightarrow\infty$, $\mu_{\Psi}\rightarrow\Tilde{\nu}$ where $\Tilde{\nu}$ is the uniform measure on $M$, and this is weak-$*$ convergence. The intuition of this conjecture is that the Laplacian will be the quantized Hamiltonian, $\Psi$ corresponds to an energy state with energy level $\lambda_\Psi$. Then the measure $\mu_\Psi$ will measure how the energy is distributed on $M$. Thus, we can interpret QUE as saying in a high energy system the energy will equidistribute. 

One might point to the work of \v Snirel\textquotesingle man \cite{SnirelmanErgodicProperties}
, Colin deVerdiere \cite{ColindeVerdiereErgodicite} and Zelditch \cite{ZelditchUniformDistribution} as the beginning of study of these problems. However, they had worked in a very general setting (compact Riemannian manifolds with negative curvature where geodesic flow on the unit cotangent bundle is ergodic), and they had proved what is known as quantum ergodicity. Quantum ergodicity is slightly weaker than QUE as it allows for the removal of a subsequence of density $0$ when taking the limit. 
It was later shown that the statement of QUE turns out to be false for a general Riemannian manifold as shown by Hassell \cite{Billards}.
Rudnick and Sarnak \cite{rudnick1994behaviour} had conjectured that QUE should hold on compact manifolds with negative curvature with this additional restriction Hassell's argument does not work suggesting that negative curvature is the necessary condition needed for QUE to hold. Rudnick and Sarnak consider the case where the surface is an arithmetic surface, so one might call this arithmetic QUE. Since in the case of an arithmetic surface the eigenfunctions, $\Psi$ will be Maass forms, so there is more structure. Now QUE was originally proven for Maass forms on compact surfaces by Lindenstrauss in \cite{LindenstraussQUE}; furthermore, his proof also worked in the case where the surface was not compact, but with the possible case where there is an escape of mass at the cusps. It was then shown by Soundararajan in \cite{SoundMaassForm} that the escape of mass does not occur and this case resolved the conjecture for Maass forms. 

Since the proof of QUE, there have been further questions that have been examined. For example, Jung gave a quantitative version of QUE which gave a rate of convergence of the measure, and connected this to studying the number of nodal domains of Maass forms \cite{JJQQUE}. Young conjectured that QUE should hold when one restricts to looking at geodesics, horocycles, and discs whose radii shrink \cite{YoungQUEThinSets}. In this work, Young manages to prove a version of QUE for Eisenstein series on the vertical geodesic connecting $0$ and $\infty$. Furthermore, he gave a QUE result for a shrinking ball $B_R(w)$ where $w$ is a fixed point of $\Gamma\backslash\hh$, and $R\sim t_\Psi^{-\delta}$ where $\lambda_\Psi = \frac{1}{4}+t_\Psi^2$, and the value of $\delta$ depends on if $\Psi$ is taken to be an Eisenstein series or Hecke-Maass cusp form and whether one assumes the Lindel\"off hypothesis. These values of $\delta$ were then later improved by Humphries \cite{HumprhiesShrinkingSets} where he also shows that QUE fails in such shrinking balls when $R\ll t_\Psi\inv(\log t_\Psi)^A$. Zenz had then studied quantum variance along the vertical geodesics for holomorphic forms \cite{PeterZenzThesis}. 


Through the analogy between Maass forms and modular forms, Luo and Sarnak formed the mass equidistribution conjecture in 
\cite{luo1995quantum} where the weight $k$ becomes the analogy of the Laplace eigenvalue. In \cite{luo2003mass}, they used incomplete Poincar\'e series to show that the mass equidistribution conjecture is equivalent to showing that for any fixed $m\in\zz$, and $\psi\in C_0^\infty(0,\infty)$ that 
\[
\lim_{k\rightarrow\infty}\frac{2\pi^2}{(k-1)L(1,\Sym^2(f))}\sum_{r\geq 1}\lambda_f(r)\lambda_f(r+m)\psi\left(\frac{k-1}{4\pi(r+\frac{m}{2})}\right) = \frac{3}{\pi}\delta_{m,0}\int_0^\infty \psi(y)\frac{dy}{y^2}.
\]
That is there is a connection between mass equidistribution and shifted convolution sums. Furthermore Luo and Sarnak study a quantum variance like sum of the form 
\[
\sum_{K\leq k\leq K+K^{\theta}}\sum_{f\in H_k}\vert\mu_f(g)-\Tilde{\nu}(g)\vert^2\ll_{\epsilon,g}K^{\theta+\epsilon}
\]
where $1/3<\theta<1$. In particular, such a bound follows from a square-root cancellation in the shifted convolution sums over this family, and this will imply that mass equidistribution holds for almost all Hecke cusp forms. 

The mass equidistribution conjecture for holomorphic cusp forms on $\SL_2(\zz)$ was proven by Holowinsky and Soundararajan in \cite{holowinsky2010mass} and their method extends to surfaces containing cusps. Holowinsky \cite{holowinsky2010sieving} and Soundararajan \cite{soundararajan2010weak} both had different methods, but neither of their proofs solved the conjecture, but combining their works, they showed that both their proofs covered all possible cases, and resolved the conjecture. 

Now in the case of mass equidistribution for a compact arithmetic surface, the lack of cusps ends up complicating the problem. In particular, there are no Poincar\'e series for these surfaces, so it is not clear if there is a relationship to shifted convolution sums. However, in recent work by Nelson \cite{nelson2022quadratic}, he ends up being able to relate the mass equidistribution conjecture on compact surfaces arising from quaternion algebras to sums of Hecke eigenvalues along irreducible quadratic polynomials. In particular, his work implies that mass equidistribution on average. In addition, Nelson remarks that the the sieving method of Holowinsky in \cite{holowinsky2010sieving} won't apply to the compact setting while Soundararajan's method should still apply. Thus, it is not unclear how one might resolve the mass equidistribution conjecture in this setting.



In the present paper, we will generalize the result of Luo and Sarnak \cite{luo2003mass} to Hecke forms on compact arithmetic surfaces. In particular, in \Cref{thm: Main Theorem}, we prove a bound for a sum of Hecke eigenvalues along an irreducible quadratic polynomial for modular forms on $\Gamma_0(N)$ which constitutes square root cancellation of this sum. This result is proven by applying the Petersson trace formula, and then bounding the corresponding terms. The key new arithmetical input is studying sums of Kloosterman sums where the first two inputs vary along irreducible quadratic polynomials. This is the topic of \Cref{subsec: Sums of Kloosterman Sums}.  Then by applying the Jacquet-Langlands correspondence and a quantitative version of Nelson's argument, we can relate this to the quantum variance, \Cref{cor: Quantum Ergodicty}. We can then use this to to get effective mass equidistribution holds for almost all forms, \Cref{cor: Quantitative QUE}. Then by using Rudnick's, \cite[Theorem 2]{rudnick2005asymptotic}, classical result of QUE implying the equidistribution of zeros, we will have that almost all forms will have their zeros equidistribute.  

In this text, we shall write $f(x) = O(g(x))$ and $f(x)\ll g(x)$ analogously to mean that there is some constant $C$ such that $\vert f(x)\vert\leq C\vert g(x)\vert$. Furthermore, if the constant $C=C(\epsilon)$ that is it has dependence upon $\epsilon$ (or some other parameter), we shall write $f(x)=O_\epsilon(g(x))$ or $f(x)\ll_\epsilon g(x)$. Furthermore, in this paper we shall write $f(x)\sim g(x)$ to mean that $f(x) \ll g(x)$ and $g(x)\ll f(x)$, and we shall similarly write $f(x)\sim_\epsilon g(x)$ when the implied constant has a dependence upon $\epsilon$. We shall also use the notation that $f(x)=o(g(x))$ if $\lim_{x\rightarrow\infty}\frac{f(x)}{g(x)}=0$.

\subsection{Quaternion Algebras, Modular Forms, and the Jacquet-Langlands Correspondence}
Let $\Acal=\left(\frac{a,b}{\qq}\right)$ be a quaternion algebra over $\qq$ with discriminant $d_{\Acal}$. Thus, we have that $\Acal$ is generated over $\qq$ by $i$ and $j$ with $i^2=a$, $j^2=b$, and $ij+ji=0$. We can view $\Acal$ as a $4$-dimensional $\qq$ algebra with the basis $1,i,j,k:=ij$. We define conjugation in $\Acal$ by $\overline{t+xi+yj+zk}=t-xi-yi-zk$. From this we get the reduced norm map $\Nred:\Acal\rightarrow \qq$ given by  $\Nred(\alpha)=\alpha\overline{\alpha}$ and the trace map $\tr:\Acal\rightarrow \qq$ given by $\tr(\alpha)=\alpha+\overline{\alpha}$.

We shall assume that $\Acal$ splits over $\rr$ that is we may identify the real completion $\Acal_\infty := \Acal\otimes_{\qq}\rr$ with $M_2(\rr)$.  We can explicitly realize this isomorphism by noting that for $\Acal$ to split over $\rr$, then either $a>0$ or $b>0$, so without loss of generality let us assume that $a>0$. Then we may identify $i$ with $\sqrt{a}$ and write $\alpha=t+xi+yj+zij=\xi+\eta j$ where $\xi,\eta\in \qq(\sqrt{a})\subset\rr$. Then we realize the isomorphism of $\Acal_\infty$ and $M_2(\rr)$ by 
\begin{equation*}\label{eq:Isomorphism of Quaternion Algebra and Matrix Algebra}
    \alpha\mapsto \begin{bmatrix}
        \overline{\xi} & \eta\\
        b\eta & \xi
    \end{bmatrix}.
\end{equation*}
Furthermore, we will have this isomorphism respects the trace and norm (determinant) maps. 

Given $R\subset \Acal$ a maximal order, we will have that $R(1)=\{\alpha\in R: \Nred(\alpha)=1\}$ will be a subgroup of $\SL_2(\rr)$ under our identification. The group $R(1)$ will be a Fuchsian group of the first kind; furthermore, the behavior of $R(1)$ vastly depends upon whether $\Acal$ splits over $\qq$ or not. 

In the case that $\Acal$ splits over $\qq$, we will may identify $\Acal$ with $M_2(\qq)$, and $R$ with $M_2(\zz)$ hence we will have that $R(1)$ will be identified with $\SL_2(\zz)$. Thus, the split case gives us a group which is cofinite; however the quotient $R(1)\backslash \hh$ will not be compact. The focus of this paper will be on the other case. If $\Acal$ is non-split over $\qq$, then we will have that $R(1)$ will be a cocompact group.

For the remainder of the paper, we shall assume that $\Acal$ is a quaternion algebra of discriminant $d_{\Acal}$ that is split over $\rr$ and non-split over $\qq$. We will always let $R$ be a maximal order and $R(1)$ we will always identify with the cocompact Fuchsian subgroup of $\SL_2(\rr)$ under the identification that we have described. 

A modular form of weight $k$ for the group $R(1)$ is a holomorphic function $f:\hh\rightarrow \cc$ such that $f(\gamma z)=j(\gamma,z)^kf(z)$. We remark that as the surface $R(1)\backslash\hh$ does not have any cusps, so we do not need a growth condition towards cusps, and it does not make sense to define cusp forms for this surface. 

We may define Hecke operators as follows. For $n\geq 1$, we define the group $R(n)=\{\alpha\in R: \Nred(\alpha)=n\}$, we have that $R(1)$ acts naturally on the space $R(n)$, and so we define our Hecke operators to be averaging operators over the cosets of this action. That is 
\begin{equation*}\label{eq: Definition of Hecke Operators}
    (T_n f)(z)=\sum_{\gamma\in R(1)\backslash R(n)}f(\gamma z).
\end{equation*}
For $(n,d_{\Acal})=1$, we have that the Hecke operators will be self-adjoint, and satisfy 
\begin{equation*}\label{eq: Hecke Relation}
    T_mT_n\sum_{d\mid (m,n)}dT_{\frac{mn}{d^2}}.
\end{equation*}
In particular, for $(m,n)=1$ we will have that $T_mT_n=T_nT_m$. Thus, we can find a simultaneous eigenbasis for the Hecke operators $T_n$. Since $R(1)\backslash\hh$ lacks a cusp, we can't Fourier expand at our cusps. We get around this difficulty by using the so-called Jacquet Langlands correspondence \cite{JacquetLanglands} which roughly says that we can associate a simultaneous Hecke form $f$ on $R(1)\backslash\hh$ with a Hecke cusp form on $\Gamma_0(d_{\Acal})\backslash\hh$, and the Hecke data of these forms will agree. This will allows us to relate our problem to one for modular forms on $\Gamma_0(d_{\Acal})$. For a modular form $f\in S_k(\Gamma_0(N))$. We shall write out the Fourier expansion of $f$ as 
\begin{equation*}\label{eq: Definition of Fourier Expansion of a Cusp Form}
    f(z)=\sum_{n\geq 1}a_f(n)e(nz).
\end{equation*}
If we assume that $f$ is a Hecke cusp form, we have that the Fourier coefficients $a_f(n)$ are related to the Hecke eigenvalues $\lambda_f(n)$ by 
\begin{equation*}
    \lambda_f(n)=\frac{a_f(n)}{a_f(1)n^{\frac{k-1}{2}}}.
\end{equation*}





One of the main tools that we will be using will be the Petersson trace formula on $\Gamma_0(N)$. This will allow us to relate a sum of Fourier coefficients to a sum of Kloosterman sums and Bessel functions. We state the trace formula here for convenience. 

\begin{thm}[Petersson Trace Formula]\label{thm: Petersson Trace Formula}
    Let $H_k(\Gamma_0(N))$ be an orthogonal basis for $S_k(\Gamma_0(N))$, then for $m,n\geq 1$, we have that 
    \begin{equation}\label{eq: Petersson Trace Formula}
        \sum_{f\in H_k}\frac{2\pi^2}{kL(1,\sym^2 f)}\lambda_f(m)\lambda_f(n)=\delta_{m,n}+2\pi i^{-k}\sum_{\substack{c>0\\c\equiv 0\Mod{N}}}c\inv S(m,n;c)J_{k-1}\left(\frac{4\pi\sqrt{mn}}{c}\right).
    \end{equation}
    Where we have that $\delta_{m,n}$ is $1$ if $m=n$ and $0$ if $m\neq n$. Furthermore, we have that $S(m,n;c)$ represents the Kloosterman sum which is defined by
    \begin{equation*}\label{eq: Definition of a Kloosterman Sum}
        S(m,n;c)=\psum_{x\Mod{c}}e\left(\frac{mx+n\overline{x}}{c}\right).
    \end{equation*}
    Where we use $\psum$ to represent summing over only the class of $x\Mod{c}$ where $(x,c)=1$ and $\overline{x}$ represents the multiplicative inverse of $x$ modulo $c$. Lastly, we have that $J_{k-1}$ represents the $J$-Bessel function. 
\end{thm}

\begin{remark}\label{rem: Kloosterman Bounds}
    Kloosterman sums always evaluate to real numbers. In this paper, there will be two main bounds we use for Kloosterman sums that we use. The first bound we refer to as the trivial bound which simply states that $\vert S(m,n;c)\vert \leq c$. Equivalently, we will have that $c\inv \vert S(m,n;c)\vert \leq 1$. 

    The second bound is the Weil bound for Kloosterman sums, \cite{WeilBound}, which states that 
    \[
    \vert S(m,n;c)\vert \leq \sqrt{(m,n,c)}c^{1/2}\tau(c).
    \]
    In particular, since $\tau(c)\ll_\epsilon c^\epsilon$, and $(m,n,c)\leq m$. We have that 
    \[
    S(m,n;c)\ll_\epsilon m^{1/2}c^{1/2+\epsilon}.
    \]
\end{remark}


\subsection{Statements of the Main Theorems}




\begin{thm}\label{thm: Main Theorem}
    Fix real $\theta$ and $\epsilon$ such that $1/3<\theta<1$ and $\epsilon>0$. Then pick $\epsilon_0<\min\left\{\frac{3\theta-1}{20},\frac{1-\theta}{2},\frac{\epsilon}{10}\right\}$, and let $q(x)=Ax^2+Bx+C\in \qq[x]$ be an irreducible integer valued quadratic polynomial such that $A\ll X^{\epsilon_0}$, $B\ll X^{1/2}$, and $C\ll X^{1-\epsilon_0}$, we shall let $D=B^2-4AC\ll X$ be the discriminant. Let $H_k(\Gamma_0(N))$ an orthogonal basis of Hecke cusp forms for $S_k(\Gamma_0(N))$, for $f\in H_k(\Gamma_0(N))$ we let $\lambda_f(r)$ denote the $r$-th Hecke eigenvalue of $f$. Let $X$ be some parameter where $1\ll X\ll K$. Assume that $\psi\in C_0^\infty(0,\infty)$ is supported on $(1/l,l)\subset (0,\infty)$ for some fixed $l>0$. Then there exists $T>0$ depending only upon $\theta$ and $\epsilon$ such that 
    \begin{equation*}
        \frac{1}{X}\sum_{\vert k-K\vert<K^\theta}\sum_{f\in H_k(\Gamma_0(N))}\frac{1}{kL(1,\sym^2(f))}\left\vert \sum_{r\geq 1}\lambda_f(\vert q(r)\vert)\psi\left(\frac{r}{X}\right)\right\vert^2\ll_{l,\epsilon,\theta}\qterm K^{\theta+\epsilon}\norm{\psi}_{W^{T,\infty}}^2
    \end{equation*}
    holds uniformly for $A\ll X^{\epsilon_0}$, $B\ll X^{1/2}$, and $C\ll X^{1-\epsilon_0}$. We can take $T=\max\left\{\frac{100}{3\theta-1},\frac{100}{\epsilon_2-\epsilon_0}\right\}$ where $\epsilon_2$ will be some parameter chosen such that $0<\epsilon_2-\epsilon_0$, and $\epsilon_2<\frac{2}{5}\epsilon$. In particular, we may choose $\epsilon_2$ so that $\frac{100}{\epsilon_2-\epsilon_0}<\frac{1000}{\epsilon}$ (in particular, pick $\epsilon_2=\frac{\epsilon}{5}$).
\end{thm}

\begin{remark}\label{rem: Bound on Coefficients and x}
    We remark that the conditions upon the coefficients of the polynomial are needed as later we shall have that $r_i\sim_l X$, and we would like to say that $q(r_i)\sim_l X^2$. However, for this to be the case, we would need the leading coefficient to be bounded by some constant, but in the course of the proof of \Cref{cor: Quantum Ergodicty}, we will only know that the leading coefficient of the polynomial is bounded by $X^{\epsilon_0}$ for some small $\epsilon_0$. Thus, we have to get away with allowing the leading term being bounded by $X^{\epsilon_0}$. The key use of this is that later on we shall define $x:=\frac{4\pi\sqrt{\vert q(r_1)q(r_2)\vert}}{c}$. Thus, we will have that $\frac{r_1r_2}{c}\ll x\ll X^{\epsilon_0}\frac{r_1r_2}{c}$. Furthermore, we will be working in the region where $r_i\sim_l X$, and in this case, we will have $\frac{X^2}{c}\ll_l x\ll_l\frac{X^{2+\epsilon_0}}{c}$.
\end{remark}

\begin{remark}
    Our theorem is truly one about a square-root cancellation that occurs in the inner sum over $r$ when averaged over the weights $k$ and bases $H_k(\Gamma_0(N))$. Since if we were to trivially bound the inner sum as it is of length $O(X)$, then we have that trivially the sum is $O_{l,\epsilon,\theta,\psi}(XK^{\theta+\epsilon})$. However, if we have square-root cancellation in the inner sum, so that it is of order $O(\sqrt{X})$, then after squaring the sum and trivially bounding everything else will give the bound $O_{l,\epsilon,\theta,\psi}(K^{\theta+\epsilon})$ which is our theorem. 
\end{remark}

\begin{remark}
    In the above, we have that $\norm{\psi}_{W^{T,\infty}}$ denotes the Sobolev norm, that is for $\psi\in C_0^\infty(0,\infty)$, we have that 
    \begin{equation}\label{eq: Sobolev Norm Def}
        \norm{\psi}_{W^{T,\infty}}=\sum_{j=0}^T\sup_{x\in (0,\infty)}\left\vert \frac{d^j}{dx^j}\psi(x)\right\vert.
    \end{equation}
    In practice, we know that as $\psi\in C_0^\infty(0,\infty)$, it is going to be a bounded function, and while $\norm{\psi}_{L^\infty}$ captures the supremum of $\psi$, in practice, we will be taking derivatives of the function $\psi$, and so we wish to not capture the supremum of $\psi$, but of all the of the derivatives we take. Thus, the variable $T$ for the Sobolev norm that we are taking will be the largest number of derivatives that we will need to take of $\psi$. We remark that by including the Sobolev norm, we actually capture the dependence of the function $\psi\in C_0^\infty(0,\infty)$ that we pick (rather than having the arbitrary constant in the $\ll$ symbol also depend upon $\psi$). 
\end{remark}


As a corollary of \cref{thm: Main Theorem}, we can get a bound on the quantum variance in this interval. In particular, we shall prove the following. 

\begin{cor}\label{cor: Quantum Ergodicty}
    For $R(1)$ a cocompact group arising from a quaternion algebra, and $g\in C_0^\infty(R(1)\backslash\hh)$, we have that for $\epsilon>0$ and $1/3<\theta<1$, there is some value of $T$ depending only upon $\epsilon$ and $\theta$, such that 
    \begin{equation*}
        \sum_{\vert k-K\vert < K^{\theta}}\sum_{f\in H_k(R(1))}\vert \mu_f(g)-\Tilde{\nu}(g)\vert^2\ll_{\Acal,\epsilon,\theta} \norm{g}^2_{W^{T,2}(R(1)\backslash\hh)}K^{\theta+\epsilon}
    \end{equation*}
    where we may take $T = \max\left\{21 + \frac{100}{3\theta-1}, 21 + \frac{2000}{\epsilon}\right\}$.
\end{cor}

The proof strategy of \cref{cor: Quantum Ergodicty} will be to apply the argument of Nelson \cite{nelson2022quadratic}; however, we must make his argument quantitative, this is the topic of \Cref{sec: Application}. Once we have this quantum variance like bound, we can use it to get a bound on the rate of convergence for almost all forms that is we get the following bound. 


\begin{cor}[Effective Mass Equidistribution]\label{cor: Quantitative QUE}
    For $\epsilon>0$ and $1/3<\theta<1$, then letting $T$ be as in \Cref{cor: Quantum Ergodicty}, then for $g\in C_0^\infty(R(1)\backslash\hh)$, and any $\delta<\frac{1}{2}$, then as $k\rightarrow \infty$ we have that almost all Hecke forms $f$ of weight $k$ satisfy 
    \begin{equation*}\label{eq: Qunatitative QUE}
        \vert \mu_f(g)-\Tilde{\nu}(g)\vert \ll_{\Acal,\delta,\theta} \frac{\norm{g}_{W^{T,2}}}{k^{\delta}}
    \end{equation*}
\end{cor}


In addition to the improvement in quantum variance, we can also apply the work of \cite[Theorem 2]{rudnick2005asymptotic} on the relationship between mass equidistribution and the equidistribution of zeros of Hecke-forms to get an equidistribution result for zeros of modular forms on $R(1)\backslash\hh$. 

\begin{cor}\label{cor: Equid of Zeros}
    The zeros of almost every Hecke form on $R(1)\backslash\hh$ equidistribute with respect to the uniform measure on the surface $R(1)\backslash\hh$.
\end{cor}

\subsection{Outline of the Paper} We have broken the paper up into two main sections and an appendix. In \Cref{sec: Proof} we prove \Cref{thm: Main Theorem}, we do this in several steps. In \Cref{subsec: Smoothing the Sum}, we apply a smoothing technique to the sum in the main theorem and then apply the Petersson trace formula breaking up the sum into a diagonal piece and off diagonal piece. Continuing in \Cref{subsec:Diagonal Contribution}, we bound the diagonal contribution. In \Cref{subsec: Breaking Up the Off-Diagonal} we take the diagonal contribution and break it up into three parts depending on if the value of $c$ is small, medium, or large. Furthermore, in the case when $c$ is not large, we will use a lemma that replaces a sum of Bessel functions with an oscillatory integral. In \Cref{subsec: Off-Diagonal Tail}, we bound the $c$ large case by exploiting the fact that in this range the $J$-Bessel function appearing gives a lot of decay. In \Cref{subsec: Off Diagonal Middle}, we bound the $c$ medium case by Taylor expanding the trigonometric functions in the oscillatory integral and applying integration by parts to get decay. In \Cref{subsec: Off Diagonal Main}, we have to use a higher order Taylor expansion of the oscillatory integral where we then apply the principle of stationary phase to get a main term which we must then use a bound of sums of Kloosterman sums to conclude bounding the main term and concluding the theorem. In \Cref{subsec: Sums of Kloosterman Sums}, we provide new bounds on sums of Kloosterman sums along irreducible quadratic polynomials. In \Cref{subsec: Bounds on Derivatives} and \Cref{subsec: Proof of Lemma} we give bounds on derivatives that are used in bounding the term coming from the main term after we applied stationary phase. 

In \Cref{sec: Application}, we use our main theorem to prove \Cref{cor: Quantum Ergodicty} by giving a quantitative approach to the work of Nelson \cite{nelson2022quadratic} and using the Jacquet-Langlands correspondence. The main idea of this section is taking our general test function $g\in C_0^\infty(R(1)\backslash\hh)$ and spectrally expanding it in terms of Maass forms $\Psi$, then if we have the bound for the Maass forms, then putting them together will give it for a general test function. However, in doing so we must know \Cref{cor: Quantum Ergodicty} for Maass forms $\Psi$, and how this bound will depend upon the Laplace eigenvalue $\lambda_{\Psi}$. This will be the content of \Cref{thm: Maass Form QUE} which is the majority of this section. Then using our quantitative version of Nelson's argument, the Jacquet-Langlands correspondence, and \Cref{thm: Main Theorem} will give us this result. At the end of this section, we then deduce how \Cref{cor: Quantitative QUE} follows from \Cref{cor: Quantum Ergodicty}. 

We also include an appendix \Cref{appendix: Bounds on Special Values} which just gives the dependence of $\lambda_{\Psi}$ of some of the bounds used by Nelson. 

Furthermore, in this paper, we choose to label Lemmas and Propositions in the convention that we prove all Propositions in the paper while Lemmas are just stated along with a reference to its proof. 

\subsection*{Acknowledgments} The author would like to thank Junehyuk Jung for countless helpful conversations and suggesting this project. The author would also like to thank Jeff Hoffstein, Henry Twiss, Zhining Wei, and Peter Zenz for helpful conversations.

\section{Mass Equidistirbution}\label{sec: Proof}

\subsection{Smoothing the Sum and Reduction Using Petersson Trace Formula}\label{subsec: Smoothing the Sum} Our first step in proving \Cref{thm: Main Theorem} will be to apply the Petersson trace formula, \Cref{eq: Petersson Trace Formula}. We will begin by smoothing out the sum by introducing a function $u\in C_0^\infty(\rr)$, so that we may sum over all even weights $k$. In practice, we may choose $u$ to be a smooth and compactly supported such that $u\equiv 1$ on the interval $[-1,1]$ whose support is contained in say $[-1.1,1.1]$. Thus, we will have that 
\[
u\left(\frac{k-K}{K^\theta}\right)=\begin{cases}
    1 & \vert k-K\vert<K^\theta\\
    x & 1\leq \frac{\vert k-K\vert}{K^\theta}\leq 1.1\\
    0 & \frac{\vert k-K\vert}{K^\theta} > 1.1
\end{cases},
\]
where in the above equation $x\in [0,1]$ is just an arbitrary value chosen so that $u$ is smooth. The advantage of smoothing out our sum is that rather than summing over just the values of $k$ where $\vert k-K\vert<K^\theta$, we can sum over all values of $k$, and as we may only be summing over more values we will have that

\begin{align}
    \begin{split}\label{eq: Smoothing Sum}
    &\frac{1}{X}\sum_{\substack{\vert k-K\vert<K^\theta\\ 2\mid k}}\sum_{f\in H_k(\Gamma_0(N))}\frac{1}{kL(1,\sym^2f)}\left\vert \sum_{r\geq 1}\lambda_f(\vert q(r)\vert)\psi\left(\frac{r}{X}\right)\right\vert^2\\
    &\leq \frac{1}{X}\sum_{\substack{k\geq 1\\ 2\mid k}}u\left(\frac{k-K}{K^\theta}\right)\sum_{f\in H_k(\Gamma_0(N))}\frac{2\pi^2}{kL(1,\sym^2(f))}\left\vert \sum_{r\geq 1}\lambda_f(\vert q(r)\vert)\psi\left(\frac{r}{X}\right)\right\vert^2.
    \end{split}
\end{align}

Furthermore, we will apply integration by parts on the function $u$ to pick up factors of $K^{-\theta}$ which will be useful in establishing our bounds. We remark that $u\notin C^\infty_0(0,\infty)$ as its support contains negative values; however, there a related function $g\in C^\infty_0(0,\infty)$ which is just a shift of $u$, we define  
\begin{equation}\label{eq: Definition of g as a shift of u}
g(\xi):=u\left(\frac{\xi+1-K}{K^\theta}\right).
\end{equation}

We note that this function $g$ is unrelated to the arbitrary function $g\in C_0^\infty(R(1)\backslash\hh)$ that was mentioned in the introduction. That use of $g$ is not used in \Cref{sec: Proof}, and as we have defined this function $g(\xi)$ will only appear in this section. We remark that for $K$ large enough, we will have that $g\in C_0^\infty(0,\infty)$. To check this, we need that for $K$ large enough we have $g(0)=0$. Which is true as $g(0)=u\left(\frac{1-K}{K^\theta}\right)$, and for $K$ large enough, we will have that $\frac{1-K}{K^\theta}<-1.1$, and so $g(0)=0$. Hence we will have that $g\in C_0^\infty(0,\infty)$.

This will allow us to apply the Petersson trace formula. We state this as the following.

\begin{prop}\label{prop: Application of Petersson Trace Formula}
    Fixing our function $u$ as above. We may rewrite the right hand side of \Cref{eq: Smoothing Sum} as $I^{diag}+I^{OD}$ where $I^{diag}$ is the diagonal term and $I^{OD}$ is the off-diagonal term. These terms are given explicitly by the formulas
    \begin{equation*}\label{eq: Diagonal Part}
        I^{diag}=\frac{1}{X}\sum_{\substack{r_1,r_2\geq 1\\ \vert q(r_1)\vert=\vert q(r_2)\vert}} \sum_{\substack{k\geq 1\\2\mid k}}u\left(\frac{k-K}{K^\theta}\right)\psi\left(\frac{r_1}{X}\right)\psi\left(\frac{r_2}{X}\right)
    \end{equation*}
    and
    \begin{align*}
        \begin{split}\label{eq: Off Diagonal Part}
        I^{OD}&=\frac{1}{X}\sum_{r_1,r_2\geq 1}\sum_{\substack{k\geq 1\\2\mid k}}u\left(\frac{k-K}{K^\theta}\right)2\pi i^{-k}\\
        &\times \sum_{\substack{c>0\\c\equiv0\Mod{N}}} \frac{S(\vert q(r_1)\vert,\vert q(r_2)\vert;c)}{c}J_{k-1}\left(\frac{4\pi\sqrt{\vert q(r_1)q(r_2)\vert}}{c}\right)\psi\left(\frac{r_1}{X}\right)\psi\left(\frac{r_2}{X}\right).
        \end{split}
    \end{align*}
    Thus, we will have that \Cref{thm: Main Theorem} follows by showing $I^{diag}$ and $I^{OD}$ are $O_{l,\epsilon,\theta}(\qterm  K^{\theta+\epsilon}\norm{\psi}^2_{W^{T,\infty}})$.
\end{prop}

\begin{proof}
    After smoothing our sum, we can open up the square by multiplying the summand by its conjugate to get that the right hand side of \Cref{eq: Smoothing Sum} is given by
    \[
    \frac{1}{X}\sum_{\substack{k\geq 1\\2\mid k}}u\left(\frac{k-K}{K^{\theta}}\right)\sum_{f\in H_k(\Gamma_0(N))}\frac{2\pi^2}{kL(1,\sym^2(f))}\sum_{r_1,r_2\geq 1}\lambda_f(\vert q(r_1)\vert)\overline{\lambda_f(\vert q(r_2)\vert)}\psi\left(\frac{r_1}{X}\right)\psi\left(\frac{r_2}{X}\right).
    \]
    Now if we apply the Petersson trace formula to the above we arrive at
    \begin{align*}
        \begin{split}
            &\frac{1}{X}\sum_{\substack{k\geq 1\\2\mid k}}u\left(\frac{k-K}{K^{\theta}}\right)\times\\
            & \sum_{r_1,r_2\geq 1}\left(\delta_{\vert q(r_1)\vert,\vert q(r_2)\vert}+2\pi i^{-k}\sum_{\substack{c>0\\c\equiv0\Mod{N}}}\frac{S(\vert q(r_1)\vert,\vert q(r_2)\vert;c)}{c}J_{k-1}\left(\frac{4\pi\sqrt{\vert q(r_1)q(r_2)\vert}}{c}\right)\right)\psi\left(\frac{r_1}{X}\right)\psi\left(\frac{r_2}{X}\right).
        \end{split}
    \end{align*}
    Interchanging the sum over $r_1$ and $r_2$, and the sum over $k$ we arrive at 
    \begin{align*}
        \begin{split}
            &\frac{1}{X}\sum_{\substack{r_1,r_2\geq 1\\ \vert q(r_1)\vert=\vert q(r_2)\vert}} \sum_{\substack{k\geq 1\\2\mid k}}u\left(\frac{k-K}{K^\theta}\right)\psi\left(\frac{r_1}{X}\right)\psi\left(\frac{r_2}{X}\right)\\
            &+\frac{1}{X}\sum_{r_1,r_2\geq 1}\sum_{\substack{k\geq 1\\2\mid k}}u\left(\frac{k-K}{K^\theta}\right)2\pi i^{-k}\sum_{\substack{c>0\\c\equiv0\Mod{N}}} \frac{S(\vert q(r_1)\vert,\vert q(r_2)\vert;c)}{c}J_{k-1}\left(\frac{4\pi\sqrt{\vert q(r_1)q(r_2)\vert}}{c}\right)\psi\left(\frac{r_1}{X}\right)\psi\left(\frac{r_2}{X}\right),
        \end{split}
    \end{align*}
    which is precisely $I^{diag}+I^{OD}$.
\end{proof}

\subsection{Diagonal Contribution}\label{subsec:Diagonal Contribution}

\begin{prop}\label{lem: Bound for diagonal contribution}
    We have that $I^{diag}=O_l(K^{\theta}\norm{\psi}^2_{L^\infty})$. In particular, we will have \\$I^{diag}=O_{l,\epsilon,\theta}(\qterm K^{\theta+\epsilon}\norm{\psi}^2_{W^{T,\infty}})$.
\end{prop}

\begin{proof}
    To get the bound on $I^{diag}$, we observe that since $q$ is a quadratic polynomial, for any fixed value $r_1$, there can be at most $4$ values of $r_2$ such that $\vert q(r_1)\vert=\vert q(r_2)\vert$. Thus, in the asymptotic, we may ignore the sum over $r_2$. Furthermore, as $u\in C_0^\infty(\rr)$ and $\psi\in C_0^\infty(0,\infty)$. We may sum over just those $r_1$ with $r_1\ll_l X$ and the $k$ such that $\vert k-K\vert\ll K^{\theta}$ (thus, there will be $O(K^\theta)$ such terms in the sum over $k$). Thus, we have that 
    \[
    I^{diag}\ll \frac{1}{X}\sum_{r_1\ll_l X}\sum_{\vert k-K\vert \ll K^\theta} \norm{\psi}^2_{L^\infty}\ll_l K^\theta\norm{\psi}^2_{L^\infty}.\qedhere
    \]
\end{proof}


\subsection{Breaking Up the Off-Diagonal Contribution}\label{subsec: Breaking Up the Off-Diagonal}
Now we wish to bound $I^{OD}$. However, we shall do this by breaking it up as $I^{OD}_{main}+I^{OD}_{mid}+I^{OD}_{tail}$ that depend upon the size of $c$. In particular, we will pick $0<\epsilon_1<\frac{3\theta-1}{12}$. Thus, as the value of $\epsilon_1$ is dependent upon $\theta$, later on when we have an $\epsilon_1$ dependence upon our constant, that is actually a dependence upon the value of $\theta$. Thus, we break up $I^{OD}$ as:
 \begin{align*}
    \begin{split}\label{eq: Off Diagonal Part main}
        I^{OD}_{main}&=\frac{1}{X}\sum_{r_1,r_2\geq 1}\sum_{\substack{k\geq 1\\2\mid k}}u\left(\frac{k-K}{K^\theta}\right)2\pi i^{-k}\\
        &\times \sum_{\substack{c\ll X^{2+\epsilon_0}K^{-1-\theta+\epsilon_1}\\c\equiv0\Mod{N}}}c\inv S(\vert q(r_1)\vert,\vert q(r_2)\vert;c)J_{k-1}\left(\frac{4\pi\sqrt{\vert q(r_1)q(r_2)\vert}}{c}\right)\psi\left(\frac{r_1}{X}\right)\psi\left(\frac{r_2}{X}\right),
    \end{split}
\end{align*}

     \begin{align*}
     \begin{split}
        I^{OD}_{mid}&=\frac{1}{X}\sum_{r_1,r_2\geq 1}\sum_{\substack{k\geq 1\\2\mid k}}u\left(\frac{k-K}{K^\theta}\right)2\pi i^{-k}\\
        &\times \sum_{\substack{X^{2+\epsilon_0}K^{-1-\theta+\epsilon_1}\ll c\ll K^{10}\\c\equiv0\Mod{N}}}c\inv S(\vert q(r_1)\vert,\vert q(r_2)\vert;c)J_{k-1}\left(\frac{4\pi\sqrt{\vert q(r_1)q(r_2)\vert}}{c}\right)\psi\left(\frac{r_1}{X}\right)\psi\left(\frac{r_2}{X}\right),
    \end{split}
    \end{align*}
    and

     \begin{align*}
     \begin{split}
        I^{OD}_{tail}&=\frac{1}{X}\sum_{r_1,r_2\geq 1}\sum_{\substack{k\geq 1\\2\mid k}}u\left(\frac{k-K}{K^\theta}\right)2\pi i^{-k}\\
        &\times \sum_{\substack{c\gg K^{10}\\c\equiv0\Mod{N}}}c\inv S(\vert q(r_1)\vert,\vert q(r_2)\vert;c)J_{k-1}\left(\frac{4\pi\sqrt{\vert q(r_1)q(r_2)\vert}}{c}\right)\psi\left(\frac{r_1}{X}\right)\psi\left(\frac{r_2}{X}\right).
        \end{split}
    \end{align*}

Furthermore, there is one additional step that we shall take here, we rewrite the $I^{OD}_{main}$ and $I^{OD}_{mid}$ terms by using the following lemma which can be found on page 86 in \cite{iwaniec1997topics} (where in their notation we are taking $\alpha =1/4$ to get our statement of this lemma). 

\begin{lemma}\label{lem: Sum of Bessel Functions}
For $g\in C_0^\infty(0,\infty)$, we have that 
\begin{equation*}\label{eq: Sum of Bessel Functions}
\sum_{\substack{k\geq 1\\ 2\mid k}} 2\pi i^k J_{k-1}(x)g(k-1) = -2\pi \int_{-\infty}^\infty \hat{g}(t)\sin(x\cos(2\pi t))dt
\end{equation*}
where 
\begin{equation}\label{eq: Fourier Transform Definiton}
\hat{g}(t)=\int_{-\infty}^\infty g(y)e(ty)dy.
\end{equation}
\end{lemma}
In this notation, we remark that Parseval's identity will state that 
\begin{equation*}\label{eq: Parseal's Identity}
    \int_{-\infty}^\infty \hat{g}(t)f(t)dt=\int_{-\infty}^\infty g(t)\hat{f}(t)dt.
\end{equation*}
We see that using the function $g(k-1)$ appears in $I^{OD}$ using our definition of $g$ given in \Cref{eq: Definition of g as a shift of u}. Thus, applying \Cref{lem: Sum of Bessel Functions} to $I^{OD}$, we see that if we denote $x:=\frac{4\pi\sqrt{\vert q(r_1)q(r_2)}}{c}$, we have
\[
I^{OD}=-2\pi\frac{1}{X} \sum_{r_1,r_2\geq 1}\psi\left(\frac{r_1}{X}\right)\psi\left(\frac{r_2}{X}\right)\sum_{\substack{c>0\\
c\equiv 0\Mod{N}}}c\inv S(\vert q(r_1)\vert,\vert q(r_2)\vert;c)\int_{-\infty }^\infty \hat{g}(t)\sin(x\cos(2\pi t))dt.
\]
Breaking this sum up into the corresponding values of $c$ for $I^{OD}_{mid}$ and $I^{OD}_{main}$ we see that
\begin{align}
\begin{split}\label{eq: Main part of the off-diagonal term}
    I^{OD}_{main} &= -2\pi \frac{1}{X}\sum_{r_1,r_2\geq 1}\psi\left(\frac{r_1}{X}\right)\psi\left(\frac{r_2}{X}\right)\\
    &\times \sum_{\substack{c\ll X^{2+\epsilon_0}K^{-1-\theta+\epsilon_1}\\
c\equiv 0\Mod{N}}}c\inv S(\vert q(r_1)\vert,\vert q(r_2)\vert;c)\int_{-\infty }^\infty \hat{g}(t)\sin(x\cos(2\pi t))dt
\end{split}
\end{align}
and 
\begin{align}
\begin{split}\label{eq: Middle of the off-Diagonal Term}
    I^{OD}_{mid}&=-2\pi \frac{1}{X}\sum_{r_1,r_2\geq 1}\psi\left(\frac{r_1}{X}\right)\psi\left(\frac{r_2}{X}\right)\\
    &\times \sum_{\substack{X^{2+\epsilon_0}K^{-1-\theta+\epsilon_1}\ll c\ll K^{10}\\
c\equiv 0\Mod{N}}}c\inv S(\vert q(r_1)\vert,\vert q(r_2)\vert;c)\int_{-\infty }^\infty \hat{g}(t)\sin(x\cos(2\pi t))dt.
\end{split}
\end{align}

\subsection{Off-Diagonal Tail Contribution}\label{subsec: Off-Diagonal Tail}
The goal of this section will be to bound the tail term of $I^{OD}_{tail}$. The main idea behind this term is that we shall use a bound for the $J$-Bessel function when it has a small input. We shall use the following bound for the $J$-Bessel function with a small input. 

    \begin{prop}\label{lem: Bound of J-Bessel Function for small input}
        Let $k\in\nn$, and $0<x<1$, then we have that 
        \begin{equation*}\label{eq: J-Bessel Bound for small input}
        J_{k}(x)\ll x^k.
    \end{equation*}
    In particular, for $c\gg K^{10}$, we will have that 
    \begin{equation*}
        J_{k-1}\left(\frac{4\pi \sqrt{\vert q(r_1)q(r_2)\vert}}{c}\right)\ll_l \left(\frac{X^{2+\epsilon_0}}{c}\right)^{k-1}
    \end{equation*}
    \end{prop}

    \begin{proof}
        We recall that the $J$-Bessel function has a series expansion given by
        \begin{equation*}\label{eq: J-Bessel function Series Expansion}
            J_k(x)=\left(\frac{x}{2}\right)^k\sum_{n=0}^\infty (-1)^n\frac{(x/2)^{2n}}{n!\Gamma(k+n+1)}.
        \end{equation*}
        Thus, we shall show that the sum converges which is clear as $0<x<1$ and $k\in\zz$, so we have that 
        \begin{align*}
            \left\vert\sum_{n=0}^\infty (-1)^n\frac{(x/2)^{2n}}{n!\Gamma(k+n+1)}\right\vert\leq\sum_{n=0}^\infty\frac{(x/2)^{2n}}{n!(k+n)!} \leq \sum_{n=0}^\infty \frac{1}{n!}\leq e.
        \end{align*}
        Now from \Cref{rem: Bound on Coefficients and x}, we have $x\ll_l\frac{X^{2+\epsilon_0}}{c}$, then as $c\gg K^{10}$ and $1\ll X\ll K$, we will have that 
        \[
        \frac{X^{2+\epsilon_0}}{c}\ll K^{-8+\epsilon_0}.
        \]
        Thus, we may apply our bound for the Bessel-function to conclude that 
        \[
        J_{k-1}\left(\frac{4\pi \sqrt{\vert q(r_1)q(r_2)\vert}}{c}\right)\ll_l \left(\frac{X^{2+\epsilon_0}}{c}\right)^{k-1}.\qedhere
        \]
    \end{proof}

    Now applying our bound for the the Bessel function to $I^{OD}_{tail}$ and applying the Weil bound for Kloosterman sums (which in this case will imply that $S(\vert q(r_1)\vert,\vert q(r_2)\vert; c)\ll_{l,\epsilon}X^{\epsilon_0}r_1c^{\frac{1}{2}+\epsilon}$), we see that 
    \[
    I^{OD}_{tail}\ll_{l,\epsilon} \frac{1}{X}\sum_{r_1,r_2\geq 1}\sum_{\substack{k\geq 1\\
    2\mid k}}u\left(\frac{k-K}{K^\theta}\right)\sum_{\substack{c\gg K^{10}\\ c\equiv 0\Mod{N}}}X^{\epsilon_0}r_1c^{\epsilon-\frac{1}{2}}\left(\frac{X^{2+\epsilon_0}}{c}\right)^{k-1}\psi\left(\frac{r_1}{X}\right)\left(\frac{r_2}{X}\right).
    \]
    Now we shall take out a power of $c^{\frac{1}{2}-2\epsilon}$ from the $\left(\frac{X^{2+\epsilon_0}}{c}\right)^{k-1}$ term to assure that the sum over $c$ will converge to some constant (depending upon $\epsilon$). Furthermore, with the remaining $\left(\frac{1}{c}\right)^{k-\frac{3}{2}+2\epsilon}$ we shall use the lower bound of $ K^{10}\ll c$ to get the upper bound of
 
    \begin{align*}
        I^{OD}_{tail}&\ll_{l,\epsilon}\frac{1}{X}\sum_{r_1,r_2\geq 1}\sum_{\substack{k\geq 1\\
    2\mid k}}u\left(\frac{k-K}{K^\theta}\right)\psi\left(\frac{r_1}{X}\right)\left(\frac{r_2}{X}\right)r_1 X^{2k-2+k\epsilon_0+\epsilon_0}\sum_{\substack{c\gg K^{10}\\ c\equiv 0\Mod{N}}}\frac{1}{c^{1+\epsilon}}\left(\frac{1}{c}\right)^{k-\frac{3}{2}-2\epsilon}\\
    &\ll_{l,\epsilon} \frac{1}{X}\sum_{r_1,r_2\geq 1}\sum_{\substack{k\geq 1\\
    2\mid k}}u\left(\frac{k-K}{K^\theta}\right)\psi\left(\frac{r_1}{X}\right)\left(\frac{r_2}{X}\right)r_1 X^{2k-2+k\epsilon_0+\epsilon_0}\left(\frac{1}{K^{10}}\right)^{k-\frac{3}{2}-2\epsilon}.
    \end{align*}

    Now by computing the sum over $r_1$ and $r_2$ by using the fact that $\psi\in C_0^\infty(0,\infty)$, we have that 

    \begin{align*}
        I^{OD}_{tail}&\ll_{l,\epsilon} \sum_{\substack{k\geq 1\\
    2\mid k}}u\left(\frac{k-K}{K^\theta}\right)X^{2k+k\epsilon_0+\epsilon_0}\left(\frac{1}{K^{10}}\right)^{k-\frac{3}{2}-2\epsilon}\norm{\psi}^2_{L^\infty}\ll_{l,\epsilon} \sum_{\substack{k\geq 1\\
    2\mid k}}u\left(\frac{k-K}{K^\theta}\right)K^{15+21\epsilon+(\epsilon_0-8)k}\norm{\psi}^2_{L^\infty}.
    \end{align*}
    Now the key observation is that the exponent of $15+21\epsilon+(\epsilon_0-8)k$ is a decreasing function in $k$, and the smallest such value of $k$ in our sum is $k=2$ which gives an exponent of $K^{-1+21\epsilon+2\epsilon_0}$. Furthermore, as we have chosen $\epsilon_0$ such that $\epsilon_0<\frac{1-\theta}{2}$, we will have that $\theta-1+2\epsilon_0<0$. Thus, as there will be $O(K^\theta)$ non-zero terms in the sum over $k$, we will have that 
    \[
    I^{OD}_{tail}\ll_{l,\epsilon} K^{\theta-1+21\epsilon+2\epsilon_0}\norm{\psi}^2_{L^\infty}\ll_{l,\epsilon} K^{21\epsilon}\norm{\psi}^2_{L^\infty}.
    \]

    From this all, we are now able to conclude the following bound which concludes our analysis of the tail term:

    \begin{prop}\label{lem: Off Diagonal Tail Bound}
        We have that 
        \[
        I^{OD}_{tail}\ll_{l,\epsilon}  K^{\epsilon}\norm{\psi}^2_{L^\infty}\ll_{l,\epsilon,\theta} \qterm K^{\theta+\epsilon}\norm{\psi}^2_{W^{T,\infty}}.
        \]
    \end{prop}
    
    \subsection{Off-Diagonal Middle Contribution}\label{subsec: Off Diagonal Middle}
    The goal of this section will to bound $I^{OD}_{mid}$ given by \Cref{eq: Middle of the off-Diagonal Term}. In reality, we shall show that $I^{OD}_{mid}\ll_{l,\theta}\norm{\psi}^2_{L^\infty}$. We first remark that we can Taylor expand the $\sin(x\cos(2\pi t))$ terms as

\begin{align*}
    \begin{split}
        \sin(x\cos(2\pi t))&=\sin\left(x(1-2\pi^2t^2\right)\left(\sum_{0\leq n,m\leq M-1}c_{n,m}(xt^4)^{2n}t^{2m}\right)\\
        &+\cos\left(x(1-2\pi^2t^2)\right)\left(\sum_{\substack{1\leq n\leq M\\ 0\leq m\leq M-1}}d_{n,m}(xt^4)^{2n-1}t^{2m}\right)\\
        &+O_M\left(\max\left\{(xt^4)^{2M},(xt^4)^{4M}\right\}\right)+O_M\left(\max\left\{t^{2M},t^{4M}\right\}\right).
    \end{split}
\end{align*}

There is a question of how far we must Taylor expand, for us we will need to pick $M>\max\{\frac{11+\theta}{2\epsilon_1}, 20\}$, and once we have fixed this $M$ in this range the constants coming from the Taylor expansion having a dependence upon $M$, will now just have a dependence upon $\epsilon_1$, and hence just upon $\theta$. Thus, applying this expansion we will write $I^{OD}_{mid}=I^{OD'}_{mid}+I^{OD''}_{mid}$ where $I^{OD'}_{mid}$ will be the main part after using the Taylor expansion, and $I^{OD''}_{mid}$ will be the error. That is we will have that
\begin{align}
    \begin{split}\label{eq: Definition of IODmid'}
        I^{OD'}_{mid}&=-2\pi \frac{1}{X}\sum_{r_1,r_2\geq 1}\psi\left(\frac{r_1}{X}\right)\psi\left(\frac{r_2}{X}\right)\sum_{\substack{X^{2+\epsilon_0}K^{-1-\theta+\epsilon_1}\ll c\ll K^{10}\\
    c\equiv 0\Mod{N}}}\frac{S(\vert q(r_1)\vert,\vert q(r_2)\vert;c)}{c}\\
    &\Bigg(\int_{-\infty}^\infty \hat{g}(t)\sin\left(x(1-2\pi^2t^2\right)\sum_{0\leq n,m\leq M-1}c_{n,m}(xt^4)^{2n}t^{2m}dt
    \\
    &+\int_{-\infty}^\infty \hat{g}(t)\cos\left(x(1-2\pi^2t^2)\right)\sum_{\substack{1\leq n\leq M\\ 0\leq m\leq M-1}}d_{n,m}(xt^4)^{2n-1}t^{2m}dt\Bigg),
    \end{split}
\end{align}
and 
\begin{align}
    \begin{split}\label{eq: def of IODmid''}
        I^{OD''}_{mid}&=-2\pi \frac{1}{X}\sum_{r_1,r_2\geq 1}\psi\left(\frac{r_1}{X}\right)\psi\left(\frac{r_2}{X}\right)\sum_{\substack{X^{2+\epsilon_0}K^{-1-\theta+\epsilon_1}\ll c\ll K^{10}\\
    c\equiv 0\Mod{N}}}\frac{S(\vert q(r_1)\vert,\vert q(r_2)\vert;c)}{c}\\
    &\Bigg(\int_{-\infty}^\infty \hat{g}(t)O_{\theta}\left(\max\left\{(xt^4)^{2M},(xt^4)^{4M}\right\}\right)dt\\
    &+\int_{-\infty}^\infty \hat{g}(t)O_{\theta}\left(\max\left\{t^{2M},t^{4M}\right\}\right)dt\Bigg).
    \end{split}
\end{align}

Now we shall bound each of these terms separately. We begin with the error term $I^{OD''}_{tail}$. Let's begin with the two integrals at the very end. We do this by considering the following auxiliary bound. We define 
\begin{equation*}\label{eq: definition of crg function}
    c_r(g):=\int_{-\infty}^\infty \vert \hat{g}(t)t^r\vert dt.
\end{equation*}
Now we bound this function in the following lemma. 

\begin{prop}\label{lem: crg bound}
    $c_r(g)\ll (K^{\theta})^{-r+1}$
\end{prop}

\begin{proof}
    We note that $\hat{g}(t)$ satisfies the formula that $\widehat{g^{(r)}}(t)=(-2\pi it)^r\hat{g}(t)$. Thus, we find that 
    \[
    c_r(g)=\int_{-\infty}^\infty \left\vert \frac{1}{(2\pi i)^r}\widehat{g^{(r)}}(t)\right\vert dt = \frac{1}{(2\pi)^{r}}\norm{\widehat{g^{(r)}}(t)}_{L^1}.
    \]
    Now we introduce the function $1+t^2$. In doing so, we can now apply H\"older's inequality to say that 
    \begin{align*}
    \frac{1}{(2\pi)^{r}}\norm{\widehat{g^{(r)}}(t)}_{L^1}&\leq \frac{1}{(2\pi)^r} \norm{\widehat{g^{(r)}}(t)(1+t^2)}_{L^2}\norm{(1+t^2)\inv}_{L^2}\\
    &\ll \norm{\widehat{g^{(r)}}(t)}_{L^2}+\norm{ \widehat{g^{(r)}}(t)t^2}_{L^2}\\
    &\ll \norm{ \widehat{g^{(r)}}(t)}_{L^2}+\norm{ \widehat{g^{(r+2)}}(t)}_{L^2}\\
    &\ll \norm{ g^{(r)}(t)}_{L^2}+\norm{ g^{(r+2)}(t)}_{L^2}
    \end{align*}
    Where in the last line we have use the fact that the Fourier transform is an isometry on $L^2$. From the definition of $g(t)=u\left(\frac{t+1-K}{K^\theta}\right)$, we will have that $g^{(r)}(t)=(K^{-\theta})^ru^{(r)}\left(\frac{t+1-K}{K^\theta}\right)$. Thus, we have that 
    \[
    \norm{ g^{(r)}(t)}_{L^2}^2 =\int_{-\infty}^\infty \left\vert g^{(r)}(t)\right\vert^2dt=K^{-2r\theta}\int_{-\infty}^\infty \left\vert u^{(r)}\left(\frac{t+1-K}{K^\theta}\right)\right\vert^2dt\ll K^{(-2r+2)\theta}\norm{ u}_{H^{r}}.
    \]
    Where in the last inequality we use the fact that $u$ is compactly supported and has support of length $K^\theta$, and we bound it with the Sobolev norm of $u$, where $\norm{u}_{H^r} = \norm{u}_{W^{r,2}}$. Since $u$ is fixed from the beginning of our computations, we may drop this Sobolev norm. From this we conclude that $\norm{ g^{(r)}(t)}_{L^2}\ll K^{(-r+1)\theta}$. Thus, we conclude that $c_r(g)\ll K^{(-r+1)\theta}$.
\end{proof}

We can now use the bound of \Cref{lem: crg bound}, to get a bound on the integral appearing in \Cref{eq: def of IODmid''} we shall do these two terms separately because of the existence of the $x$ in the integrand we will have us use different methods.

\begin{prop}\label{lem: Bound of integral g O(xt4) breaking up into a sum}
    We have the following bound,
    \begin{equation*}
        \int_{-\infty}^\infty \hat{g}(t)O_{\theta}\left(\max\left\{(xt^4)^{2M},(xt^4)^{4M}\right\}\right)dt\ll_{l,\theta} (r_1r_2X^{\epsilon_0})^{2M}\frac{K^{(-8M+1)\theta}}{c^{2M}}+(r_1r_2X^{\epsilon_0})^{4M}\frac{K^{(-16M+1)\theta}}{c^{2M}}.
    \end{equation*}
\end{prop}

\begin{proof}
    Clearly, we can bound the integral by
    \[
     x^{2M}\int_{-\infty}^\infty \hat{g}(t)t^{8M}dt+x^{4M}\int_{-\infty}^\infty \hat{g}(t)t^{16M}dt.
    \]
    Now we shall apply \Cref{lem: crg bound}, to give us that the above is bounded above by 
    \[
    \ll_{l,\theta} x^{2M}K^{(-8M+1)\theta}+x^{4M}K^{(-16M+1)\theta}.
    \]
    We from \Cref{rem: Bound on Coefficients and x}, we have that $x\ll X^{\epsilon_0}\frac{r_1r_2}{c}$, so we can conclude that the above is 
    \[
    \ll_{l,\theta} (X^{\epsilon_0}r_1r_2)^{2M}\frac{K^{(-8M+1)\theta}}{c^{2M}}+(X^{\epsilon_0}r_1r_2)^{4M}\frac{K^{(-16M+1)\theta}}{c^{4M}},
    \]
    which is exactly the bound that we stated.
\end{proof}
We can now work out the corresponding sum in \Cref{eq: def of IODmid''} that is 

\begin{prop}\label{lem: ImidOD'' first part}
    We have the following bound for $M>\frac{11+\theta}{2\epsilon_1}$
    \begin{align*}
        & \frac{1}{X}\sum_{r_1,r_2\geq 1}\psi\left(\frac{r_1}{X}\right)\psi\left(\frac{r_2}{X}\right)\sum_{\substack{ X^{2+\epsilon_0}K^{-1-\theta+\epsilon_1}\ll c\ll K^{10}\\
    c\equiv 0\Mod{N}}}\frac{S(\vert q(r_1)\vert,\vert q(r_2)\vert;c)}{c}\int_{-\infty}^\infty \hat{g}(t)O_{\theta}\left(\max\left\{(xt^4)^{2M},(xt^4)^{4M}\right\}\right)dt\\
    &\ll_{l,\theta} \norm{\psi}^2_{L^\infty}.
    \end{align*}
\end{prop}

\begin{proof}
    We shall apply \Cref{lem: Bound of integral g O(xt4) breaking up into a sum} to bound the integral and break up the sum into two parts corresponding to $(r_1r_2X^{\epsilon_0})^{2M}\frac{K^{(-8M+1)\theta}}{c^{2M}}$ and $(r_1r_2X^{\epsilon_0})^{4M}\frac{K^{(-16M+1)\theta}}{c^{4M}}$.

    After applying the $(r_1r_2X^{\epsilon_0})^{2M}\frac{K^{(-8M+1)\theta}}{c^{2M}}$ bound for the integral, the trivial bound for the Kloosterman sum, and the lower bound of $c\gg X^{2+\epsilon_0}K^{-1-\theta+\epsilon_1}$, we get that the corresponding term is 
    \begin{align*}
        &\ll_{l,\theta} \frac{1}{X}\sum_{r_1,r_2\geq 1}\psi\left(\frac{r_1}{X}\right)\psi\left(\frac{r_2}{X}\right)\sum_{\substack{X^{2+\epsilon_0}K^{-1-\theta+\epsilon_1}\ll c\ll K^{10}\\ c\equiv 0\Mod{N}}}\frac{(r_1r_2X^{\epsilon_0})^{2M}K^{(-8M+1)\theta}}{c^{2M}}\\
        &\ll_{l,\theta} \frac{1}{X}\sum_{r_1,r_2\geq 1}\psi\left(\frac{r_1}{X}\right)\psi\left(\frac{r_2}{X}\right)\frac{(r_1r_2X^{\epsilon_0})^{2M}K^{(-8M+1)\theta}}{(X^{2+\epsilon_0}K^{-1-\theta+\epsilon_1})^{2M}}\sum_{X^{2+\epsilon_0}K^{-1-\theta+\epsilon_1}\ll c\ll K^{10}}1\\
        &\ll_{l,\theta} K^{10-8M\theta+\theta+2M+2M\theta-2M\epsilon_1}X\norm{\psi}^2_{L^\infty}\\
        &\ll_{l,\theta} K^{11+\theta-6M\theta+2M-2M\epsilon_1}\norm{\psi}^2_{L^\infty}.
    \end{align*}
    Now as $\frac{1}{3}<\theta<1$, we will have that $2M-6M\theta<0$, and as $M>\frac{11+\theta}{2\epsilon_1}$, we have that  $11+\theta-2M\epsilon_1<0$, we get that the above will be $\ll_{l,\theta} \norm{\psi}^2_{L^\infty}$.

    Now after applying the $(X^{\epsilon_0}r_1r_2)^{4M}\frac{K^{(-16M+1)\theta}}{c^{4M}}$ bound for the integral, the trivial bound for the Kloosterman sum, and the lower bound of $c\gg X^{2+\epsilon_0}K^{-1-\theta+\epsilon_1}$, we get that the corresponding term is 
    \begin{align*}
        &\ll_{l,\theta} \frac{1}{X}\sum_{r_1,r_2\geq 1}\psi\left(\frac{r_1}{X}\right)\psi\left(\frac{r_2}{X}\right)\sum_{\substack{X^{2+\epsilon_0}K^{-1-\theta+\epsilon_1}\ll c\ll K^{10}\\ c\equiv 0\Mod{N}}}\frac{(X^{\epsilon_0}r_1r_2)^{4M}K^{(-16M+1)\theta}}{c^{4M}}\\
        &\ll_{l,\theta} \frac{1}{X}\sum_{r_1,r_2\geq 1}\psi\left(\frac{r_1}{X}\right)\psi\left(\frac{r_2}{X}\right)\frac{(X^{\epsilon_0}r_1r_2)^{4M}K^{(-16M+1)\theta}}{(X^{2+\epsilon_0}K^{-1-\theta+\epsilon_1})^{4M}}\sum_{X^{2+\epsilon_0}K^{-1-\theta+\epsilon_1}\ll c\ll K^{10}}1\\
        &\ll_{l,\theta} K^{10-16M\theta+\theta+4M+4M\theta-4M\epsilon_1}X\norm{\psi}^2_{L^\infty}\\
        &\ll_{l,\theta} K^{11+\theta-12M\theta+4M-4M\epsilon_1}\norm{\psi}^2_{L^\infty}.
    \end{align*}
    Now as $\frac{1}{3}<\theta<1$, we have that $4M-12M\theta<0$, and as $M>\frac{11+\theta}{2\epsilon_1}$, we have that $11+\theta-4M\epsilon_1<0$, we get that the above will be $\ll_{l,\theta} \norm{\psi}^2_{L^\infty}$.
\end{proof}

\begin{prop}\label{lem: crg O max t^2M bound}
    We have the following bound 
    \begin{equation*}
        \int_{-\infty}^\infty \hat{g}(t)O\left(\max\left\{t^{2M},t^{4M}\right\}\right)dt\ll K^{(-2M+1)\theta }.
    \end{equation*}
\end{prop}

\begin{proof}
    We similarly bound the integral by the sum of the two possible max values, and apply \Cref{lem: crg bound} to get that 
    \[
    \int_{-\infty}^\infty \hat{g}(t)t^{2M}dt+\int_{-\infty}^\infty \hat{g}(t)t^{4M}dt\ll K^{(-2M+1)\theta}+K^{(-4M+1)\theta}\ll K^{(-2M+1)\theta}.\qedhere
    \]
\end{proof}

\begin{prop}\label{lem: ImidOD'' second part}
    For $M>20$, we have the following bound 
    \begin{align*}
        &-2\pi \frac{1}{X}\sum_{r_1,r_2\geq 1}\psi\left(\frac{r_1}{X}\right)\psi\left(\frac{r_2}{X}\right)\sum_{\substack{ X^{2+\epsilon_0}K^{-1-\theta+\epsilon}\ll c\ll K^{10}\\
    c\equiv 0\Mod{N}}}\frac{S(\vert q(r_1)\vert,\vert q(r_2)\vert;c)}{c}\int_{-\infty}^\infty \hat{g}(t)O_{\theta}\left(\max\left\{t^{2M},t^{4M}\right\}\right)dt\\
    &\ll_{l,\theta} \norm{\psi}^2_{L^{\infty}}.
    \end{align*}
\end{prop}

\begin{proof}
    For this proof, we can just apply the bound from \Cref{lem: crg O max t^2M bound}, along with the trivial bound for Kloosterman sums to get that the above is 
    \begin{align*}
        &\ll_{l,\theta}  \frac{1}{X}\sum_{r_1,r_2\geq 1}\psi\left(\frac{r_1}{X}\right)\psi\left(\frac{r_2}{X}\right)\sum_{\substack{ X^{2+\epsilon_0}K^{-1-\theta+\epsilon_1}\ll c\ll K^{10}\\
    c\equiv 0\Mod{N}}}K^{(-2M+1)\theta}\\
    &\ll_{l,,\theta} XK^{10+\theta-2M\theta}\norm{\psi}^2_{L^{\infty}}\\
    &\ll_{l,\theta} K^{11+\theta-2M\theta}\norm{\psi}^2_{L^\infty}.
    \end{align*}
    Now as $M>20$, we will have that $11+\theta-2M\theta<0$, and we conclude our desired bound. 
\end{proof}

Now combining the bounds of \Cref{lem: ImidOD'' first part} and \Cref{lem: ImidOD'' second part}, we get the desired bound for $I^{OD''}_{mid}$. We state this as the following lemma.
\begin{prop}\label{lem: Bound for ImidOD''}
    We have that $I_{mid}^{OD''}\ll_{l,\theta} \norm{\psi}^2_{L^\infty}$.
\end{prop}

Now we hope to conclude the desired bound for $I^{OD}_{mid}$, we just need to get a bound for $I_{mid}^{OD'}$. To do this we wish to apply the following lemma \cite[Equations (5.66) and (5.67)]{iwaniec1997topics}. 

\begin{lemma}\label{lem: Fresnel Integrals}
\begin{equation*}\label{eq: Fresnel Integral 1}
2\int_{-\infty}^\infty \hat{g}(t)\sin(x-2\pi^2t^2x)dt = \int_0^\infty g(\sqrt{2yx})\sin\left(y+x-\frac{\pi}{4}\right)(\pi y)^{-1/2}dy
\end{equation*}

\begin{equation*}
    2\int_{-\infty}^\infty \hat{g}(t)\cos(x-2\pi^2t^2x)dt = \int_0^\infty g(\sqrt{2yx})\cos\left(y+x-\frac{\pi}{4}\right)(\pi y)^{-1/2}dy
\end{equation*}
\end{lemma}
However, as written we can only apply these equations to those terms where the exponent of $t$ is $0$, which is only the $c_{0,0}$ term. To get around this, we use the equality $\widehat{g^{(p)}}(t)=(-2\pi i t)^p\hat{g}(t)$. Applying this to \Cref{eq: Definition of IODmid'} we have that 
\begin{align*}
    \begin{split}
    I_{mid}^{OD'}&=\frac{-2\pi}{X} \sum_{r_1,r_2\geq 1}\psi\left(\frac{r_1}{X}\right)\psi\left(\frac{r_2}{X}\right)\sum_{\substack{ X^{2+\epsilon_0}K^{-1-\theta+\epsilon_1}\ll c\ll K^{10}\\ c\equiv0\Mod{N}}}\frac{S\left(\vert q(r_1)\vert,\vert q(r_2)\vert ;c\right)}{c}\\
    &\Bigg(\sum_{0\leq n,m\leq M-1}c_{n,m}x^{2n}\left(\frac{1}{-2\pi i}\right)^{8n+2m}\int_{-\infty}^\infty \widehat{g^{(8n+2m)}}(t)\sin\left(x-2\pi^2t^2x\right)dt\\
    &+\sum_{\substack{1\leq n\leq M\\
    0\leq m\leq M-1}}d_{n,m}x^{2n-1}\left(\frac{1}{-2\pi i}\right)^{8n+2m-4}\int_{-\infty}^\infty \widehat{g^{(8n+2m-4)}}(t)\cos\left(x-2\pi^2 t^2x\right)dt\Bigg).
    \end{split}
\end{align*}

Now we may use \Cref{lem: Fresnel Integrals} to rewrite $I^{OD'}_{mid}$ as 
\begin{align*}
    \begin{split}
        I^{OD'}_{mid}&=\frac{-\pi}{X}\sum_{r_1,r_2\geq 1}\psi\left(\frac{r_1}{X}\right)\psi\left(\frac{r_2}{X}\right)\sum_{\substack{X^{2+\epsilon_0}K^{-1-\theta+\epsilon_1}\ll c\ll K^{10}\\ c\equiv0\Mod{N}}}\frac{S\left(\vert q(r_1)\vert,\vert q(r_2)\vert ;c\right)}{c}\\
    &\Bigg(\sum_{0\leq n,m\leq M-1}c_{n,m}x^{2n}\left(\frac{1}{-2\pi i}\right)^{8n+2m}\int_0^\infty g^{(8n+2m)}(\sqrt{2yx})\sin\left(y+x-\frac{\pi}{4}\right)(\pi y)^{-1/2}dy\\
    &+\sum_{\substack{1\leq n\leq M\\
    0\leq m\leq M-1}}d_{n,m}x^{2n-1}\left(\frac{1}{-2\pi i}\right)^{8n+2m-4}\int_0^\infty g^{(8n+2m-4)}(\sqrt{2yx})\cos\left(y+x-\frac{\pi}{4}\right)(\pi y)^{-1/2}dy\Bigg). 
    \end{split}
\end{align*}

Now we observe that the range of values of $n$ in the two sums will end up giving us all possible powers of $x$ from $0\leq n\leq 2M-1$, and we will have $0\leq m\leq M-1$. Thus, replacing our trigonometric functions with exponentials, we get that the above is 
\begin{align*}
    &\ll \frac{1}{X}\sum_{r_1,r_2\geq 1}\psi\left(\frac{r_1}{X}\right)\psi\left(\frac{r_2}{X}\right)\sum_{\substack{X^{2+\epsilon_0}K^{-1-\theta+\epsilon_1}\ll c\ll K^{10}\\ c\equiv0\Mod{N}}}\frac{S\left(\vert q(r_1)\vert,\vert q(r_2)\vert ;c\right)}{c}\\
    &\Bigg(\sum_{\substack{0\leq n\leq 2M-1\\ 0\leq m\leq M-1}} x^n\int_0^\infty g^{(4n+2m)}(\sqrt{2yx})\exp\left(i\left(x-y-\frac{\pi}{4}\right)\right) \frac{dy}{\sqrt{y}} \Bigg).\\
\end{align*}
We remark that as $x\ll_l\frac{X^{2+\epsilon_0}}{c}$ and $c\gg X^{2+\epsilon_0}K^{-1-\theta+\epsilon_1}$, we will have that $x\ll_l K^{1+\theta-\epsilon_1}$. Thus, we will end up having that:
\begin{align}
    \begin{split}\label{eq: c large Bound with I Integral}
    I^{OD'}_{mid}\ll_l& \max_{p\leq 10M} \frac{1}{X}\sum_{r_1,r_2\geq 1}\psi\left(\frac{r_1}{X}\right)\psi\left(\frac{r_2}{X}\right)\\
    &\times\sum_{\substack{X^{2+\epsilon_0}K^{-1-\theta+\epsilon_1}\ll c\ll K^{10}\\ c\equiv0\Mod{N}}}\frac{S\left(\vert q(r_1)\vert,\vert q(r_2)\vert ;c\right)}{c} (K^{1+\theta-\epsilon_1})^{2M} I^p_{r_1,r_2,c},
    \end{split}
\end{align}
where we have defined 
\begin{align*}
    \begin{split}
       I^p_{r_1,r_2,c}:=\int_0^\infty g^{(p)}\left(\sqrt{c\inv y\Delta}\right)\exp\left(i\left(\frac{\Delta c\inv}{2}+y-\frac{\pi}{4}\right)\right)\frac{dy}{\sqrt{y}},
    \end{split}
\end{align*}
and where $\Delta = 8\pi\sqrt{\vert q(r_1)q(r_2)}=2cx$. Now in order to get the decay we want, it will be useful to rewrite $g\left(\sqrt{c\inv y \Delta}\right)=u\left(\frac{\sqrt{c\inv y \Delta}+1-K}{K^\theta}\right)$. Now from the fact that $\supp(u)\subset[-1.1,1.1]$, this will give us that $\sqrt{c\inv y \Delta}\sim K$, or equivalently we will have that $y\sim \frac{K^2}{x}$. Furthermore, we can rewrite the integral $I^p_{r_1,r_2,c}$ as 
\[
I^p_{r_1,r_2,c} = \int_0^\infty \frac{d^p}{dy^p}\left(u\left(\frac{\sqrt{c\inv y \Delta}+1-K}{K^\theta}\right)\right)\exp\left(i\left(\frac{\Delta c\inv}{2}+y-\frac{\pi}{4}\right)\right)\frac{dy}{\sqrt{y}}. 
\]

The key fact is that we will give bounds on the derivatives of $u$, and the $K^\theta$ in the denominator will give us some decay so that when we apply partial integration many times we can get as much decay as we want in $I^p_{r_1,r_2,c}$. We shall do this in two steps first giving bounds for the derivatives of $u$, and then through partial integration leading us to our next two propositions.

\begin{prop}
    Let $u$ be as described, so smooth with support contained in $[-1.1,1.1]$, now if $c\inv\Delta\sim x\ll_l\frac{X^{2+\epsilon_0}}{c}$ and $c\gg X^{2+\epsilon_0}K^{-1-\theta+\epsilon_1}$, then we will have that for each $p\in\zz_{\geq 0}$
    \begin{equation*}
        \frac{d^{p+1}}{dy^{p+1}}\left(u\left(\frac{\sqrt{c\inv y \Delta}+1-K}{K^\theta}\right)\right)\ll_l K^{-\epsilon_1}\frac{d^{p}}{dy^{p}}\left(u'\left(\frac{\sqrt{c\inv y \Delta}+1-K}{K^\theta}\right)\right)
    \end{equation*}
\end{prop}

\begin{proof}
    We observe that by the support of $u$, we have that $\sqrt{c\inv y\Delta}\sim\sqrt{xy}\sim K$, or equivalently $y\sim\frac{K^2}{x}$. Now we see that by the chain rule that 
    \begin{align*}
    \frac{d^{p+1}}{dy^{p+1}}\left(u\left(\frac{\sqrt{c\inv y \Delta}+1-K}{K^\theta}\right)\right)&= \frac{d^{p}}{dy^{p}}\left(\frac{d}{dy}\left(u\left(\frac{\sqrt{c\inv y \Delta}+1-K}{K^\theta}\right)\right)\right)\\
    &=\frac{d^{p}}{dy^{p}}\left(\frac{\sqrt{c\inv \Delta}}{2K^\theta\sqrt{y}}u'\left(\frac{\sqrt{c\inv y \Delta}+1-K}{K^\theta}\right)\right).
    \end{align*}
    Now we observe that $c\inv \Delta\sim x$, and so we will have that 
    \[
    \frac{\sqrt{c\inv\Delta}}{K^\theta\sqrt{y}}\sim\frac{1}{K^\theta}\sqrt{\frac{x}{y}}\sim\frac{1}{K^\theta}\sqrt{\frac{x^2}{K^2}}\sim\frac{x}{K^{1+\theta} }\ll_l\frac{X^{2+\epsilon_0}}{cK^{1+\theta}}\ll_l \frac{X^{2+\epsilon_0}}{X^{2+\epsilon_0}K^{-1-\theta+\epsilon_1}K^{1+\theta}}\ll_l K^{-\epsilon_1}.
    \]
    Thus, we conclude the statement of the proposition. 
\end{proof}

\begin{prop}\label{prop: Bound on Ipr1r2c integral}
    For every $r\in \zz_{\geq 0}$, we have that $I_{r_1,r_2,c}^{p}\ll_l K^{2-r\epsilon_1}$.
\end{prop}
\begin{proof}
    This result comes from applying integration by parts on $I_{r_1,r_2,c}^p$ where we shall always differentiate $\frac{1}{\sqrt{y}}\frac{d^p}{dy^p}\left(u\left(\frac{\sqrt{c\inv y \Delta}+1-K}{K^\theta}\right)\right)$ and integrate $\exp\left(i\left(\frac{\Delta c\inv}{2}+y-\frac{\pi}{4}\right)\right)$. We first remark that 
    \[
    \int \exp\left(i\left(\frac{\Delta c\inv}{2}+y-\frac{\pi}{4}\right)\right)dy=i\inv \exp\left(i\left(\frac{\Delta c\inv}{2}+y-\frac{\pi}{4}\right)\right).
    \]
    Thus, if we were to apply integration by parts multiple times by always integrating the exponential factor, this factor will remain the same. Thus, we shall apply integration by parts many times by continuously differentiating $\frac{1}{\sqrt{y}}\frac{d^p}{dy^p}\left(u\left(\frac{\sqrt{c\inv y \Delta}+1-K}{K^\theta}\right)\right)$. We see by taking the derivative by the product rule that 
    \begin{align*}
    \frac{d}{dy}\left(\frac{1}{\sqrt{y}}\frac{d^p}{dy^p}\left(u\left(\frac{\sqrt{c\inv y \Delta}+1-K}{K^\theta}\right)\right)\right)=&-\frac{1}{2y^{3/2}}\frac{d^p}{dy^p}\left(u\left(\frac{\sqrt{c\inv y \Delta}+1-K}{K^\theta}\right)\right)\\
    &+\frac{1}{\sqrt{y}}\frac{d^{p+1}}{dy^{p+1}}\left(u\left(\frac{\sqrt{c\inv y \Delta}+1-K}{K^\theta}\right)\right).
    \end{align*}

    Now we see that 
    \[
    \frac{1}{y}\sim \frac{x}{K^2}\ll_l\frac{X^{2+\epsilon_0}}{cK^2}\ll_l\frac{X^{2+\epsilon_0}}{X^{2+\epsilon_0}K^{-1-\theta+\epsilon_1}K^2}\ll_l K^{-\epsilon_1}.
    \]

    Thus, by applying the previous proposition, we see that 
    \begin{align*}
    \frac{d}{dy}\left(\frac{1}{\sqrt{y}}\frac{d^p}{dy^p}\left(u\left(\frac{\sqrt{c\inv y \Delta}+1-K}{K^\theta}\right)\right)\right)\ll_l K^{-\epsilon_1}\frac{1}{\sqrt{y}}\frac{d^p}{dy^p}\left(u_1\left(\frac{\sqrt{c\inv y \Delta}+1-K}{K^\theta}\right)\right),
    \end{align*}

    where 
    \[
    u_1\left(\frac{\sqrt{c\inv y \Delta}+1-K}{K^\theta}\right)=\begin{cases}u\left(\frac{\sqrt{c\inv y \Delta}+1-K}{K^\theta}\right) & \text{if }\frac{d^p}{dy^p}u'\left(\frac{\sqrt{c\inv y \Delta}+1-K}{K^\theta}\right)\leq \frac{d^p}{dy^p}u\left(\frac{\sqrt{c\inv y \Delta}+1-K}{K^\theta}\right)\\
    u'\left(\frac{\sqrt{c\inv y \Delta}+1-K}{K^\theta}\right) & \text{if }\frac{d^p}{dy^p}u\left(\frac{\sqrt{c\inv y \Delta}+1-K}{K^\theta}\right)\leq \frac{d^p}{dy^p}u'\left(\frac{\sqrt{c\inv y \Delta}+1-K}{K^\theta}\right)
        \end{cases}.
    \]

    Thus, we see that when we apply integration by parts to $I^p_{r_1,r_2,c}$, we will have that 
    \[
    I^p_{r_1,r_2,c}\ll_l K^{-\epsilon}\int_0^\infty \frac{d^p}{dy^p}\left(u_1\left(\frac{\sqrt{c\inv y \Delta}+1-K}{K^\theta}\right)\right)\exp\left(i\left(\frac{\Delta c\inv}{2}+y-\frac{\pi}{4}\right)\right)\frac{dy}{\sqrt{y}}. 
    \]

    Now the key is that $u_1\in C_0^\infty[-1.1,1.1]$ just like $u$, and hence we can now induct where we will define 
    \[
    u_{r+1}\left(\frac{\sqrt{c\inv y \Delta}+1-K}{K^\theta}\right)=\begin{cases}u_r\left(\frac{\sqrt{c\inv y \Delta}+1-K}{K^\theta}\right) & \text{if }\frac{d^p}{dy^p}u_r'\left(\frac{\sqrt{c\inv y \Delta}+1-K}{K^\theta}\right)\leq \frac{d^p}{dy^p}u_r\left(\frac{\sqrt{c\inv y \Delta}+1-K}{K^\theta}\right)\\
    u_r'\left(\frac{\sqrt{c\inv y \Delta}+1-K}{K^\theta}\right) & \text{if }\frac{d^p}{dy^p}u_r\left(\frac{\sqrt{c\inv y \Delta}+1-K}{K^\theta}\right)\leq \frac{d^p}{dy^p}u_r'\left(\frac{\sqrt{c\inv y \Delta}+1-K}{K^\theta}\right)
        \end{cases}.
    \]
    Then after applying integration by parts $r$ times, we will have that 
    \[
    I^p_{r_1,r_2,c}\ll_l K^{-r\epsilon_1}\int_0^\infty \frac{d^p}{dy^p}\left(u_{r}\left(\frac{\sqrt{c\inv y \Delta}+1-K}{K^\theta}\right)\right)\exp\left(i\left(\frac{\Delta c\inv}{2}+y-\frac{\pi}{4}\right)\right)\frac{dy}{\sqrt{y}}. 
    \]
    Now as we will have that $u_r\in C_0^\infty[-1.1,1.1]$, we will have that $\frac{\sqrt{c\inv \Delta y}-K}{K^\theta}\ll 1$, and so 
    \[
    y\ll \frac{(K+1.1K^\theta)^2}{x}\ll K^2.
    \]
    Thus, at worst the length of the integral will be $K^2$, and we can crudely bound the integrand by a value $\ll 1$, so we can conclude that 
    $I^p_{r_1,r_2,c}\ll_l K^{2-r\epsilon_1}$. In particular, as $2$ is just a constant, we can get as much decay as we want from our integral. 
\end{proof}

Now combining our bounds \Cref{prop: Bound on Ipr1r2c integral} and \Cref{eq: c large Bound with I Integral} we bound $I_{mid}^{OD'}$ as follows. We state it as the following proposition. 
\begin{prop}\label{lem: Bound for ImidOD'}
    $I_{mid}^{OD'}\ll_{l,\theta} \norm{\psi}^2_{L^\infty}$.
\end{prop}
\begin{proof}
    Now we will apply \Cref{prop: Bound on Ipr1r2c integral} to \Cref{eq: c large Bound with I Integral}, and trivial bound for Kloosterman sums, we get that 
    \begin{align*}
        I_{mid}^{OD'}&\ll_l \frac{1}{X}\sum_{r_1,r_2\geq 1}\psi\left(\frac{r_1}{X}\right)\psi\left(\frac{r_2}{X}\right)\sum_{\substack{X^{2+\epsilon_0}K^{-1-\theta+\epsilon_1}\ll c\ll K^{10}\\ c\equiv0\Mod{N}}}\frac{S\left(\vert q(r_1)\vert,\vert q(r_2)\vert ;c\right)}{c} (K^{1+\theta-\epsilon_1})^{2M} I^p_{r_1,r_2,c}\\
        &\ll_l K^{13+2M+2M\theta-2M\epsilon_1-r\epsilon_1}\norm{\psi}^2_{L^\infty}.
    \end{align*}
    Now by picking $r$ large enough so that $13+2M+2M\theta-2M\epsilon_1-r\epsilon_1<0$, we conclude the proposition. 
\end{proof}

Now combining \Cref{lem: Bound for ImidOD'} and \Cref{lem: Bound for ImidOD''}, we conclude that $I_{mid}^{OD}\ll_{l,\theta} \norm{\psi}^2_{L^\infty}$, which was our goal for this section.

\subsection{Off-Diagonal Main Contribution}\label{subsec: Off Diagonal Main}

Now for the main part of the off-diagonal contribution $I_{main}^{OD}$. In the proof of \Cref{prop: Bound on Ipr1r2c integral}, we needed to use the lower bound of $c\gg X^{2+\epsilon_0}K^{-1-\theta+\epsilon_1}$, but we don't have this bound in this range. Thus, we will need to take a higher-order Taylor expansion and apply the principle of stationary phase. Similar to the tail we will begin by writing $I^{OD}_{main}=I^{OD'}_{main}+I^{OD''}_{main}$ where the $I^{OD'}_{main}$ term comes from the main part of the Taylor expansion and the $I^{OD''}_{main}$ term comes from the error in the Taylor expansion. Then we shall examine these two parts separately.  We begin by taking the following Taylor expansion. 

\begin{align*}
    \begin{split}\label{eq: Second Taylor Approximation of sin thing}
        \sin(x\cos(2\pi t))&=
        \sin\left(x(1-2\pi^2t^2+\frac{2}{3}\pi^4t^4\right)\left(\sum_{0\leq n,m\leq M-1}a_{n,m}(xt^6)^{2n}t^{2m}\right)\\
        &+\cos\left(x(1-2\pi^2t^2+\frac{2}{3}\pi^4t^4)\right)\left(\sum_{\substack{1\leq n\leq M\\ 0\leq m\leq M-1}}b_{n,m}(xt^6)^{2n-1}t^{2m}\right)\\
        &+O_M\left(\max\left\{(xt^6)^{2M},(xt^6)^{4M}\right\}\right)+O_M\left(\max\left\{t^{2M},t^{4M}\right\}\right).
    \end{split}
\end{align*}

Similar to the previous section, we will pick $M$ large enough. In this case we just need $M>6$; thus, we can drop the $M$ dependence by just fixing some $M>6$. Now applying this to our definition of $I^{OD}_{main}$, we find that $I^{OD}_{main}=I^{OD'}_{main}+I^{OD''}_{main}$ where 
\begin{align*}
    \begin{split}
        I^{OD'}_{main}&=\frac{-2\pi}{X} \sum_{r_1,r_2\geq 1}\psi\left(\frac{r_1}{X}\right)\psi\left(\frac{r_2}{X}\right)\sum_{\substack{c\ll X^{2+\epsilon_0}K^{-1-\theta+\epsilon_1}\\
    c\equiv 0\Mod{N}}}\frac{S(\vert q(r_1)\vert,\vert q(r_2)\vert;c)}{c}\\
    &\Bigg(\int_{-\infty}^\infty \hat{g}(t)\sin\left(x(1-2\pi^2t^2+\frac{2}{3}\pi^4t^4\right)\sum_{0\leq n,m\leq M-1}a_{n,m}(xt^6)^{2n}t^{2m}dt
    \\
    &+\int_{-\infty}^\infty \hat{g}(t)\cos\left(x(1-2\pi^2t^2+\frac{2}{3}\pi^4t^4)\right)\sum_{\substack{1\leq n\leq M\\ 0\leq m\leq M-1}}b_{n,m}(xt^6)^{2n-1}t^{2m}dt\Bigg),
    \end{split}
\end{align*}
and 
\begin{align*}
    \begin{split}
        I^{OD''}_{main} & =\frac{-2\pi}{X} \sum_{r_1,r_2\geq 1}\psi\left(\frac{r_1}{X}\right)\psi\left(\frac{r_2}{X}\right)\sum_{\substack{c\ll X^{2+\epsilon_0}K^{-1-\theta+\epsilon_1}\\
    c\equiv 0\Mod{N}}}\frac{S(\vert q(r_1)\vert,\vert q(r_2)\vert;c)}{c}\\
    &\Bigg(\int_{-\infty}^\infty \hat{g}(t)O\left(\max\left\{(xt^6)^{2M},(xt^6)^{4M}\right\}\right)dt\\
    &+\int_{-\infty}^\infty \hat{g}(t)O\left(\max\left\{t^{2M},t^{4M}\right\}\right)dt\Bigg).
    \end{split}
\end{align*}

Now for the terms in this $I_{main}^{OD''}$ that are coming from the integrals, we can apply the exact same analysis that we did when bounding $I^{OD''}_{mid}$; thus, we will break up the problem into the two integral; thus, we will have the following proposition.  

\begin{prop}\label{prop: bound for thing for int hatg Oxt in main term}
    \[
    \int_{-\infty}^\infty \hat{g}(t)O\left(\max\left\{(xt^6)^{2M},(xt^6)^{4M}\right\}\right)dt\ll_{l} \left(\frac{X^{2+\epsilon_0}}{c}\right)^{2M}K^{(-12M+1)\theta}+\left(\frac{X^{2+\epsilon_0}}{c}\right)^{4M}K^{(-24M+1)\theta}.
    \]
\end{prop}

\begin{proof}
    Clearly, we can bound the integral by 
    \[
    x^{2M}\int_{-\infty}^\infty \hat{g}(t)t^{12M}dt + x^{4M}\int_{-\infty}^\infty \hat{g}(t)t^{24M}dt.
    \]
    Now we shall apply \Cref{lem: crg bound} to give us that this is bounded above by 
    \[
    x^{2M}K^{(-12M+1)\theta}+ x^{4M}K^{(-24M+1)\theta}.
    \]
    Now from \Cref{rem: Bound on Coefficients and x}, we have that $x\ll_l \frac{X^{2+\epsilon_0}}{c}$, so we conclude the proposition. 
\end{proof}

\begin{prop}\label{prop: Prop Bound 1}
    We have the following bound for $M>6$
    \begin{align*}
    &\frac{-2\pi}{X} \sum_{r_1,r_2\geq 1}\psi\left(\frac{r_1}{X}\right)\psi\left(\frac{r_2}{X}\right)\sum_{\substack{c\ll X^{2+\epsilon_0}K^{-1-\theta+\epsilon_1}\\
    c\equiv 0\Mod{N}}}\frac{S(\vert q(r_1)\vert,\vert q(r_2)\vert;c)}{c}\int_{-\infty}^\infty \hat{g}(t)O\left(\max\left\{(xt^6)^{2M},(xt^6)^{4M}\right\}\right)dt\\
    &\ll_{l} \norm{\psi}^2_{L^\infty}.
    \end{align*}
\end{prop}

\begin{proof}
    We shall apply \Cref{prop: bound for thing for int hatg Oxt in main term} to bound the integral, and then we break up the sum into two parts corresponding to $\left(\frac{X^{2+\epsilon_0}}{c}\right)^{2M}K^{(-12M+1)\theta}$ and $\left(\frac{X^{2+\epsilon_0}}{c}\right)^{4M}K^{(-24M+1)\theta}$.

    After applying the $\left(\frac{X^{2+\epsilon_0}}{c}\right)^{2M}K^{(-12M+1)\theta}$ bound for the integral and  the trivial bound for the Kloosterman sum, we get that the corresponding term is 
    \begin{align*}
        &\ll_{l}\frac{1}{X}\sum_{r_1,r_2\geq 1}\psi\left(\frac{r_1}{X}\right)\psi\left(\frac{r_2}{X}\right)\sum_{\substack{c\ll X^{2+\epsilon_0}K^{-1-\theta+\epsilon_1}\\c\equiv 0\Mod{N}}}\left(\frac{X^{2+\epsilon_0}}{c}\right)^{2M}K^{(-12M+1)\theta}\\
        &\ll_{l}X^{1+2M(2+\epsilon_0)}K^{(-12M+1)\theta}\norm{\psi}^2_{L^\infty}\sum_{\substack{c\ll X^{2+\epsilon_0}K^{-1-\theta+\epsilon_1}\\c\equiv 0\Mod{N}}}c^{-2M}\\
        &\ll_{l}X^{1+2M(2+\epsilon_0)}K^{(-12M+1)\theta}(X^{2+\epsilon_0}K^{-1-\theta+\epsilon_1})^{-2M+1}\norm{\psi}^2_{L^\infty}\\
        &\ll_{l} X^{3+\epsilon_0}K^{-12M\theta+\theta+2M+2M\theta-2M\epsilon_1-1-\theta+\epsilon_1}\norm{\psi}^2_{L^\infty}\\
        &\ll_{l} K^{2+2M-10M\theta+\epsilon_0}\norm{\psi}^2_{L^\infty}.
    \end{align*}
    Now as $1/3<\theta<1$, we will have that $2M-6M\theta<0$. Furthermore as we have assumed that $M>6$, we will have that $2-4M\theta+\epsilon_0<0$. Thus, we see that the above is $\ll_l \norm{\psi}^2_{L^\infty}$.

    After applying the $\left(\frac{X^{2+\epsilon_0}}{c}\right)^{4M}K^{(-24M+1)\theta}$ bound for the integral and the trivial bound for the Kloosterman sum, we get that the corresponding term is 
    \begin{align*}
        &\ll_{l}\frac{1}{X}\sum_{r_1,r_2\geq 1}\psi\left(\frac{r_1}{X}\right)\psi\left(\frac{r_2}{X}\right)\sum_{\substack{c\ll X^{2+\epsilon_0}K^{-1-\theta+\epsilon_1}\\c\equiv 0\Mod{N}}}\left(\frac{X^{2+\epsilon_0}}{c}\right)^{4M}K^{(-24M+1)\theta}\\
        &\ll_{l}X^{1+4M(2+\epsilon_0)}K^{(-24M+1)\theta}\norm{\psi}^2_{L^\infty}\sum_{\substack{c\ll X^{2+\epsilon_0}K^{-1-\theta+\epsilon_1}\\c\equiv 0\Mod{N}}}c^{-4M}\\
        &\ll_{l}X^{1+4M(2+\epsilon_0)}K^{(-24M+1)\theta}\left(X^{2+\epsilon_0}K^{-1-\theta+\epsilon_1}\right)^{-4M+1}\norm{\psi}^2_{L^\infty}\\
        &\ll_{l}X^{3+\epsilon_0}K^{-24M\theta+\theta+4M+4M\theta-1-\theta-4M\epsilon_1+\epsilon_1}\norm{\psi}^2_{L^\infty}\\
        &\ll_l K^{2+4M-20M\theta+\epsilon_0}\norm{\psi}^2_{L^\infty}.
    \end{align*}
    Now as $1/3<\theta<1$, we will have that $4M-12M\theta<0$. Furthermore as we have assumed that $M>6$, we will have that $2-8M\theta+\epsilon_0<0$. Thus, we see that the above is $\ll_l\norm{\psi}^2_{L^\infty}$, and we conclude the proposition. 
\end{proof}

\begin{prop}\label{prop: Prop Bound 2}
    For $M>6$, we have the following bound 
    \begin{align*}
    &\frac{-2\pi}{X} \sum_{r_1,r_2\geq 1}\psi\left(\frac{r_1}{X}\right)\psi\left(\frac{r_2}{X}\right)\sum_{\substack{c\ll X^{2+\epsilon_0}K^{-1-\theta+\epsilon_1}\\
    c\equiv 0\Mod{N}}}\frac{S(\vert q(r_1)\vert,\vert q(r_2)\vert;c)}{c}\int_{-\infty}^\infty \hat{g}(t)O\left(\max\left\{t^{2M},t^{4M}\right\}\right)dt\\
    &\ll_l \norm{\psi}^2_{L^\infty}
    \end{align*}
\end{prop}

\begin{proof}
    For this bound, we can just apply the bound from \Cref{lem: crg O max t^2M bound}, along with the trival bound for Kloosterman sums to get that the above is 
    \begin{align*}
        &\ll_{l}\frac{1}{X}\sum_{r_1,r_2\geq 1}\psi\left(\frac{r_1}{X}\right)\psi\left(\frac{r_2}{X}\right)\sum_{\substack{c\ll X^{2+\epsilon_0}K^{-1-\theta+\epsilon_1}\\ c\equiv 0\Mod{N}}}K^{(-2M+1)\theta}\\
        &\ll_{l}X^{3+\epsilon_0}K^{-1-2M\theta+\epsilon_1}\norm{\psi}^2_{L^\infty}\\
        &\ll_lK^{2-2M\theta+\epsilon_0+\epsilon_1}\norm{\psi}^2_{L^\infty}.
    \end{align*}
    Now as $M>6$, we will have that $2-2M\theta+\epsilon_0+\epsilon_1<0$, and we conclude the proposition. 
\end{proof}

Now combining \Cref{prop: Prop Bound 1} and \Cref{prop: Prop Bound 2} we conclude the following bound. 

\begin{prop}\label{lem: ImainOD'' bound}
    We have the bound $I_{main}^{OD''}\ll_{l} \norm{\psi}^2_{L^\infty}$. 
\end{prop}

Thus, it now remains to bond $I_{main}^{OD'}$ which will be the most involved. We begin by applying $\widehat{g^{(p)}}(t)=(-2\pi it)^p \hat{g}(t)$ to the $I^{OD'}_{main}$ to arrive at 
\begin{align*}
    \begin{split}
        I^{OD'}_{main}&=\frac{-2\pi}{X} \sum_{r_1,r_2\geq 1}\psi\left(\frac{r_1}{X}\right)\psi\left(\frac{r_2}{X}\right)\sum_{\substack{c\ll X^{2+\epsilon_0}K^{-1-\theta+\epsilon_1}\\
    c\equiv 0\Mod{N}}}\frac{S(\vert q(r_1)\vert,\vert q(r_2)\vert;c)}{c}\\
    &\Bigg(\sum_{0\leq n,m\leq M-1}\frac{a_{n,m}x^{2n}}{(-2\pi i)^{12n+2m}}\int_{-\infty}^\infty \widehat{g^{(12n+2m)}}(t)\sin\left(x\left(1-2\pi^2t^2+\frac{2}{3}\pi^4t^4\right)\right)dt
    \\
    &+\sum_{\substack{1\leq n\leq M\\ 0\leq m\leq M-1}}\frac{b_{n,m}x^{2n-1}}{(-2\pi i)^{12n+2m-6}}\int_{-\infty}^\infty \widehat{g^{(12n+2m-6)}}(t)\cos\left(x\left(1-2\pi^2t^2+\frac{2}{3}\pi^4t^4\right)\right)dt\Bigg).
    \end{split}
\end{align*}
We will now use the definition of $\hat{g}(t)$ given by \Cref{eq: Fourier Transform Definiton} as an integral. Furthermore, since $g\in C_0^\infty (0,\infty)$, we will remark that the integral is supported on $(0,\infty)$. Thus, we have that 
\begin{align*}
    \begin{split}
        I^{OD'}_{main}&=\frac{-2\pi}{X} \sum_{r_1,r_2\geq 1}\psi\left(\frac{r_1}{X}\right)\psi\left(\frac{r_2}{X}\right)\sum_{\substack{c\ll X^{2+\epsilon_0}K^{-1-\theta+\epsilon_1}\\
    c\equiv 0\Mod{d}}}\frac{S(\vert q(r_1)\vert,\vert q(r_2)\vert;c)}{c}
    \end{split}\\
    &\Bigg(\sum_{0\leq n,m\leq M-1}\frac{a_{n,m}x^{2n}}{(-2\pi i)^{12n+2m}}\int_{-\infty}^\infty\Bigg(\int_0^\infty  g^{(12n+2m)}(y)e(ty)dy\Bigg)\sin\left(x\left(1-2\pi^2t^2+\frac{2}{3}\pi^4t^4\right)\right)dt \notag
    \\
    &+\sum_{\substack{1\leq n\leq M\\ 0\leq m\leq M-1}}\frac{b_{n,m}x^{2n-1}}{(-2\pi i)^{12n+2m-6}}\int_{-\infty}^\infty \Bigg(\int_0^\infty g^{(12n+2m-6)}(y)e(ty)dy\Bigg)\cos\left(x\left(1-2\pi^2t^2+\frac{2}{3}\pi^4t^4\right)\right)dt\Bigg). \notag
\end{align*}
Now by swapping the order of integration, we will have that 
\begin{align*}
    \begin{split}
        I^{OD'}_{main}&=\frac{-2\pi}{X} \sum_{r_1,r_2\geq 1}\psi\left(\frac{r_1}{X}\right)\psi\left(\frac{r_2}{X}\right)\sum_{\substack{c\ll X^{2+\epsilon_0}K^{-1-\theta+\epsilon_1}\\
    c\equiv 0\Mod{d}}}\frac{S(\vert q(r_1)\vert,\vert q(r_2)\vert;c)}{c}
    \end{split}\\
    &\Bigg(\sum_{0\leq n,m\leq M-1}\frac{a_{n,m}x^{2n}}{(-2\pi i)^{12n+2m}}\int_0^\infty  g^{(12n+2m)}(y)\Bigg(\int_{-\infty}^\infty \sin\left(x\left(1-2\pi^2t^2+\frac{2}{3}\pi^4t^4\right)\right)e(ty)dt\Bigg)dy\notag
    \\
    &+\sum_{\substack{1\leq n\leq M\\ 0\leq m\leq M-1}}\frac{b_{n,m}x^{2n-1}}{(-2\pi i)^{12n+2m-6}}\int_0^\infty g^{(12n+2m-6)}(y) \Bigg(\int_{-\infty}^\infty \cos\left(x\left(1-2\pi^2t^2+\frac{2}{3}\pi^4t^4\right)\right)e(ty)dt\Bigg)dy\Bigg).\notag
\end{align*}

Now we have the factors of $x^{2n}$ in front of our sums, we can only bound this in the worst case scenario by $x^{2n}\ll_l \left(\frac{X^{2+\epsilon_0}}{c}\right)^{2n}\ll_l K^{(2+\epsilon_0)2n}$. To deal with this large factor of $K$, we will rewrite our integral by writing $g(y)=u\left(\frac{y+1-K}{K^\theta}\right)$, and so by the chain rule, we will have that $g^{(n)}(y)=K^{-n\theta}u^{(n)}\left(\frac{y+1-K}{K^\theta}\right)$. Thus, we will have that 

\begin{align*}
    I^{OD'}_{main}&\ll_l\frac{-2\pi}{X} \sum_{r_1,r_2\geq 1}\psi\left(\frac{r_1}{X}\right)\psi\left(\frac{r_2}{X}\right)\sum_{\substack{c\ll X^{2+\epsilon_0}K^{-1-\theta+\epsilon_1}\\
    c\equiv 0\Mod{d}}}\frac{S(\vert q(r_1)\vert,\vert q(r_2)\vert;c)}{c}\\
    &\Bigg(\sum_{0\leq n,m\leq M-1}\frac{K^{(2+\epsilon_0)2n}}{K^{12n\theta}K^{2m\theta}}\int_0^\infty  u^{(12n+2m)}\left(\frac{y+1-K}{K^\theta}\right)\\
    &\times\Bigg(\int_{-\infty}^\infty \sin\left(x\left(1-2\pi^2t^2+\frac{2}{3}\pi^4t^4\right)\right)e(ty)dt\Bigg)dy
    \\
    &+\sum_{\substack{1\leq n\leq M\\ 0\leq m\leq M-1}}\frac{K^{(2+\epsilon_0)(2n-1)}}{K^{(12n-6)\theta}K^{2m\theta}}\int_0^\infty u^{(12n+2m-6)}\left(\frac{y+1-K}{K^\theta}\right)  \\
    &\times\Bigg(\int_{-\infty}^\infty\cos\left(x\left(1-2\pi^2t^2+\frac{2}{3}\pi^4t^4\right)\right)e(ty)dt\Bigg)dy\Bigg).
\end{align*}

Now here is where the miracle occurs, as we have $1/3<\theta<1$, and we have chosen $\epsilon_0<\frac{3\theta-1}{20}<6\theta-2$, we will have that $2+\epsilon_0-6\theta<0$, so we will have that $\frac{K^{(2+\epsilon_0)2n}}{K^{12n\theta}}\ll 1$ and $\frac{K^{(2+\epsilon_0)(2n-1)}}{K^{(12n-6)\theta}}\ll 1$. Thus, the factor of $K^\theta$ in the denominator will end up canceling out the large power of $x$ in the numerator. We might remark this is one place where we do need that $1/3<\theta$, as if we were to try this approach with $\theta<1/3$, then we are unable to cancel out all of the contribution from $x$ in the numerator. Furthermore, as $0\leq m\leq M-1$ (in particular $m=0$), we can't hope for any extra cancellation coming from the $K^{-2m\theta}$ term. 

Now in a similar manner to the $I^{OD'}_{mid}$ case, we will rewrite our integral terms in terms of exponentials. In particular, we will define
\begin{align*}
    \begin{split}\label{eq: Definition of J integral}
        J_{r_1,r_2,c}^{p}:= \int_0^\infty u^{(p)}\left(\frac{y+1-K}{K^{\theta}}\right)
        \left(\int_{-\infty}^\infty \exp\left(\frac{i\Delta c\inv}{2}\left(1-2\pi^2t^2+\frac{2}{3}\pi^4t^4\right)\right)e(ty)dt\right)dy,
    \end{split}
\end{align*}
where $\Delta=8\pi\sqrt{\vert q(r_1)q(r_2)\vert}=2cx$. Hence we see that 
\begin{equation*}\label{eq: J Bound for IOD'Main}
    I^{OD'}_{main}\ll_l \max_{p\leq 14M}\left\vert\frac{1}{X}\sum_{r_1,r_2\geq 1}\psi\left(\frac{r_1}{X}\right)\psi\left(\frac{r_2}{X}\right)\sum_{\substack{c\ll X^{2+\epsilon_0}K^{-1-\theta+\epsilon_1}\\ c\equiv 0\Mod{N}}}\frac{S(\vert q(r_1)\vert,\vert q(r_2)\vert;c)}{c}J_{r_1,r_2,c}^{p}\right\vert. 
\end{equation*}

Now in the definition of $J_{r_1,r_2,c}^{p}$, we can factor out a $\exp\left(\frac{i\Delta c\inv}{2}\right)$, by doing so we can start to estimate $J_{r_1,r_2,c}^{p}$ by first estimating the integral 
\begin{equation}\label{eq: Integral to apply stationary phase to}
    \int_{-\infty}^\infty \exp\left(ix\left(-2\pi^2t^2+\frac{2}{3}\pi^4t^4\right)\right)e(ty)dt.
\end{equation}
To estimate this integral we shall apply the principle of stationary phase. We have that the phase function and its derivatives are given by 
\begin{align}
    f(t)&=x\left(-2\pi^2t^2+\frac{2}{3}\pi^4t^4\right)+2\pi ty\\
    f'(t)&=x\left(-4\pi^2t+\frac{8}{3}\pi^4t^3\right)+2\pi y\label{eq: Derivative of f}\\
    f''(t)&= x\left(-4\pi^2+8\pi^4t^2\right).\label{eq: Second Derivative of f}
\end{align}

We remark that in the definition of $J_{r_1,r_2,c}^{p}$ since $u\in C_0^\infty(0,\infty)$, we will have that $1\ll\frac{y-K}{K^\theta}\ll 1$, so $K\ll K^\theta+K\ll y\ll K^\theta+K\ll K$ or  $y\sim K$. Furthermore, from \Cref{rem: Bound on Coefficients and x}, we have that $\frac{X^2}{c}\ll_l x\ll_l\frac{X^{2+\epsilon_0}}{c}$, and so we will have that $\frac{cK}{X^{2+\epsilon_0}}\ll_l \frac{y}{x}\ll_l \frac{cK}{X^2}$. Furthermore, as $c\ll X^{2+\epsilon_0}K^{-1-\theta+\epsilon_1}$, we will have that $\frac{y}{x}\ll_l X^{\epsilon_0}K^{-\theta+\epsilon_1}$.


Now we shall consider the polynomial $-4\pi^2 t+\frac{8}{3}\pi^4t^3$, which will have roots $0,\pm\sqrt{\frac{3}{2}}\pi\inv$. Let's call these roots $s_i$ for $i=1,2,3$. Now if we Taylor expand $f'(t)$ around $s_i$, we will have that 
\begin{equation*}\label{eq: Taylor expansion of derivative in Stationary Phase}
    f'(t)=f'(s_i)+f''(s_i)(t-s_i)+O((t-s_i)^2).
\end{equation*}
Now if we let $\beta_i$ be the points of stationary phase (i.e. $f'(\beta_i)=0)$) that is closest to $s_i$. Plugging our stationary phase point into our Taylor expansion, we find that 
\[
0=2\pi y+f''(s_i)(\beta_i-s_i)+O((\beta_i-s_i)^2).
\]
By observing that $f''(s_i)\sim x$, we will have that 
\[
\frac{y}{x}\sim (\beta_i-s_i)+O((\beta_i-s_i)^2).
\]
Now we have that $\frac{y}{x}\ll_l X^{\epsilon_0}K^{-\theta+\epsilon_1}$, which is small. Let us write our stationary phase points as $\beta_i=s_i+t_i$ where $t_i\sim \frac{y}{x}$. In particular, our points of stationary phase will be $t_1$, $\sqrt{\frac{3}{2}}\pi\inv-t_2$ and $\sqrt{\frac{3}{2}}\pi\inv -t_3$. 

Now we shall apply the principle of stationary phase which we encapsulate in the following two lemmas. One can find these as for example \cite[Lemma 4.3 and Lemma 4.4]{luo2003mass} where they cite \cite[Lemma 4.2]{titchmarsh1986theory} and \cite[Lemma 10]{heath1983pjateckiui}.

\begin{lemma}\label{lem: Stationary Phase away from Stationary Points}
    Let $g(t)$ be a real differentiable function such that $g'(t)$ is monotonic and $g'(t)\geq \mu>0$ or $g'(t)\leq -\mu<0$ throughout the intervale $[a,b]$. Then 
    \[
    \left\vert \int_a^b e^{ig(t)}dt\right\vert\leq \frac{4}{\mu}.
    \]
\end{lemma}

\begin{lemma}\label{lem: Stationary Phase for Stationary Points}
    Let $g(t)$ be holomorphic on an open rectangle $R$ containing the real line segment $[a,b]$ and let $\vert g''(z)\vert\leq M$ in $R$. Let $g(t)$ be real when $t\in R$ is real and let $g''(t)\leq -\kappa M$ with $\kappa>0$. Then if $g'(\beta)=0$ with $a<\beta< b$, we have 
    \begin{align*}
        \begin{split}\label{eq: lemma integral of egt}
            \int_a^be\left(g(t)\right)dt= & e\left(-\frac{1}{8}\right)\vert g''(\beta)\vert^{-1/2}e\left(g(\beta)\right)\\
            &+O\left(\left((\beta-a)M\right)\inv\right)+O\left(\left((b-\beta)M\right)\inv\right)
        \end{split}
    \end{align*}
    with the implied constants depending on $\kappa$ and $R$ only. 
\end{lemma}

We remark that the statement of the notation in the two lemmas above is independent of that within our paper. Our goal in applying stationary phase to \Cref{eq: Integral to apply stationary phase to}, is that we shall break our integral up into $5$ pieces, one for each of the two tails of integral, and one for each of the stationary phase points. In particular, we may break up our integral as 
\[
\int_{-\infty}^\infty=\int_{-\infty}^{-\sqrt{\frac{3}{2}}\pi\inv - \frac{1}{2\pi}\sqrt{\frac{3}{2}}}+\int_{-\sqrt{\frac{3}{2}}\pi\inv - \frac{1}{2\pi}\sqrt{\frac{3}{2}}}^{- \frac{1}{2\pi}\sqrt{\frac{3}{2}}}+\int_{- \frac{1}{2\pi}\sqrt{\frac{3}{2}}}^{ \frac{1}{2\pi}\sqrt{\frac{3}{2}}}+\int_{ \frac{1}{2\pi}\sqrt{\frac{3}{2}}}^{\sqrt{\frac{3}{2}}\pi\inv + \frac{1}{2\pi}\sqrt{\frac{3}{2}}}+\int_{\sqrt{\frac{3}{2}}\pi\inv +  \frac{1}{2\pi}\sqrt{\frac{3}{2}}}^\infty.
\]

Now let us first deal with the tail terms. Since our function is a quadratic, we will have that our function will be monotonic after our stationary phase points, so we just need a lower bound on $f'(t)$ in this region. As we are away from the stationary phase points, we see that $f'(t)\gg x+y\gg x$. Thus, using \Cref{lem: Stationary Phase away from Stationary Points}, each of the integral $\int_{-\infty}^{-\sqrt{\frac{3}{2}}\pi\inv - \frac{1}{2\pi}\sqrt{\frac{3}{2}}}$ and $\int_{\sqrt{\frac{3}{2}}\pi\inv +  \frac{1}{2\pi}\sqrt{\frac{3}{2}}}^\infty$ will contribute $O(x\inv)$. 

Now in the application of \Cref{lem: Stationary Phase for Stationary Points}, we need to work out the bound $M$ on the second derivative to get the error term. As $t_i$ is small, from \Cref{eq: Second Derivative of f}, we see that the bound can be taken as $M\sim x$. Furthermore, as each of these integrals length is finite, we will have that $(\beta-a)\sim 1$ and $(b-\beta)\sim 1$. Hence we see that each of the error terms will contribute $O(x\inv)$. 

Combining this all, we get the that:
\begin{align*}
    \begin{split}\label{eq: Application of Stationary Phase}
        &\int_{-\infty}^\infty \exp\left(ix\left(-2\pi^2t^2+\frac{2}{3}\pi^4t^4\right)\right)e(ty)dt\\
        =&e\left(-\frac{1}{8}\right)\left\vert x\frac{(-4\pi^2+8\pi^4t_1^2)}{2\pi}\right\vert^{-1/2} e\left(\frac{x}{2\pi}\left(-2\pi^2t_1^2+\frac{2}{3}\pi^4t_1^4\right)+t_1y\right)\\
       +&e\left(-\frac{1}{8}\right)\left\vert x\frac{(-4\pi^2+8\pi^4\left(\sqrt{\frac{3}{2}}\pi\inv-t_2\right)^2}{2\pi}\right\vert^{-1/2}e\left(\left(\sqrt{\frac{3}{2}}\pi\inv -t_2\right)y\right)\\
       \times &e\left(\frac{x}{2\pi}\left(-2\pi^2\left(\sqrt{\frac{3}{2}}\pi\inv -t_2\right)^2+\frac{2}{3}\pi^4\left(\sqrt{\frac{3}{2}}\pi\inv -t_2\right)^4\right)\right)\\
       +&e\left(-\frac{1}{8}\right)\left\vert x\frac{(-4\pi^2+8\pi^4\left(-\sqrt{\frac{3}{2}}\pi\inv-t_3\right)^2}{2\pi}\right\vert^{-1/2}e\left(\left(-\sqrt{\frac{3}{2}}\pi\inv -t_3\right)y\right)\\
       \times& e\left(\frac{x}{2\pi}\left(-2\pi^2\left(-\sqrt{\frac{3}{2}}\pi\inv -t_3\right)^2+\frac{2}{3}\pi^4\left(-\sqrt{\frac{3}{2}}\pi\inv -t_3\right)^4\right)\right)\\
       +&O(x\inv).
    \end{split}
\end{align*}
Thus, we shall break up $J^p_{r_1,r_2,c}$ as 
\[
J^p_{r_1,r_2,c}=J_{r_1,r_2,c}^{p,\beta_1}+J_{r_1,r_2,c}^{p,\beta_2}+J_{r_1,r_2,c}^{p,\beta_3}+J_{r_1,r_2,c}^{p,x},
\]
where 
\[
J^{p,\beta_i}_{r_1,r_2,c}=e\left(-\frac{1}{8}\right)\exp\left(\frac{i\Delta}{2c}\right)\int_0^\infty \vert f''(\beta_i)\vert^{-1/2}u^{(p)}\left(\frac{y+1-K},{K^{\theta}}\right)e\left(f(\beta)\right)dy
\]
and 
\[
J^{p,x}_{r_1,r_2,c}=\exp\left(\frac{i\Delta}{2c}\right)\int_0^\infty u^{(p)}\left(\frac{y+1-K}{K^{\theta}}\right)O(x\inv)dy.
\]
Analogously, we shall define 
\[
I_{main}^{OD',\alpha} = \max_{p\leq 14M} \frac{1}{X}\sum_{r_1,r_2\geq 1}\psi\left(\frac{r_1}{X}\right)\psi\left(\frac{r_2}{X}\right)\sum_{\substack{c\ll X^{2+\epsilon_0}K^{-1-\theta+\epsilon_1}\\ c\equiv 0\Mod{N}}}\frac{S(\vert q(r_1)\vert,\vert q(r_2)\vert;c)}{c}J_{r_1,r_2,c}^{p,\alpha},
\]
for $\alpha = \beta_1,\beta_2,\beta_3,x$. Then it is clear from this definition that 
\begin{equation*}\label{eq: I^OD_main < I^OD_main beta1,beta2,beta3,x}
    I_{main}^{OD'}\ll_l I_{main}^{OD',\beta_1}+I_{main}^{OD',\beta_2}+I_{main}^{OD',\beta_3}+I_{main}^{OD',x}.
\end{equation*}

Thus, we must show that $I^{OD',\alpha}_{main}\ll_{l,\epsilon,\theta}(1+\vert D\vert^{3/2})K^{\theta+\epsilon}\norm{\psi}^2_{W^{T,\infty}}$. We shall first do this for $\alpha=x$, and then $\alpha=\beta_2,\beta_3$ as these are the easier cases. Then lastly, we shall bound the term coming from $\alpha=\beta_1$. 

We begin with looking at the value of $J_{r_1,r_2,c}^{p,x}$, we see that as $u\in C^\infty_0[-1.1,1.1]$, we will have that the integral over $y$ will only have support for a length $\sim K^\theta$. Thus, we get the trivial upper bound that $J_{r_1,r_2,c}^{p,x}\ll K^\theta x\inv$. From this we can now deduce the bound for the term of $I^{OD',x}_{main}$. We state this as the following lemma. 

\begin{prop}\label{prop: Bound for ImainODx}
    For every $\epsilon>0$, we have that 
    \[
    I_{main}^{OD',x}\ll_{l,\epsilon    }K^{\theta+\epsilon}\norm{\psi}^2_{L^\infty}.
    \]
\end{prop}

\begin{proof}
    This follows from the observation that$J^{p,x}_{r_1,r_2,c}\ll K^\theta x\inv$, and by using the Weil bound for Kloosterman sums $ S(\vert q(r_1)\vert,\vert q(r_2)\vert;c)\ll_{\epsilon} (\vert q(r_1)\vert,\vert q(r_2)\vert,c)^{1/2} c^{1/2+\epsilon/5}$. Furthermore, we use that $\frac{r_1r_2}{c}\ll_lx \ll_lX^{\epsilon_0}\frac{r_1r_2}{c}$. Thus, we see that
    \[
    I^{OD', x}_{main}\ll_{\epsilon}\frac{K^\theta}{X}\sum_{r_1,r_2\geq 1}\psi\left(\frac{r_1}{X}\right)\psi\left(\frac{r_2}{X}\right)\frac{1}{r_1r_2}\sum_{c\ll X^{2+\epsilon_0}K^{-1-\theta+\epsilon_1}} (\vert q(r_1)\vert,\vert q(r_2)\vert,c)^{1/2}c^{1/2+\epsilon/5}.
    \]
    Now we observe that 
\begin{align*}
\sum_{c\ll X^{2+\epsilon_0}K^{-1-\theta+\epsilon_1}} (\vert q(r_1)\vert,\vert q(r_2)\vert,c)^{1/2}c^{1/2+\epsilon_0}&\leq \sum_{c\ll X^{2+\epsilon_0}K^{-1-\theta+\epsilon_1}} (\vert q(r_1)\vert,c)c^{1/2+\epsilon/5}\\
&\leq \sum_{c\ll X^{2+\epsilon_0}K^{-1-\theta+\epsilon_1}} c^{1/2+\epsilon/5}\sum_{d\mid q(r_1)}\phi(d)\delta_{d\mid c}\\
&=\sum_{d\mid q(r_1)}\phi(d)\sum_{\substack{c\ll X^{2+\epsilon_0}K^{-1-\theta+\epsilon_1}\\ d\mid c}}c^{1/2+\epsilon/5}\\
&=\sum_{d\mid q(r_1)}\phi(d)d^{1/2+\epsilon/5}\sum_{\substack{c\ll \frac{X^{2+\epsilon_0}K^{-1-\theta+\epsilon_1}}{d}}}c^{1/2+\epsilon/5}\\
&\ll \sum_{d\mid q(r_1)}\phi(d)d^{1/2+\epsilon/5}\left(\frac{X^{2+\epsilon_0}K^{-1-\theta+\epsilon_1}}{d}\right)^{\frac{3}{2}+\epsilon/5}\\
&\ll \left(X^{2+\epsilon_0}K^{-1-\theta+\epsilon_1}\right)^{\frac{3}{2}+\epsilon/5}\sum_{d\mid q(r_1)}\frac{\phi(d)}{d}.
\end{align*}
Now as 
\[
\sum_{d\mid q(r_1)}\frac{\phi(d)}{d}\leq \tau(q(r_1))\ll_{l,\epsilon}X^{\epsilon/5}\ll_{l,\epsilon}K^{\epsilon/5},
\]
and \[
\sum_{r_1,r_2\geq 1}\frac{\psi\left(\frac{r_1}{X}\right)\psi\left(\frac{r_2}{X}\right)}{r_1r_2}\ll_l (\log X)^2\norm{\psi}^2_{L^\infty}\ll_{l,\epsilon}K^{2\epsilon/5}\norm{\psi}^2_{L^\infty},
\]
we have that 
\[
I^{OD',x}_{main}\ll_{l,\epsilon}\frac{K^\theta}{X}\left(X^{2+\epsilon_0}K^{-1-\theta+\epsilon_1}\right)^{\frac{3}{2}+\epsilon/5}K^{3\epsilon/5}\norm{\psi}^2_{L^\infty}.
\]
Thus, as $X^{2+\epsilon_0}K^{-1-\theta+\epsilon_1}\ll_l K^2$, we will have that $(X^2K^{-1-\theta+\epsilon_1})^{\epsilon/5}\ll_l K^{2\epsilon/5}$. Combining this all, we see that 
\[
I^{OD',x}_{main}\ll_{l,\epsilon}K^{\theta+\epsilon}K^{\frac{1}{2}(1-3\theta+3\epsilon_0+3\epsilon_1)}
\]
    Now as we assumed that $\epsilon_0,\epsilon_1<\frac{3\theta-1}{6}$, we have that $3\epsilon_0+3\epsilon_1<3\theta-1$, so $(1-3\theta+3\epsilon_0+3\epsilon_1)<0$ and we can conclude the proposition.
\end{proof}

\begin{prop}\label{lem: Bound for after stationary phase not term we are interested in}
    For $i=2,3$, we have that 
    \[
    I_{main}^{OD',\beta_i}\ll_l \norm{\psi}^2_{L^\infty}.
    \]
\end{prop}

\begin{proof}
    By the symmetry of the $\beta_2$ and $\beta_3$, we shall just prove this for $\beta_2$. 
    We observe that as $t_2\sim \frac{y}{x}\ll_l X^{\epsilon_0}K^{-\theta+\epsilon_1}$, we will have that $t_2$ is small, so $\beta_2$ is roughly constant in this region; furthermore, we see by implicit differentiating \Cref{eq: Derivative of f}, we have that 
    \[
    \frac{d}{dy}\beta_2 = \frac{1}{2\pi x(1-2\pi^2\beta_2^2)}\sim \frac{1}{x}\ll_l X^{\epsilon_0}K^{-1-\theta+\epsilon_1}.
    \]
    Hence, we have that not only is $\beta_2\sim 1$, but it also doesn't vary too much in its value. Hence we shall treat it as being constant. Now to get our desired bound, we observe that 

    \[
    \left\vert x\frac{(-4\pi^2+8\pi^4\left(\sqrt{\frac{3}{2}}\pi\inv-t_2\right)^2}{2\pi}\right\vert^{-1/2}=\vert x\vert^{-1/2}\left\vert \frac{(-4\pi^2+8\pi^4\left(\sqrt{\frac{3}{2}}\pi\inv-t_2\right)^2}{2\pi}\right\vert^{-1/2}\sim \vert x\vert^{-1/2}.
    \]

    Similarly, we see that $f(\beta_2)\sim \beta_2 y$. Thus, we shall bound $J^{p,\beta_2}_{r_1,r_2,c}$ by 
    \[
    \vert x\vert^{-1/2}\int_0^\infty u^{(p)}\left(\frac{y+1-K}{K^\theta}\right)e(\beta_2 y)dy.
    \]
    Now we see that by integrating by parts $r$-times where we always differentiate the function $u$ and integrate the exponential part, we see that the above is 
    \[
    \vert x\vert^{-1/2}\left(\frac{K^{-\theta}}{2\pi i \beta_2 }\right)^{r}\int_0^\infty u^{(p+r)}\left(\frac{y+1-K}{K^\theta}\right)e(\beta_2 y)dy\ll \vert x\vert^{-1/2}K^{\theta(1-r)}.
    \]
    Where we used that $\beta_2\sim 1$, we see that we can get as much decay as we would like in this integral. We remark that this the reason this proof does not work in the $\beta_1$ case is that $\beta_1=t_1\sim \frac{y}{x}$ which is small and hence if we tried doing an integration by parts argument, the denominator now explodes, so we are left to deal with the $\beta_1$ case separately. 

    Now, we conclude the theorem by using the trivial bound for Kloosterman sums, and observing that
    \begin{align*}
        I^{OD',\beta_2}_{main}&\ll_l\frac{1}{X}\sum_{r_1,r_2\geq 1}\psi\left(\frac{r_1}{X}\right)\psi\left(\frac{r_2}{X}\right)\sum_{\substack{c\ll X^{2+\epsilon_0}K^{-1-\theta+\epsilon_1}\\ c\equiv 0\Mod{N}}}\frac{S(\vert q(r_1)\vert,\vert q(r_2)\vert;c)}{c}J^{p,\beta_2}_{r_1,r_2,c}\\
        &\ll_l\frac{1}{X}\sum_{r_1,r_2\geq 1}\psi\left(\frac{r_1}{X}\right)\psi\left(\frac{r_2}{X}\right)\sum_{\substack{c\ll X^{2+\epsilon_0}K^{-1-\theta+\epsilon_1}\\ c\equiv 0\Mod{N}}}\vert x\vert^{-1/2}K^{\theta(1-r)}\\
        &\ll_l X^{3+\epsilon_0}K^{-1-\theta+\epsilon_1+\theta(1-r)}\norm{\psi}^2_{L^\infty}.
    \end{align*}
    Then we see that the proposition follows by choosing $r$ such that $\theta r> 2+\epsilon_0+\epsilon_1$.
\end{proof}


Thus, in order to prove \Cref{thm: Main Theorem}, it suffices to show the final bound of
\begin{equation*}\label{eq: Last term after stationary phase}
    I_{main}^{OD',\beta_1}\ll_{l,\epsilon,\theta}\qterm K^{\theta+\epsilon}\norm{\psi}^2_{W^{T,\infty}}
\end{equation*}
We shall now just drop the subscript on $\beta_1=t_1\sim\frac{y}{x}$ calling it just $\beta$. Now we have that 
\begin{align*}
J_{r_1,r_2,c}^{p,\beta}=&\exp\left(\frac{i\Delta}{2c}\right)e\left(-\frac{1}{8}\right)
\\
&\times 
\int_0^\infty \left\vert x\frac{\left(-4\pi^2+8\pi^4\beta^2\right)}{(2\pi)}\right\vert^{-1/2}u^{(p)}\left(\frac{y+1-K}{K^{\theta}}\right)e\left(\frac{x}{2\pi}\left(-2\pi^2\beta^2+\frac{2}{3}\pi^4\beta^4\right)\right)e\left(\beta y\right)dy.
\end{align*}
In particular, as $\beta$ has a dependence upon $y$, we have let $J^{p,\beta}_{r_1,r_2,c}(y)$ be the part of the integrand which depends upon $y$. That is we will let 
\begin{equation*}
    J^{p,\beta}_{r_1,r_2,c}(y):= \exp\left(\frac{i\Delta}{2c}\right)e\left(-\frac{1}{8}\right) 
\left\vert x\frac{\left(-4\pi^2+8\pi^4\beta^2\right)}{(2\pi)}\right\vert^{-1/2}e\left(\frac{x}{2\pi}\left(-2\pi^2\beta^2+\frac{2}{3}\pi^4\beta^4\right)\right)e\left(\beta y\right).
\end{equation*}
In particular, we have defined it in such a way that 
\[
J^{p,\beta}_{r_1,r_2,c} = \int_0^\infty u^{(p)}\left(\frac{y+1-K}{K^\theta}\right)J^{p,\beta}_{r_1,r_2,c}(y)dy.
\]

We observe that there is the integral over $y$ in the definition of $J^{p,\beta}_{r_1,r_2,c}$. However, as we have that $u\in C_0^\infty(\rr)$, we are working in the range where $y\sim K$; furthermore, the length of the integral will be $K^\theta$. Thus, we shall prove the following bound that does not involve the integral. However, as the length of the integral is $K^\theta$, it will be clear that our desired bound follows from the lemma. 
\begin{prop}\label{lem: Bound for J_main}
    For $y\sim K$, $D\ll X$, $r_i\sim_lX$, and for any fixed $\epsilon>0$, if we pick $\epsilon_2$ such that $\epsilon_0<\epsilon_2<\frac{2\epsilon}{5}$. Then we have that for $T>\max\left\{\frac{100}{3\theta-1},\frac{100}{\epsilon_2-\epsilon_0}\right\}$,  we have that 
    \begin{align*}
        J_{main}&:=\frac{1}{X}\sum_{r_1,r_2\geq 1}\psi\left(\frac{r_1}{X}\right)\psi\left(\frac{r_2}{X}\right)\sum_{\substack{c\ll X^{2+\epsilon_0}K^{-1-\theta+\epsilon_1}\\ c\equiv N\Mod{N}}}\frac{S(\vert q(r_1)\vert,\vert q(r_2)\vert;c)}{c}J^{p,\beta}_{r_1,r_2,c}(y)\\
        &\ll_{l,\epsilon,\theta}\qterm K^{\epsilon}\norm{\psi}^2_{W^{T,\infty}}.
    \end{align*}
    In particular, we will have that this bound implies that 
    \[
    I_{main}^{OD',\beta}\ll_{l,\epsilon,\theta}\qterm K^{\theta+\epsilon}\norm{\psi}^2_{W^{T,\infty}},
    \]
    which in turns concludes the proof of \Cref{thm: Main Theorem}.
\end{prop}

The proof of \Cref{lem: Bound for J_main} is quite involved, and we shall delay some of the more technical aspects of the proof to later sections. We will set up the needed notation and propositions to provide a proof of this proposition, and conclude this subsection with its proof. We shall begin by looking at the exponential term $\exp\left(\frac{i\Delta}{2c}\right)$, we shall let $D = B^2-4AC$ and $R_i = 2Ar_i+B$. In doing this we see that that
\begin{align*}
    \frac{i\Delta}{2c}&=\frac{4\pi i}{c}\sqrt{\vert (Ar_1^2+Br_1+C)(Ar_2^2+Br_2+C)\vert}\\
    &=\frac{4\pi i}{c}\sqrt{\left\vert A\left(r_1+\frac{B}{2A}\right)^2+\left(C-\frac{B^2}{4A}\right)\right\vert\cdot \left\vert A\left(r_2+\frac{B}{2A}\right)^2+\left(C-\frac{B^2}{4A}\right)\right\vert}\\
    &=\frac{4\pi i}{c}A\left(r_1+\frac{B}{2A}\right)\left(r_2+\frac{B}{2A}\right)\sqrt{\left\vert 1+\frac{\left(C-\frac{B^2}{4A}\right)}{A\left(r_1+\frac{B}{2A}\right)^2}\right\vert}\sqrt{\left\vert 1+\frac{\left(C-\frac{B^2}{4A}\right)}{A\left(r_2+\frac{B}{2A}\right)^2}\right\vert}\\ 
    &=\frac{4\pi i}{c}A\left(\frac{R_1}{2A}\right)\left(\frac{R_2}{2A}\right)\sqrt{\left\vert 1-\frac{D}{4A^2\left(r_1+\frac{B}{2A}\right)^2}\right\vert}\sqrt{\left\vert 1-\frac{D}{4A^2\left(r_2+\frac{B}{2A}\right)^2}\right\vert}\\
    &=\frac{2\pi i}{c}\left(\frac{R_1R_2}{2A}\right)\sqrt{\left\vert 1-\frac{D}{R_1^2}\right\vert}\sqrt{\left\vert 1-\frac{D}{R_2^2}\right\vert}\\
    &=\frac{2\pi i}{c}\left(\frac{R_1R_2}{2A}\right)\left(1-\frac{1}{2}\frac{D}{R_1^2}-\frac{1}{8}\left(\frac{D}{R_1^2}\right)^2-...\right)\left(1-\frac{1}{2}\frac{D}{R_2^2}-\frac{1}{8}\left(\frac{D}{R_2^2}\right)^2-...\right)\\
    &=\frac{2\pi i}{c}\left(\frac{R_1R_2}{2A}\right)\left(1-\frac{1}{2}\left(\frac{D}{R_1^2}+\frac{D}{R_2^2}\right)+\frac{1}{4}\left(\frac{D^2}{R_1^2R_2^2}\right)^2-...\right).
\end{align*}

In particular, we see that the main term in the expression above will come from the $\frac{R_1R_2}{2A} = 2Ar_1r_2+B(r_1+r_2)+\frac{B^2}{2A}$ part of the exponential term. We shall let $e_c(x)=e\left(\frac{x}{c}\right)=\exp\left(\frac{2\pi ix}{c}\right)$. Thus, by taking out the main term from $\exp\left(\frac{i\Delta}{2c}\right)$, we shall define the function $f_c(y,r_1,r_2)$ by having  
\begin{equation*}\label{eq:Formula for Jpt1r1r2c after taking out the exponential term eciDelta2c}
    \psi\left(\frac{r_1}{X}\right)\psi\left(\frac{r_2}{X}\right)J^{p,\beta}_{r_1,r_2,c}(y)=e_c\left(2Ar_1r_2+B(r_1+r_2)+\frac{B^2}{2A}\right)f_{c}(y,r_1,r_2)=e_c\left(\frac{R_1R_2}{2A}\right)f_c(y,r_1,r_2).
\end{equation*}

Now we will have separate bounds depending upon the value of $c$. Thus, we shall define $J_{main}(c)$ to be such that 
\[
J_{main} = \frac{1}{X}\sum_{\substack{c\ll X^{2+\epsilon_0}K^{-1-\theta+\epsilon_1}\\ c\equiv N\Mod{d}} }J_{main}(c).
\]
In particular, for a fixed $c$, we have that 
\begin{align*}\label{eq: defn of J_main(c)}
    \begin{split}
    J_{main}(c):=&\sum_{r_1,r_2\geq 1}\psi\left(\frac{r_1}{X}\right)\psi\left(\frac{r_2}{X}\right)\frac{S(\vert q(r_1)\vert,\vert q(r_2)\vert;c)}{c}J^{p,\beta}_{r_1,r_2,c}(y)\\
    =&\sum_{r_1,r_2\geq 1}\frac{S(\vert q(r_1)\vert,\vert q(r_2)\vert;c)}{c}e_c\left(2Ar_1r_2+B(r_1+r_2)+\frac{B^2}{2A}\right)f_c(y,r_1,r_2).
    \end{split}
\end{align*}

Now we shall rewrite $cJ_{main}(c)$ by breaking the sum over $r_1$ and $r_2$ up into congruence classes modulo $c$. For the first equality below we are going through all $r_1$ by picking some equivalence class $a\tmod{c}$, and then for that equivalence class we will go through all $r_1\equiv a\tmod{c}$. This will go through all such $r_1$. Similarly, we shall go through all $b\tmod{c}$ for $r_2$. As for the last equality, we are simply multiplying by $e_c(ua+bv)e_c(-ur_1-vr_2)=1$ as $a\equiv r_1\tmod{c}$ and $b\equiv r_2\tmod{c}$. We do this for all $c^2$ possible choices of equivalence classes of $u,v\tmod{c}$ which is why we must divide by the factor of $c^2$. We remark that this addition of the $u$ and $v$ is to separate the analytic and arithmetic parts of $f$. This gives us: 

\begin{align*}
    cJ_{main}(c)&=\sum_{r_1,r_2\geq 1}S\left(\vert q(r_1)\vert,\vert q(r_2)\vert;c\right)e_c(2Ar_1r_2+B(r_1+r_2))f_c(y,r_1,r_2)\\
    &=\sum_{a,b\Mod{c}}S\left(\vert q(a)\vert,\vert q(b)\vert;c\right)e_c\left(2Aab+B(a+b)\right)\sum_{\substack{r_1\equiv a\Mod{c}\\
    r_2\equiv b\Mod{c}}}f_c(y,r_1,r_2)\\
    &=\frac{1}{c^2}\sum_{a,b,u,v\Mod{c}}S\left(\vert q(a)\vert,\vert q(b)\vert;c\right)e_c\left(2Aab+a(B+u)+b(B+v)\right)\\
    &\times \left(\sum_{(r_1,r_2)\in\zz^2}f_c(y,r_1,r_2)e_c(-ur_1-vr_2)\right).\\
\end{align*}

We note that as $f_c(y,r_1,r_2)$ has the function $\psi\left(\frac{r_1}{X}\right)\psi\left(\frac{r_2}{X}\right)$ as part of its definition rather than summing over all $r_1,r_2\geq 1$, we shall sum over all $(r_1,r_2)\in\zz^2$. This is important as we will use Poisson summation formula to bound the sum. In particular, we will prove the following bound. 

\begin{prop}\label{lem: Bound for Sum of fc}
    For $\epsilon>0$, $y\sim K$ and  $D\ll X$, choose $\epsilon_2$ such that  $\epsilon_0<\epsilon_2<\frac{2\epsilon}{5}$, then for  $DX^{\epsilon_2}\ll c\ll X^{2+\epsilon_0}K^{-1-\theta+\epsilon_1}$, we have that for fixed $T>\max\{100/(3\theta-1),100/(\epsilon_2-\epsilon_0)\}$, then the following bound holds
    \[
    \sum_{(r_1,r_2)\in\zz^2} f_c(r_1,r_2)e_c(-ur_2-vr_2)\ll_{l,\epsilon,\theta} c^{-1/2}X^{3+2\epsilon_0}K^{-2}\log K\norm{\psi}^2_{W^{T,\infty}}.
    \]
    Furthermore, there is a constant $\eta>0$ depending only on $l$ such that if either $u$ or $v$ are not in the range
    \[
    \left[\eta\inv\frac{c^2K^2}{2\pi X^{3-\epsilon_0}},\eta\frac{c^2K^2}{2\pi X^{3-\epsilon_0}}\right],
    \]
    then
    \[
    \sum_{(r_1,r_2)\in\zz^2}f_c(r_1,r_2)e_c(-ur_1-vr_2)\ll_{l,\epsilon,\theta}K^{-20}\norm{\psi}^2_{W^{T,\infty}}.
    \]
\end{prop}
The proof of this proposition is delayed and is the entire subject of \Cref{subsec: Proof of Lemma}, and also relies heavily on the technical bounds derived in \Cref{subsec: Bounds on Derivatives}.

As we now can bound this inner most sum, it now remains to bound
\[
\frac{1}{c^2}\sum_{a,b,u,v\Mod{c}}S\left(\vert q(a)\vert,\vert q(b)\vert;c\right)e_c\left(2Aab+a(B+u)+b(B+v)\right).
\]
To this affect, we shall now define 
\begin{equation*}\label{eq: Definition of Scd}
    S_c^{B,C}(\gamma)=\sum_{a,b\Mod{c}}S\left(\vert \gamma a^2+Ba+C\vert, \vert \gamma b^2+Bb+C\vert;c\right)e_c\left(2 \gamma ab + a(B+u)+b(B+v)\right),
\end{equation*}
so that we have 
\[
cJ_{main}(c) = \frac{1}{c^2}\sum_{u,v\Mod{c}}S_c^{B,C}(A)\sum_{(r_1,r_2)\in\zz^2}f_c(r_1,r_2)e_c(-ur_1-vr_2).
\]
Now we shall use the following proposition which we delay the proof to \Cref{subsec: Sums of Kloosterman Sums}. Furthermore, we remark that the condition that $(v,c_1)\mid u$ can be replaced with $(v,c_1)=(u,c_1)$ by the symmetry of $u$ and $v$ of the sum in terms of $u$ and $v$. 

\begin{prop}\label{lem: Kloosterman Sum Bound}
    For $c>0$, writing  $c=c_1c_2$ with $(4A,c_1)=1$, and $c_2\mid (2A)^\infty$, we have that  
    \begin{equation*}\label{eq: Kloosterman Sum Bound}
        S_c^{B,C}(A)=\begin{cases}
            O_\epsilon\left((v,c_1)c_1^{3/2}c_2^{5/2+\epsilon}\right) & \text{if $(v,c_1)\mid u$}\\
            0 & \text{else}.
        \end{cases}
    \end{equation*}
\end{prop}

We can now combine our propositions to prove the following bound when $DX^{\epsilon_2}\ll c\ll X^{2+\epsilon_0}K^{-1-\theta+\epsilon_1}$. It in turn will allow us to finish the proof of \Cref{lem: Bound for J_main}.

\begin{prop}
    For $\epsilon>0$, $y\sim K$,and $D\ll X$, choose $\epsilon_2$ such that $\epsilon_0<\epsilon_2<\frac{2\epsilon}{5}$, then we have the following bound for $DX^{\epsilon_2}\ll c\ll X^{2+\epsilon_0}K^{-1-\theta+\epsilon_1}$:
    \begin{equation*}
        J_{main}(c)\ll_{l,\epsilon,\theta } c_1^2c_2^3X^{-3}K^{2+\frac{\epsilon}{2}}\norm{\psi}^2_{W^{T,\infty}},
    \end{equation*}
    where $T>\max\{100/(3\theta-1),100/(\epsilon_2-\epsilon_0)\}$ and $c=c_1c_2$ with $(4A,c_1)=1$ and $c_2\mid (2A)^\infty$.
\end{prop}

\begin{proof}
    By applying \Cref{lem: Bound for Sum of fc} we have that consider two cases depending on whether or not one of $u$ or $v$ is not in the interval $\left[\eta\inv\frac{c^2K^2}{2\pi X^{3-\epsilon_0}},\eta\frac{c^2K^2}{2\pi X^{3-\epsilon_0}}\right]$. 
    In this case, we can just take the trivial bound on the Kloosterman sums, and we see that 
    \begin{align*}
        J_{main}(c)\ll_{l,\theta,\epsilon} & c^2K^{-20}\norm{\psi}^2_{W^{T,\infty}}.
    \end{align*}
    This clearly is less than the stated bound in the proposition. Thus, we now bound the main contribution coming from the case where $u,v\in\left[\eta\inv\frac{c^2K^2}{2\pi X^{3-\epsilon_0}},\eta\frac{c^2K^2}{2\pi X^{3-\epsilon_0}}\right]$. Then by \Cref{lem: Bound for Sum of fc} and \Cref{lem: Kloosterman Sum Bound}, we will have that 
    \[
    J_{main}(c)\ll_{l,\epsilon,\theta}\frac{1}{c^3}\sum_{\substack{u,v\in \left[\eta\inv\frac{c^2K^2}{2\pi X^{3-\epsilon_0}},\eta\frac{c^2K^2}{2\pi X^{3-\epsilon_0}}\right]\\ (u,c_1)=(v,c_1)}}(v,c_1)c_1^{3/2}c_2^{5/2+\frac{\epsilon}{100}}c^{-1/2}X^{3+2\epsilon_0}K^{-2}\log K\norm{\psi}^2_{W^{T,\infty}}.
    \]
    Now we note for any $M>0$, we have that 
    \begin{align*}
        \sum_{\substack{u,v<M\\ (u,c_1)=(v,c_1)}}(v,c_1)&=\sum_{\substack{d\mid c_1\\ d<M}}d\sum_{\substack{u,v<M\\ (u,c_1)=(v,c_1)=d}}1\\
        &=\sum_{\substack{d\mid c_1\\
        d<M}}d\sum_{\substack{u,v<M/d\\ (u,c_1/d)=(v,c_1/d)=1}}1\\
        &\leq \sum_{\substack{d\mid c_1\\ d<M}}d\sum_{u,v<M/d}1\\
        &\leq \sum_{\substack{d\mid c_1\\ d<M}}M^2/d\\
        &\leq \sum_{d<M}M^2/d\\
        &\ll_{\epsilon} M^{2+\frac{\epsilon}{100}}.
    \end{align*}
    Thus, we have that applying this with $M=\frac{c^2K^2}{X^{3-\epsilon_0}}$, and trivially bounding $\frac{c^2K^2}{X^{3-\epsilon_0}}\ll K^2$ this will give us that 
    \begin{align*}
        \sum_{\substack{u,v\in \left[\eta\inv\frac{c^2K^2}{2\pi X^{3-\epsilon_0}},\eta\frac{c^2K^2}{2\pi X^{3-\epsilon_0}}\right]\\ (u,c_1)=(v,c_1)}}(v,c_1)\ll_\epsilon \left(\frac{c^2K^2}{X^{3-\epsilon_0}}\right)^2K^{\frac{\epsilon}{50}}.
    \end{align*}
    Thus, we see that 
    \begin{align*}
        J_{main}(c)&\ll_{l,\epsilon,\theta}\frac{1}{c^3}\left(\frac{c^2K^2}{X^{3-\epsilon_0}}\right)^{2}K^{\frac{\epsilon}{50}}c_1^{3/2}c_2^{5/2+\frac{\epsilon}{100}}c^{-1/2}X^{3+2\epsilon_0}K^{-2}\log K\norm{\psi}^2_{W^{T,\infty}}\\
    &\ll_{l,\epsilon,\theta} c_1^2c_2^{3+\frac{\epsilon}{100}}X^{-3+4\epsilon_0}K^{2+\frac{\epsilon}{50}}\log K\norm{\psi}^2_{W^{T,\infty}}.\\
    \end{align*}
    Now as we have chosen $\epsilon_0<\frac{\epsilon}{10}$, and we trivially have that $c\ll K^2$, we will have that 
    \[
    c_2^{\frac{\epsilon}{100}}X^{4\epsilon_0}K^{\frac{\epsilon}{50}}\log K\ll K^{\frac{\epsilon}{2}},
    \]
    and we conclude that 
    \[
    J_{main}(c)\ll_{l,\epsilon,\theta}c_1^2c_2^3X^{-3}K^{2+\frac{\epsilon}{2}}\norm{\psi}^2_{W^{T,\infty}}.\qedhere
    \]
\end{proof}

\begin{prop}
    For $y\sim K$, $D\ll X$, and $c\ll D X^{\epsilon_2}$ we have the following bound 
    \begin{equation*}
        J_{main}(c)\ll_{l} c^{1/2}X \norm{\psi}^2_{L^\infty}.
    \end{equation*}
\end{prop}

\begin{proof}
    We see that by trivially bounding the Kloosterman sum and \Cref{lem: Bound on derivative of gc r1 r2} applied when $k_1=k_2=0$ so that we may bound $f_c(r_1,r_2)\ll_l X\inv c^{1/2}\norm{\psi}^2_{L^\infty} $. Thus, we will have that in this range that 
    \begin{align*}
        J_{main}(c) &\ll_l \sum_{r_1,r_2\geq 1}X^{-1}c^{1/2}\psi\left(\frac{r_1}{X}\right)\psi\left(\frac{r_2}{X}\right)\ll_lc^{1/2}X \norm{\psi}^2_{L^\infty}.\qedhere
    \end{align*}
\end{proof}

We can now conclude \Cref{lem: Bound for J_main} and hence finish the proof of \Cref{thm: Main Theorem} under the assumptions of \Cref{lem: Bound for Sum of fc} and \Cref{lem: Kloosterman Sum Bound}.

\begin{proof}[Proof of \Cref{lem: Bound for J_main}]
    We have that 
    \begin{align*}
        J_{main}=&\frac{1}{X}\sum_{c\ll X^{2+\epsilon_0}K^{-1-\theta+\epsilon_1}}J_{main}(c)\\
        =& \frac{1}{X}\sum_{c\ll D X^{\epsilon_2}}J_{main}(c)+\frac{1}{X}\sum_{D X^{\epsilon_2}\ll c\ll X^{2+\epsilon_0}K^{-1-\theta+\epsilon_1}}J_{main}(c)\\
        \ll_{l,\epsilon,\theta}&\frac{1}{X}\sum_{c\ll D X^{\epsilon_2}}c^{1/2}XK^{\epsilon_2}\norm{\psi}^2_{L^\infty}+\frac{1}{X}\sum_{D X^{\epsilon_2}\ll c\ll X^{2+\epsilon_0}K^{-1-\theta+\epsilon_1}}c_1^{2}c_2^{3}X^{-3}K^{2+\frac{\epsilon}{2}}\norm{\psi}^2_{W^{T,\infty}}\\
        \ll&\frac{1}{X}D^{3/2}X^{1+\frac{3\epsilon_2}{2}}K^{\epsilon_2}\norm{\psi}^2_{L^\infty}+X^{-4}K^{2+\frac{\epsilon}{2}}\norm{\psi}^2_{W^{T,\infty}}\sum_{c\ll X^{2+\epsilon_0}K^{-1-\theta+\epsilon_1}}c_1^{2}c_2^{3}.
    \end{align*}
    Now we observe that as we have chosen $\epsilon_2<\frac{2}{5}\epsilon$, we will have that 
    \[
    \frac{1}{X}D^{3/2}X^{1+\frac{3\epsilon_2}{2}}K^{\epsilon_2}\norm{\psi}^2_{L^\infty}\ll_{l,\epsilon,\theta}D^{3/2}K^{\epsilon}\norm{\psi}^2_{L^\infty}.
    \]
    Now for the second sum, we observe that 
    \begin{align*}
        \sum_{c\ll X^{2+\epsilon_0}K^{-1-\theta+\epsilon_1}}c_1^{2}c_2^{3}&<\sum_{\substack{c_2\ll X^{2+\epsilon_0}K^{-1-\theta+\epsilon_1}\\c_2\mid (2A)^\infty}}c_2^{3}\sum_{c_1\ll \frac{X^{2+\epsilon_0}K^{-1-\theta+\epsilon_1}}{c_2}}c_1^{2}\\
            &\ll \sum_{\substack{c_2\ll X^{2+\epsilon_0}K^{-1-\theta+\epsilon_1}\\c_2\mid (2A)^\infty}}X^{3(2+\epsilon_0)}K^{3(-1-\theta+\epsilon_1)}.
    \end{align*}
    Then as
    \[
    \sum_{\substack{c_2\ll X^2K^{-1-\theta+\epsilon_1}\\c_2\mid (2A)^\infty}}1 \ll \log_{2A}(K)=\frac{\log(K)}{\log(2A)}\ll_{\epsilon}K^{\frac{\epsilon}{2}},
    \]
    we conclude that 
    \begin{align*}
         \sum_{c\ll X^{2+\epsilon_0}K^{-1-\theta+\epsilon_1}}c_1^{2}c_2^{3}&\ll_\epsilon X^{6+3\epsilon_0}K^{-3-3\theta+3\epsilon_1+\frac{\epsilon}{2}}\\
        &\ll_\epsilon X^6K^{-4-(3\theta-1)/2+\frac{\epsilon}{2}}\\
        &\ll_\epsilon X^4K^{-2+\frac{\epsilon}{2}},
    \end{align*}
    where we have used that as $\epsilon_0,\epsilon_1<\frac{3\theta-1}{12}$, we will have that $3\epsilon_0+3\epsilon_1-3\theta<-1-\frac{3\theta-1}{2}$. Thus, combining these all we see that 
    \[
    J_{main}\ll_{l,\epsilon,\theta} (1+D^{3/2}) K^{\epsilon}\norm{\psi}^2_{W^{T,\infty}}.\qedhere
    \]
\end{proof}

\subsection{Kloosterman Sums}\label{subsec: Sums of Kloosterman Sums}
The  goal of this subsection will be to prove \Cref{lem: Kloosterman Sum Bound}. To prove this bound, we will show that $S_c^{B,C}(\gamma)$ will satisfy a multiplicativity relation in $c$. To define this multiplicativity relation, we set some notation for this section. If we factor $c=c_1c_2$ with $(c_1,c_2)=1$, we will then define the numbers $d_1,d_2$ where we take $d_1\Mod{c_2}$ and $d_2\Mod{c_1}$ such that $c_1d_1\equiv 1\Mod{c_2}$ and $c_2d_2\Mod{c_1}$. Furthermore, as $S_c^{B,C}(\gamma)$ is defined as a sum over $a,b\Mod{c}$, we shall break up the $a$ and $b$ as follows. Using the Chinese remainder theorem, we can find $a_1,b_1\Mod{c_1}$ and $a_2,b_2\Mod{c_2}$ such that $a=a_1c_2d_2+a_2c_1d_1$ and $b=b_1c_2d_2+b_2c_1d_1$.

\begin{prop}\label{prop: Proposition for multiplicative relation for e_c}
    With $a_1,b_1,c_1,d_1,a_2,b_2,c_2,d_2$ as described above, and $\gamma, u,v\in \zz$, we have that 
    \begin{align*}\label{eq: multiplicative relation for e_c}
        \begin{split}
        &e_c\left(2 \gamma ab + a(B+u)+b(B+v)\right)=\\
        &e_{c_1}\left(2\gamma a_1b_1d_2+a_1d_2(B+u)+b_1d_2(B+v)\right)e_{c_2}\left(2\gamma a_2b_2d_1+a_2d_1(B+u)+b_2d_1(B+v)\right).
        \end{split}
    \end{align*}
\end{prop}
\begin{proof}
    We simply rewrite $a,b$ in terms of the $a_i,b_i,c_i,d_i$. This will give us 
    \begin{align*}
        &e_c\left(2 \gamma ab + a(B+u)+b(B+v)\right) =\\ &e_{c}\Big(2\gamma(a_1c_2d_2+a_2c_1d_1)(b_1c_2d_2+b_2c_1d_1)+(a_1c_2d_2+a_2c_1d_1)(B+u)+(b_1c_2d_2+b_2c_1d_1)(B+v)\Big)
    \end{align*}
    Now we find that $ab=a_1b_1c_2^2d_2^2+a_2b_2c_1^2d_1^2+c(a_1b_2d_1d_2+a_2b_1d_1d_2)$, now due to the factor of $c$ in the denominator, we will have that the cross terms will just give a multiplicative factor of $1$ and can be canceled out. Now if we break up $c=c_1c_2$, one easily confirms that the statement is true by canceling out a factor of $c_2$ for each of the terms in the $e_{c_1}$ and similarly for the $e_{c_2}$ term. Then observing that as $c_1d_1\equiv 1\Mod{c_2}$ and $c_2d_2\equiv 1\Mod{c_1}$, we deduce the proposition. 
    \begin{align*}
        e_c\left(2 \gamma ab + a(u+B)+b(v+B)\right)=&e_{c_1c_2}\Big(2\gamma (a_1b_1c_2^2d_2^2+a_2b_2c_1^2d_1^2)+(a_1c_2d_2+a_2c_1d_1)(B+u))\\&+(b_1c_2d_2+b_2c_1d_1)(B+v)\Big)\\
        =&e_{c_1}\left(2\gamma a_1b_1c_2d_2^2+ a_1d_2(B+u)+b_1d_2(B+v)\right)\\
        \times& e_{c_2}\left(2\gamma a_2b_2c_1d_1^2+a_2d_1(B+u)+b_2d_1(B+v)\right)\\
        =&e_{c_1}\left(2\gamma a_1b_1d_2+ a_1d_2(B+u)+b_1d_2(B+v)\right)\\
        \times& e_{c_2}\left(2\gamma a_2b_2d_1+a_2d_1(B+u)+b_2d_1(B+v)\right).\qedhere
    \end{align*}
\end{proof}

Using the above multiplicative relation for the $ e_c\left(2 \gamma ab + a(u+B)+b(v+B)\right)$ term, we can easily derive the following multiplicative relation for $S_c^{B,C}(\gamma)$ using the Chinese remainder theorem. 

\begin{prop}\label{prop: Proposition for multiplicative relationship of Scd}
    With the $a_i,b_i,c_i,d_i$ described above, we have that $S_c^{B,C}(\gamma)$ satisfies the following multiplicative relation
    \begin{equation*}\label{eq: Multiplicativity property in c for the ScBC}
        S_{c_1c_2}^{B,C}(\gamma)=S_{c_1}^{B,Cd_2}(\gamma c_2)S_{c_2}^{B,Cd_1}(\gamma c_1).
    \end{equation*}
\end{prop}

\begin{proof}
    Using the Chinese remainder theorem, we break the sum $S_c^{B,C}(\gamma)$ into a sum over $a_1,b_1\Mod{c_1}$ and $a_2,b_2\Mod{c_2}$, and we can then apply \Cref{prop: Proposition for multiplicative relation for e_c}. Furthermore, we will use the classical multiplicative relation for the Kloosterman sums that 
    \begin{align*}
    \begin{split}
    S(\vert \gamma a^2+Ba +C\vert ,\vert \gamma b^2+Bb+C\vert;c)=&S\left(d_2\vert \gamma a^2+Ba +C\vert,d_2\vert \gamma b^2+Bb+C\vert;c_1\right)\\
    \times &S(d_1\vert \gamma a^2+Ba +C\vert,d_1\vert \gamma b^2+Bb+C\vert;c_2).
    \end{split}
    \end{align*}
    Using both of these will give us
    \begin{align*}
        &S_c^{B,C}(\gamma)=\sum_{a,b\Mod{c}}S(\vert \gamma a^2+Ba +C\vert ,\vert \gamma b^2+Bb+C\vert;c)e_c\left(2\gamma ab+a(B+u)+b(B+v)\right)\\
        &=\sum_{\substack{a_1,b_1\Mod{c_1}}}S\left(d_2\vert \gamma a_1^2+Ba_1 +C\vert,d_2\vert \gamma b_1^2+Bb_1+C\vert;c_1\right) e_{c_1}\left(2\gamma a_1b_1d_2+ a_1d_2(B+u)+b_1d_2(B+v)\right)\\
        &\times \sum_{\substack{a_2,b_2\Mod{c_2}}}S(d_1\vert \gamma a_2^2+Ba_2 +C\vert,d_1\vert \gamma b_2^2+Bb_2+C\vert;c_2)e_{c_2}\left(2\gamma a_2b_2d_1+a_2d_1(B+u)+b_2d_1(B+v)\right).
    \end{align*}
    Now if we do the change of variables $a_1\mapsto c_2a_1$, $b_1\mapsto c_2b_1$, $a_2\mapsto c_1a_2$, and $b_2\mapsto c_1b_2$, this will give us 
    \begin{align*}
        &\sum_{\substack{a_1,b_1\Mod{c_1}}}S\left(\vert \gamma c_2a_1^2+Ba_1 +Cd_2\vert,\vert \gamma c_2b_1^2+Bb_1+Cd_2\vert;c_1\right)e_{c_1}\left(2\gamma c_2a_1b_1+ a_1(B+u)+b_1(B+v)\right)\\
        &\times \sum_{\substack{a_2,b_2\Mod{c_2}}}S(\vert \gamma c_1 a_2^2+Ba_2 +Cd_1\vert,\vert \gamma c_1b_2^2+Bb_2+Cd_1\vert;c_2)e_{c_2}\left(2\gamma c_1a_2b_2+a_2(B+u)+b_2(B+v)\right).
    \end{align*}
    From this, we see that the above is $S_{c_1}^{B,Cd_2}(\gamma c_2)S_{c_2}^{B,Cd_1}(\gamma c_1)$.
\end{proof}

Now \Cref{prop: Proposition for multiplicative relationship of Scd} tells us that to understand $S_c^d(\gamma)$ it suffices to understand it for $c$ being prime powers. There are different behaviors depending upon whether or not $c$ is even or not. For $c$ odd, we can rewrite $S_c^d(\gamma)$ using Gauss sums in the following manner. 

\begin{prop}\label{prop: Rewriting Kloosterman sum for c odd using Gauss sums}
    For $(4\gamma,c)=1$ we have for that 
    \begin{align}\label{eq: Kloosterman Sum c odd rewrite with Gauss Sum}
        \begin{split}
        S_c^{B,C}(\gamma)=&\epsilon_cc^{1/2}\left(\frac{\gamma}{c}\right)\psum_{x\Mod{c}}\left(\frac{x}{c}\right)e_c\left((c-\overline{4\gamma}(B+U)^2)\overline{x}-\overline{4\gamma}B^2x-\overline{2\gamma}B(B+U)\right)\\
        &\times\sum_{b\Mod{c}}e_c\left(b(v-\overline{x}u)\right).
        \end{split}
    \end{align}
\end{prop}

\begin{proof}
    We derive this formula by opening up the Kloosterman sum 
    \begin{align*}
        S_c^{B,C}(\gamma)&=\sum_{a,b\Mod{c}}S\left(\vert \gamma a^2+Ba+C\vert, \vert \gamma b^2+Bb+C\vert;c\right) e_c(2\gamma ab+a(B+u)+b(B+v))\\
        &=\sum_{a,b\Mod{c}}\psum_{x\Mod{c}}e_c\left(x\vert \gamma a^2+Ba+C\vert +\overline{x}\vert \gamma b^2+Bb+C\vert +2\gamma ab+a(B+u)+b(B+v)\right).
    \end{align*}
    Now we observe that we may complete the square in the variable $a$ by observing that 
    \begin{align*}
        &x\vert \gamma a^2+Ba+C\vert +\overline{x}\vert \gamma b^2+Bb+C\vert +2\gamma ab+a(B+u)+b(B+v)\\
        &=x\gamma\left((\overline{2\gamma}B+\overline{x}b+\overline{2\gamma x}(B+u))+a\right)^2+(v-\overline{x}u)b+(c-\overline{4\gamma }(B+u)^2)\overline{x}-\overline{4\gamma}B^2x-\overline{2\gamma}B(B+u).
    \end{align*}
    If we substitute this into the above and rearrange our summation we have that 
    \begin{align*}
    S_c^{B,C}(\gamma)&=\psum_{x\Mod{c}}e_c\left((c-\overline{4\gamma}(B+u)^2)\overline{x}-\overline{4\gamma}B^2x-\overline{2\gamma}B(B+u)\right)\\
    &\times \sum_{b\Mod{c}}e_c\left(b(v-\overline{x}u)\right)\sum_{a\Mod{c}}e_c\left(x\gamma\left((\overline{2\gamma}B+\overline{x}b+\overline{2\gamma x}(B+u))+a\right)^2\right).
    \end{align*}
    Now we observe that our innermost sum over $a\Mod{c}$ is a Gauss sum, and we can explicitly compute as we have that 
    \[
    \sum_{t\Mod{c}}e_c(nt^2)=\left(\frac{n}{c}\right)\epsilon_cc^{1/2}.
    \]
    This will give us that 
    \[
    S_c^{B,C}(\gamma)=\epsilon_cc^{1/2}\left(\frac{\gamma}{c}\right)\psum_{x\Mod{c}}\left(\frac{x}{c}\right)e_c\left((c-\overline{4\gamma}(B+U)^2)\overline{x}-\overline{4\gamma}B^2x-\overline{2\gamma}B(B+u)\right)\sum_{b\Mod{c}}e_c\left(b(v-\overline{x}u)\right)
    \]
    which is exactly what we set out to prove. 
\end{proof}

We will now prove some preliminary bounds that when we put them together we will be able to get be able to prove \Cref{lem: Kloosterman Sum Bound}. 

\begin{prop}\label{prop: c odd gcd does not divide implies sum vanishes}
    For $(4c,\gamma)=1$ and $(v,c)\nmid u$, then we will have that $S_c^{B,C}(\gamma)=0$.
\end{prop}

\begin{proof}
    Using \Cref{prop: Rewriting Kloosterman sum for c odd using Gauss sums}, we will note that it suffices to prove that 
    \begin{equation}\label{eq: Sum of Roots of unity will be zero for odd Kloosterman sum}
    \sum_{b\Mod{c}}e_c\left(b(v-\overline{x}u)\right)=0.
    \end{equation}
    We shall do this by showing that the sum will be over all roots of unity and hence will be $0$. We claim that if $(v,c)\nmid u$, then we will have that $c\nmid (v-\overline{x}u)$. This is clear by the contrapositive. Since if $c\mid v-\overline{x}u$, then we will have that $v-\overline{x}u=nc$ for some $n\in\zz$. Thus, we will have that $v-nc=\overline{x}u$. Let us denote $d:=(v,c)$, so if we write $v=dv'$ and $c=dc'$, then we will have that $\overline{x}u=d(v'-nc')$. Thus, we conclude that $d\mid \overline{x}u$, but as $(x,c)=1$, we may conclude that $d=(v,c)\mid u$.

    Thus, as we have that $c\nmid (v-\overline{x}u)$. Now if we consider $f:=(c,v-\overline{x}u)$, then if we write $c=fc'$ and $v-\overline{x}u=f(v-\overline{x}u)'$, we will have that $e_{c}(b(v-\overline{x}u))=e_{c'}\left(b(v-\overline{x}u)'\right)$ will be a primitive root of unity. Then can see that \Cref{eq: Sum of Roots of unity will be zero for odd Kloosterman sum} will simply be a sum of all $c'$ roots of unity (with multiplicity $f$), and hence will be $0$. 
\end{proof}

\begin{prop}\label{prop: c odd gcd divides Kloosterman bound}
    If $(4c,\gamma)=1$ and $(v,c)\mid u$, then we will have that $S_c^{B,C}(\gamma)=O((v,c)c^{3/2})$
\end{prop}

\begin{proof}
We note that
\[
\sum_{b\Mod{c}}e_c(b(v-\overline{x}u))=0
\]
unless we have that $c\mid v-\overline{x}u$, where in the latter case this sum is $c$. Thus, using the trivial bound of $1$ for all the exponential terms in \Cref{eq: Kloosterman Sum c odd rewrite with Gauss Sum}, we see that 
\[
S_c^{B,C}(\gamma)\ll c^{3/2}\psum_{\substack{x\Mod{c}\\ c\mid v-\overline{x}u}}1.
\]
Now the proposition follows if we show that the sum is $O((v,c))$. Indeed, if we denote $d:=(v,c)$, then we may write $c=dc'$, $v=dv'$ and $u=du'$ where we will have that $(c',v')=1$. Thus, the condition that $(x,c)=1$, and $c\mid v-\overline{x}u$ may be rewritten as $x\equiv u'\overline{v'}\Mod{c'}$. There is a unique $x$ satisfying this congruence modulo $c'$, but there will be $d$ such solutions modulo $c$. Thus, we see that this sum will be $d=(v,c)$, and the proposition is proven.
\end{proof}

\begin{prop}\label{prop: c power of two Kloosterman bound}
    If $c\mid (4\gamma)^\infty$, then $S_c^{B,C}(\gamma)=O(c^{5/2+\epsilon})$.
\end{prop}

\begin{proof}
    In this case we shall simply use the definition of $S_c^{B,C}(\gamma)$ and apply the Weil bound for Kloosterman sums. From this we conclude that 
    \[
    S_c^{B,C}(\gamma)\ll_\epsilon  c^{1/2+\epsilon/3}\sum_{a,b\Mod{c}}(\vert\gamma a^2+Ba+C\vert,\vert\gamma b^2+Bb+C\vert;c)^{1/2}.
    \]
    We see that clearly, the sum over $b$ can be bounded by $c$, and as $(\vert\gamma a^2+Ba+C\vert,\vert\gamma b^2+Bb+C\vert;c)^{1/2}\leq (\vert \gamma a^2+Ba+C\vert,c)$, we will have that the proposition will follow if we show that 
    \[
    \sum_{a\Mod{c}}(\vert \gamma a^2+Ba+C\vert, c)\ll_\epsilon c^{1+2\epsilon/3}.
    \]
    Indeed, we can write 
    \begin{align*}
        \sum_{a\Mod{c}}(\vert \gamma a^2+Ba+C\vert, c)&=\sum_{a\Mod{c}}\sum_{d\mid c}\phi(d)\delta_{d\mid \vert \gamma a^2+Ba+C\vert}\\
        &=\sum_{d\mid c} \phi(d)\sum_{\substack{a\Mod{c}\\ d\mid \vert \gamma a^2+Ba+C\vert}}1.
    \end{align*}
    Then, we see that the inner sum is counting the number of $a\Mod{c}$ where $d\mid \vert \gamma a^2+Ba+C\vert$. However, we can bound this above by looking at the number of solutions modulo $d$ instead, and observing that for each solution modulo $c$, there will  be $c/d$ such solutions modulo $d$, that is 
    \[
    \sum_{\substack{a\Mod{c}\\ \vert \gamma a^2+Ba+C\vert\equiv 0\Mod{d}}}1\leq \frac{c}{d}\sum_{\substack{a\Mod{d}\\ \vert \gamma a^2+Ba+C\vert\equiv 0\Mod{d}}}1.
    \]
    Now the number of $a\Mod{d}$ where $d\mid \vert \gamma a^2+Ba+C\vert$, is just counting the number of roots modulo $d$ to the polynomial $\vert \gamma a^2+Ba+C\vert = 0$. For each prime power dividing $d$, there are at most $2$ solutions to $\vert\gamma a^2+Ba+C\vert\equiv 0\Mod{d}$; thus, we have that by the Chinese remainder theorem there are at most $2^{\omega(d)}$ solutions where $\omega(d)$ counts the number of distinct prime factors of $d$. However, as $2^{\omega(d)}\leq \tau(d)\ll_\epsilon d^{\epsilon/3}\ll_\epsilon c^{\epsilon/3}$. Thus, we conclude the proposition as 
    \[
    \sum_{d\mid c}\frac{\phi(d)}{d}\leq \tau(d)\ll_\epsilon c^{\epsilon/3}.\qedhere
    \]
\end{proof}

We can now combine all of our propositions to give a proof of \Cref{lem: Kloosterman Sum Bound}.

\begin{proof}[Proof of \Cref{lem: Kloosterman Sum Bound}]
    We shall prove this in two steps, as $q(x)=Ax^2+Bx+C$ is an integer valued polynomial. The first case to consider is when $A,B,C\in\zz$. In this case, to get our bound on $S_c^{B,C}(A)$, we use \Cref{prop: Proposition for multiplicative relationship of Scd} to rewrite 
    \[
    S_c^{B,C}(A)=S_{c_1}^{B,Cd_2}(Ac_2)S_{c_2}^{B,Cd_1}(Ac_1).
    \]
    Now using \Cref{prop: c odd gcd does not divide implies sum vanishes} we will have that $S_{c_1}^{B,Cd_2}(Ac_2)=0$ if $(v,c_1)\nmid u$. Hence, $S_c^{B,C}(A)=0$ if $(v,c_1)\nmid u$. Then in the case where $(v,c_1)\mid u$, by combining the bounds from \Cref{prop: c odd gcd divides Kloosterman bound} and \Cref{prop: c power of two Kloosterman bound} we conclude the bound in the proposition. 

    Now there is another case to consider, in particular, if $q(x)\notin \zz[x]$. However, as $q$ is integer valued, this will happen only if $A,B\notin\zz$, but $2A,2B\in\zz$. In this case we shall write $A=\frac{A'}{2}$ and $B=\frac{B'}{2}$ with $A',B'\in\zz$. Now there are two cases to consider depending on whether or not $c$ is even or odd. If $c$ is even, then we can write $c=2c'$, and then it is clear that $S_c^{B,C}(A)=S_{c'}^{B',C}(A')$, and then everything is an integer, so our previous proof holds. In the case when $c$ is odd, then as our sums are working modulo $c$, and the only denominator appearing is $2$ which will be invertible modulo $c$, then our proof will still hold. 
\end{proof}

\subsection{Bounds on Derivatives}\label{subsec: Bounds on Derivatives}

The goal of this section is to provide bounds on derivatives. The main goal at the end of this section is to give prove the following proposition.

\begin{prop}\label{lem: Bound on derivative d/dr1dr2 fc}
For $y\sim K$, $r_i\sim_l X$, $D\ll X$,  $DX^{\epsilon_2}\ll c\ll X^{2+\epsilon_0}K^{-1-\theta+\epsilon_1}$, and all $k_1,k_2\geq 0$, we have the following bound
    \begin{equation*}
        \frac{\partial^{k_1+k_2}}{\partial r_1^{k_1}\partial r_2^{k_2}} f_c(y,r_1,r_2)\ll_{l,k_1,k_2}c^{1/2}X^{-1}\left(K^{\frac{1-3\theta}{4}}\right)^{k_1+k_2}\norm{\psi}^2_{W^{k_1+k_2,\infty}}.
    \end{equation*}
\end{prop}

Throughout this section, we shall always assume that $y\sim K$, $r_i\sim_l X$, and $c\ll X^{2+\epsilon_0}K^{-1-\theta+\epsilon_1}$. This is so that $\frac{X^2}{c}\ll_lx\ll_l\frac{X^{2+\epsilon_0}}{c}$ and $\beta\sim\frac{y}{x}\ll_l \frac{Kc}{X^{2}}\ll_l X^{\epsilon_0}K^{-\theta+\epsilon_1}$. Furthermore, there will be some theorems where we must also assume the lower bound on $c$, namely that $DX^{\epsilon_2}\ll c\ll X^{2+\epsilon_0}K^{-1-\theta+\epsilon_1}$.

Now it takes a bit of work to derive this bound. In particular, we must apply the chain rule and product rule multiple times. There are two ways one might write down how to prove this bound. Namely, we could do a bottom up approach where we derive different bounds and put them together building up to our bound for $\frac{\partial^{k_1+k_2}}{\partial r_1^{k_1}\partial r_2^{k_2}} f_c(r_1,r_2,y)$. This approach has the disadvantage of it not being clear why we are proving the bounds at the beginning until we put them together at the end, but has the advantage that everything we use has been already proven. The other option is to do a top down approach where we apply derivative rules to $\frac{\partial^{k_1+k_2}}{\partial r_1^{k_1}\partial r_2^{k_2}} f_c(r_1,r_2,y)$ to deduce the bound if we knew other bounds for derivatives of other functions. This approach has the disadvantage of us continually passing the buck to a new proposition that we must prove to give the proof of our desired bound, but has the advantage of it becomes apparent why we must prove all the bounds we need. We opt to derive the bound in this top down approach where we will take the derivative, and then use a bound which we shall prove in a subsequent proposition. 

The first step is to break up the function $f_c(r_1,r_2,y)$ into a real part and then an oscillatory term coming from the exponential factors. In particular, we shall write 
\begin{equation*}\label{eq: definition of fc as product of gc exp i phi}
f_c(r_1,r_2,y):= g_c(r_1,r_2,y)\exp(i\phi(r_1,r_2,y)).
\end{equation*}
In the above, we shall write 
\begin{equation*}\label{def of phi r1,r2,y}
    \phi(r_1,r_2,y)=\alpha(r_1,r_2,y)+\left(x-\frac{2\pi}{c}(2AR_1R_2)\right),
\end{equation*}
where 
\begin{equation*}\label{eq: def of alpha r1,r2,y}
    \alpha(r_1,r_2,y) = \frac{x}{2\pi}\left(-2\pi^2\beta^2+\frac{2}{3}\pi^4\beta^4\right)+\beta y-\frac{\pi}{4}.
\end{equation*}

We remark that as $\beta\sim \frac{y}{x}$, both terms in this function asymptotically look like $\frac{y^2}{x}$, so both terms in this function will end up behaving in the same manner when taking derivatives. Furthermore, we will have that the function $g_c(r_1,r_2,y)$ will be given by
\begin{equation}\label{eq: def of gcr1r2y)}
    g_c(r_1,r_2,y)=\psi\left(\frac{r_1}{X}\right)\psi\left(\frac{r_2}{X}\right)\left\vert \frac{x}{2\pi}(-2\pi^2+8\pi^4\beta^2)\right\vert^{-1/2}.
\end{equation}

Now before getting any of our bounds, we remark that the chain rule will be one of the primary tools that we use; however, as we are working with higher derivatives in a multi-variable setting, we need to use the lesser known Fa\'a di Bruno's formula for the chain rule. We state it here along with the notation that is used. 

\begin{lemma}\label{lem: Faa Di Bruno}
    Let $y=g(x_1,...,x_n)$, where the $x_i$ may be distinct or repeat. Then for $f$ a function of $y$, we have that  
    \[
    \frac{\partial^n}{\partial x_1...\partial x_n} f(y)=\sum_{\pi\in\Pi} f^{(\vert \pi\vert)}(y)\cdot \prod_{B\in \pi}\frac{\partial^{\vert B\vert }}{\prod_{j\in B}\partial x_j}y,
    \]
    where we have that $\Pi$ is the set of all partitions of $\{1,2,...,n\}$, $B\in \pi$ means that $B$ runs through the list of all the blocks of the partition $\pi$, $\vert \pi\vert$ is the number of parts in the partition $\pi$ and $\vert B\vert$ denotes the size of the block $B$. 
\end{lemma}

As in our application, we are just trying to get bounds on our partial derivatives, the important part about Fa\'a Di Brunos formula is that we have that our derivative is a linear combination of terms. In particular, if $y=g(r_1,r_2)$, and then given a function $f$, we will have that $\frac{\partial^{k_1+k_2}}{\partial r_1^{k_1}\partial r_2^{k_2}} f(y)$
will be a linear combination of terms involving $f^{(m)}(y)$. where $1\leq m\leq k_1+k_2$ is associated to some partition of $k_1+k_2$ into $m$ parts. Then we will have that this term $f^{(m)}(y)$ will be multiplied by the product of the derivatives of $y$ corresponding to the partition. In particular, we will have that $\frac{\partial^{k_1+k_2}}{\partial r_1^{k_1}\partial r_2^{k_2}} f(y)$ will be a linear combination of terms of the form
\[
f^{(m)}(y)\prod_{i=1}^{m} \frac{\partial^{a_i+b_i}}{\partial r_1^{a_i}\partial r_2^{b_i}} y,
\]
where we have that for each $i$, $a_i,b_i\geq 0$, $a_i+b_i\geq 1$, $\sum_{i=1}^ma_i = k_1$ and $\sum_{i=1}^mb_i = k_2$.

\begin{proof}[Proof of \Cref{lem: Bound on derivative d/dr1dr2 fc}]
    As we have that $f_c(r_1,r_2,y) = g_c(r_1,r_2,y)\exp(i\phi(r_1,r_2,y))$, when we apply the product rule, we will have that $\frac{\partial^{k_1+k_2}}{\partial r_1^{k_1}\partial r_2^{k_2}}f_c(r_1,r_2,y)$ will be a linear combination of terms of the form 
    \[
    \frac{\partial^{a_1+a_2}}{\partial r_1^{a_1}\partial r_2^{a_2}}\left(g_c(r_1,r_2,y)\right)\frac{\partial^{k_1-a_1+k_2-a_2}}{\partial r_1^{k_1-a_1}\partial r_2^{k_2-a_2}}\left(\exp(i\phi(r_1,r_2,y))\right),
    \]
    where $0\leq a_i\leq k_i$. Now using \Cref{eq: Bound on derivative d/dr1dr2 gcr1r2} and \Cref{eq: Bound on derivative d/dr1dr2 e(i phi)}, we will have that $\frac{\partial^{a_1+a_2}}{\partial r_1^{a_1}\partial r_2^{a_2}}\left(g_c(r_1,r_2,y)\right)\ll_l  c^{1/2}X^{-1+(2\epsilon_0-1)(a_1+a_2)}\norm{\psi}^2_{W^{a_1+a_2,\infty}}$ and $\frac{\partial^{k_1-a_1+k_2-a_2}}{\partial r_1^{k_1-a_1}\partial r_2^{k_2-a_2}}\left(\exp(i\phi(r_1,r_2,y))\right)\ll_{l,k_1,k_2} \left(\frac{cK^2}{X^{3-2\epsilon_0}}\right)^{k_1-a_1+k_2-a_2}$. Thus, we have that 
    \begin{align*}
    \frac{\partial^{k_1+k_2}}{\partial r_1^{k_1}\partial r_2^{k_2}}f_c(r_1,r_2,y)&\ll_{l,k_1,k_2} c^{1/2}X^{-1+(2\epsilon_0-1)(k_1+k_2)}\norm{\psi}^2_{W^{k_1+k_2,\infty}}\left(\frac{cK^2}{X^2}\right)^{k_1-a_1+k_2-a_2}\\
    &\ll_{l,k_1,k_2}c^{1/2}X^{-1}\left(\frac{cK^2}{X^{3-2\epsilon_0}}\right)^{k_1+k_2}\norm{\psi}^2_{W^{k_1+k_2,\infty}},
    \end{align*}
    where we have replaced the $k_1-a_1+k_2-a_2$ with $k_1+k_2$ as we have that as $c\ll X^{2+\epsilon_0}K^{-1-\theta+\epsilon}$ implies that $\frac{cK^2}{X^2}$ will be largest when the exponent is largest.

    Now as we have $1\ll DX^{\epsilon_2}\ll c\ll X^{2+\epsilon_0}K^{-1-\theta+\epsilon_1}$, we will have that $K^{\frac{1+\theta-\epsilon_1}{2}}X^{\frac{-\epsilon_0}{2}}\ll X$, and so we will have that
    \[
    \frac{cK^2}{X^{3-2\epsilon_0}}\ll \frac{K^{1-\theta+\epsilon_1}}{X^{1-3\epsilon_0}}\ll K^{\frac{1}{2}-\frac{3\theta}{2}+\frac{3\epsilon_1}{2}}X^{\frac{5\epsilon_0}{2}}\ll K^{\frac{1-3\theta}{4}},
    \]
    where in the last inequality we have used the assumptions that $\epsilon_0<\frac{3\theta-1}{20}$ and $\epsilon_1<\frac{3\theta-1}{12}$, so we have that $\frac{3\epsilon_1}{2}+\frac{5\epsilon_0}{2}<\frac{3\theta-1}{4}$. 
    Thus, we may conclude the proposition.
\end{proof}

Thus, we see that we all that remains is to get our bounds on the partial derivatives of $g_c$ and $\phi$. We shall first get bounds on the derivatives $g_c$, and then at the end of this section we shall get bounds on the derivatives of $\phi$. As in the next section we shall apply the principle of stationary phase, and the phase function will end up being $\phi$, so it will be a natural progression. 

\begin{prop}\label{lem: Bound on derivative of gc r1 r2}
    For $y\sim K$, $r_i\sim_l X$, $c\ll X^{2+\epsilon_0}K^{-1-\theta+\epsilon_1}$, and all $k_1,k_2\geq 0$, we have the following bound
    \begin{equation}\label{eq: Bound on derivative d/dr1dr2 gcr1r2}
        \frac{\partial^{k_1+k_2}}{\partial r_1^{k_1}\partial r_2^{k_2}}g_c(r_1,r_2,y)\ll_l  c^{1/2}X^{-1+(2\epsilon_0-1)(k_1+k_2)}\norm{\psi}^2_{W^{k_1+k_2,\infty}}.
    \end{equation}
\end{prop}

\begin{proof}
    We see that using the explicit formula, \Cref{eq: def of gcr1r2y)} , for $g_c(r_1,r_2,y)$ and the product rule that we will have that $\frac{\partial^{k_1+k_2}}{\partial r_1^{k_1}\partial r_2^{k_2}}g_c(r_1,r_2,y)$ is a linear combination of terms of the form 
    \begin{align*}
        X^{-a_1-a_2}\psi^{(a_1)}\left(\frac{r_1}{X}\right)\psi^{(a_2)}\left(\frac{r_2}{X}\right)\frac{\partial^{k_1-a_1+k_2-a_2}}{\partial r_1^{k_1-a_1}\partial r_2^{k_2-a_2}}\left\vert x\frac{(-4\pi+8\pi^4\beta^2)}{2\pi}\right\vert^{-1/2},
    \end{align*}
    where $0\leq a_i\leq k_i$. Now applying the bound of \Cref{eq: Bound on derivaitve d/dr1dr2 x(4pi+2pi beta2)1/2}, we have $\frac{\partial^{k_1-a_1+k_2-a_2}}{\partial r_1^{k_1-a_1}r_2^{k_2-a_2}}\left\vert x\frac{(-4\pi+8\pi^4\beta^2)}{2\pi}\right\vert^{-1/2}\ll_l c^{1/2}X^{-1+(2\epsilon_0-1)(k_1-a_1k_2-a_2)}$. Thus, we get that
    \[
    \frac{\partial^{k_1+k_2}}{\partial r_1^{k_1}\partial r_2^{k_2}}g_c(r_1,r_2,y)\ll_l c^{1/2}X^{-1+(2\epsilon_0-1)(k_1+k_2)}\norm{\psi}^2_{W^{k_1+k_2,\infty}}.\qedhere
    \]
\end{proof}

We note that the derivatives of $\psi$ contribute to the Sobolev norms in our estimates, so it remains to bound $\frac{\partial^{k_1+k_2}}{\partial r_1^{k_1}r_2^{k_2}}\left\vert x\frac{(-4\pi+8\pi^4\beta^2)}{2\pi}\right\vert^{-1/2}$ which depends upon knowing derivatives of $x$ and $\beta$, we do the derivatives with respect to $x$ first and delay finding the derivative with respect to $\beta$ until afterwards. The next three lemmas are those involving the $x$, and then we will deal with those involving $\beta$.

\begin{prop}\label{lem: Bound on derivaitve d/dr1dr2 x(4pi+2pi beta2)1/2}
    For $y\sim K$, $r_i\sim_l X$, $c\ll X^{2+\epsilon_0}K^{-1-\theta+\epsilon_1}$, and for all $k_1,k_2\geq 1$ we have the following bound 
    \begin{equation}\label{eq: Bound on derivaitve d/dr1dr2 x(4pi+2pi beta2)1/2}
    \frac{\partial^{k_1+k_2}}{\partial r_1^{k_1}r_2^{k_2}}\left\vert x\frac{(-4\pi+8\pi^4\beta^2)}{2\pi}\right\vert^{-1/2}\ll_l c^{1/2}X^{-1+(2\epsilon_0-1)(k_1+k_2)}.
    \end{equation}
\end{prop}

\begin{proof}
    Using the chain rule, we see that the this will be a linear combination of terms of the form 
    \[
    \left\vert x\frac{(-4\pi+8\pi^4\beta^2)}{2\pi}\right\vert^{-1/2-m}\prod_{i=1}^m\frac{\partial^{a_i+b_i}}{\partial r_1^{a_i}\partial r_2^{b_i}}\left(x\frac{(-4\pi+8\pi^4\beta^2)}{2\pi}\right),
    \]
    where $1\leq m\leq k_1+k_2$, $a_i,b_i\geq 0$, $a_i+b_i\geq 1$, $\sum_{i=1}^ma_i=k_1$, and $\sum_{i=1}^mb_i=k_2$. Since $\beta\ll_l X^{\epsilon_0}K^{-\theta+\epsilon_1}$, we will have that $\left\vert x\frac{(-4\pi+8\pi^4\beta^2)}{2\pi}\right\vert^{-1/2-m}\sim_l x^{-1/2-m}$. Now using \Cref{eq: Bound on derivative d/dr1dr2 x(4pi+8pi beta2)}, we have that $\frac{\partial^{k_1+k_2}}{\partial r_1^{k_1}\partial r_2^{k_2}}\left(x\frac{(-4\pi+8\pi^4\beta^2)}{2\pi}\right)\ll_l \frac{1}{c}X^{2+\epsilon_0+(\epsilon_0-1)(k_1+k_2)}$. Thus as $\frac{X^2}{c}\ll_lx\ll_l\frac{X^{2+\epsilon_0}}{c}$, we see that each of these terms is 
    \[
    \ll_l  x^{-1/2-m}\prod_{i=1}^m\frac{1}{c}X^{2+\epsilon_0+(\epsilon_0-1)(a_i+b_i)}\ll_l \frac{X^{-1-2m}}{c^{-1/2-m}}\frac{X^{2m+m\epsilon_0+(\epsilon_0-1)(k_1+k_2)}}{c^m}\ll_l c^{1/2}X^{-1+(2\epsilon_0-1)(k_1+k_2)}.\qedhere
    \]
\end{proof}

\begin{prop}\label{lem: Bound on derivative d/dr1dr2 x(4pi+8pi beta2)}
    For $y\sim K$, $r_i\sim_l X$, $c\ll X^{2+\epsilon_0}K^{-1-\theta+\epsilon_1}$, and all $k_1,k_2\geq 0$, we have the following bound
    \begin{equation}\label{eq: Bound on derivative d/dr1dr2 x(4pi+8pi beta2)}
    \frac{\partial^{k_1+k_2}}{\partial r_1^{k_1}\partial r_2^{k_2}}\left(x\frac{(-4\pi+8\pi^4\beta^2)}{2\pi}\right)\ll_l \frac{1}{c}X^{2+\epsilon_0+(\epsilon_0-1)(k_1+k_2)}.
    \end{equation}
\end{prop}

\begin{proof}
    We shall use the product rule to see that we have that the partial derivative will be a linear combination of terms of the form
    \[
    \frac{\partial^{a_1+a_2}}{\partial r_1^{a_1}\partial r_2^{a_2}}\frac{x}{2\pi}\frac{\partial^{b_1+b_2}}{\partial r_1^{b_1}\partial r_2^{b_2}}(-4\pi+8\pi^4\beta^2),
    \]
    where we have that $a_i,b_i\geq 0$, $a_1+b_1=k_1$, and $a_2+b_2=k_2$. Now we shall use \Cref{eq: Bound on derivative d/dr1dr2 x} and \Cref{eq: Bound on derivative d/dr1dr2 beta n}, we will have that $\frac{\partial^{a_1+a_2}}{\partial r_1^{a_1}\partial r_2^{a_2}}\frac{x}{2\pi}\ll_l \frac{X^{2+\epsilon_0-a_1-a_2}}{c}$ and $\frac{\partial^{b_1+b_2}}{\partial r_1^{b_1}\partial r_2^{b_2}}(-4\pi+8\pi^4\beta^2)\ll_l \beta^2X^{(\epsilon_0-1)(b_1+b_2)}\ll_l X^{(\epsilon_0-1)(b_1+b_2)}$. 
\end{proof}

\begin{prop}\label{lem: Bound on derivative d/dr1dr2 x}
    For  $r_i\sim_l X$, $c\ll X^{2+\epsilon_0}K^{-1-\theta+\epsilon_1}$, and all $k_1,k_2\geq 0$, we have the following bound
    \begin{equation}\label{eq: Bound on derivative d/dr1dr2 x}
    \frac{\partial^{k_1+k_2}}{\partial r_1^{k_1}\partial r_2^{k_2}}x\ll_l \frac{1}{c}X^{2+\epsilon_0-k_1-k_2}.
    \end{equation}
\end{prop}

\begin{proof}
    As $x=\frac{4\pi\sqrt{\vert q(r_1)q(r_2)\vert}}{c}$, and $c$ is just a constant we can pull it out of the derivative. Furthermore, as we have fixed $q(r_i)=Ar_i^2+BR_i+C$, we can pull out a factor of $A\ll X^{\epsilon_0}$ from $x$ as well and we may assume that $q$ is monic. Thus, we just compute $\frac{\partial^{k_1+k_2}}{\partial r_1^{k_1}\partial r_2^{k_2}}\left(\vert q(r_1)q(r_2)\vert^{1/2}\right)$ in the case where $q$ is monic. Now using the chain rule, we will have that $\frac{\partial^{k_1+k_2}}{\partial r_1^{k_1}\partial r_2^{k_2}}\left(\vert q(r_1)q(r_2)\vert^{1/2}\right)$ will be a linear combination of terms of the form 
    \[
    (\vert q(r_1)q(r_2)\vert)^{\frac{1}{2}-m}\prod_{i=1}^m\frac{\partial^{a_i+b_i}}{\partial r_1^{a_i}\partial r_2^{b_i}}(\vert q(r_1)q(r_2)\vert)
    \]
    where $1\leq m\leq k_1+k_2$, $a_i,b_i\geq 0$, $a_i+b_i\geq 1$, $\sum_{i=1}^ma_i= k_1$, and $\sum_{i=1}^mb_i = k_2$. Now as $q$ is a quadratic polynomial, we remark that if any $a_i,b_i\geq 3$, then the partial derivative will be $0$. Thus, we will make the assumption that $a_i,b_i<3$. Thus, we will have that 
    \[
    \frac{\partial^{a_i+b_i}}{\partial r_1^{a_i}\partial r_2^{b_i}}(\vert q(r_1)q(r_2)\vert)\ll r_1^{2-a_i}r_2^{2-b_i}.
    \]
    Now we observes that 
    \begin{align*}
    (\vert q(r_1)q(r_2)\vert)^{\frac{1}{2}-m}\prod_{i=1}^m\frac{\partial^{a_i+b_i}}{\partial r_1^{a_i}\partial r_2^{b_i}}(\vert q(r_1)q(r_2)\vert)&\ll r_1^{1-2m}r_2^{1-2m}r_1^{\sum_{i=1}^m2-a_i}r_2^{\sum_{i=1}^m2-b_i}\\
    &\ll r_1^{1-\sum_{i=1}^m a_i}r_2^{1-\sum_{i=1}^m b_i}\\
    &\ll_l X^{2-k_1-k_2}.
    \end{align*}
    Where in the last line, we have used that $r_i\sim_l X$.
\end{proof}

Now it remains only to calculate the derivative $\frac{\partial^{k_1+k_2}}{\partial r_1^{k_1}\partial r_2^{k_2}}\beta$ and powers of $\beta$. This is the most involved derivative to take. The key observation is that we shall use the fact that $\beta$ is a point of stationary phase. Now we shall let $g(t)=-4\pi^2t+\frac{8}{3}\pi^4t^3$. The importance of this function is that we have that if $\beta$ is our point of stationary phase, we have that 
    \[
    x(-4\pi^2\beta+\frac{8}{3}\pi^4\beta^3)+2\pi y=0.
    \]
    We will have that $\beta$ is one of the $3$ points in the level set $g(\beta)=\frac{-2\pi y}{x}$. Thus, we will have that $\beta=g^{-1}\left(\frac{-2\pi y}{x}\right)$. As this $\beta$ is our point of stationary phase near $0$ by $g^{-1}$ we mean the local inverse given by the implicit function theorem. Now we wish to get bounds on $\frac{\partial^n}{\partial x^n}\beta$. In doing so, we will need bounds on $\frac{\partial^n}{\partial x^n}g^{-1}$. We do this in multiple steps. 

    The first important idea, is that we shall consider $\beta$ being a function of $x$, but as $x$ has a dependence upon $r_1$ and $r_2$, we can then apply the chain rule to get $\frac{\partial^{k_1+k_2}}{\partial r_1^{k_1}\partial r_2^{k_2}}\beta$ in terms of derivatives $\frac{\partial^m}{\partial x^m}\beta$ and $\frac{\partial^{k_1+k_2}}{\partial r_1^{k_1}\partial r_2^{k_2}} x$. We state this as the following lemma. 

    \begin{prop}\label{lem: Bound on derivative d/dr1dr2 beta}
        For $y\sim K$, $r_i\sim_l X$, $c\ll X^{2+\epsilon_0}K^{-1-\theta+\epsilon_1}$, and all $k_1,k_2\geq 0$, we have the following bound
        \begin{equation}\label{eq: Bound on derivative d/dr1dr2 beta}
        \frac{\partial^{k_1+k_2}}{\partial r_1^{k_1}\partial r_2^{k_2}}\beta\ll_l \beta X^{(\epsilon_0-1)(k_1+k_2)}.
        \end{equation}
    \end{prop}

    \begin{proof}
        We use the chain rule to write $\frac{\partial^{k_1+k_2}}{\partial r_1^{k_1}\partial r_2^{k_2}}\beta$ as a linear combination of terms of the form  
        \[
        \frac{\partial^m}{\partial x^m}\beta\prod_{i=1}^m\frac{\partial^{a_i+b_i}}{\partial r_1^{a_i}\partial r_2^{b_i}}x,
        \]
        where $1\leq m\leq k_1+k_2$, $a_i, b_i\geq 0$, $a_i+b_i\geq 1$, $\sum_{i=1}^m a_i=k_1$, and $\sum_{i=1}^m b_i=k_2$.

        Now by \Cref{eq: Bound on derivative d/dx beta}, we have that $\frac{\partial^m}{\partial x^m}\beta\ll \frac{\beta}{x^m}$. Furthermore, by \Cref{eq: Bound on derivative d/dr1dr2 x}, we have that $\frac{\partial^{a_i+b_i}}{\partial r_1^{a_i}\partial r_2^{b_i}}x\ll_l \frac{1}{c}X^{2+\epsilon_0-a_i-b_i}$. Thus, we see that 
        \[
        \frac{\partial^{k_1+k_2}}{\partial r_1^{k_1}\partial r_2^{k_2}}\beta\ll_l\frac{\beta}{x^m}\prod_{i=1}^m\frac{1}{c}X^{2+\epsilon_0-a_i-b_i}.
        \]
        Now as $\frac{X^2}{c}\ll_lx\ll_l\frac{X^{2+\epsilon_0}}{c}$, we see that 
        \[
        \frac{\partial^{k_1+k_2}}{\partial r_1^{k_1}\partial r_2^{k_2}}\beta\ll_l\beta\frac{c^m}{X^{2m}}\frac{1}{c^m}X^{2m+m\epsilon_0-k_1-k_2}\ll_l \beta X^{(\epsilon_0-1)(k_1+k_2)}.\qedhere
        \]
    \end{proof}

    We remark that as we are sometimes taking derivatives of powers of $\beta$, we can use the product rule and induction to get the following generalization. 

    \begin{prop}\label{lem: Bound on derivative d/dr1dr2 beta n}
    For $y\sim K$, $r_i\sim_l X$, $c\ll X^{2+\epsilon_0}K^{-1-\theta+\epsilon_1}$, and all $n\geq 1$, we have the following bound
    \begin{equation}\label{eq: Bound on derivative d/dr1dr2 beta n}
    \frac{\partial^{k_1+k_2}}{\partial r_1^{k_1}\partial r_2^{k_2}}\beta^n\ll_l \beta^n X^{(\epsilon_0-1)(k_1+k_2)}.
    \end{equation}
    \end{prop}

    \begin{proof}
    We shall do this by inducting on $n$. We remark that we shall treat the $n=1$ case in \Cref{lem: Bound on derivative d/dr1dr2 beta}. Now, for the inductive step, assume that we have the bound for $n-1$, then we can apply the product rule to see that the partial derivative will be a linear combination of terms of the form
    \[
    \frac{\partial^{a_1+a_2}}{\partial r_1^{a_1}\partial r_2^{a_2}}\beta\frac{\partial^{b_1+b_2}}{\partial r_1^{b_1}\partial r_2^{b_2}}\beta^{n-1},
    \]
    where $a_i,b_i\geq0$, $a_1+b_1=k_1$, and $a_2+b_2=k_2$. Now from \Cref{eq: Bound on derivative d/dr1dr2 beta} and our induction hypothesis, we have that $\frac{\partial^{a_1+a_2}}{\partial r_1^{a_1}\partial r_2^{a_2}}\beta\ll_l \beta X^{(\epsilon_0-1)(a_1+a_2)}$ and $\frac{\partial^{b_1+b_2}}{\partial r_1^{b_1}\partial r_2^{b_2}}\beta^{n-1}\ll_l \beta^{n-1} X^{(\epsilon_0-1)(b_1+b_2)}$. Thus, we see that 
    \[
    \frac{\partial^{k_1+k_2}}{\partial r_1^{k_1}\partial r_2^{k_2}}\beta^n\ll_l \beta X^{(\epsilon_0-1)(a_1+a_2)} \beta^{n-1}X^{(\epsilon_0-1)(b_1+b_2)}\ll_l\beta^n X^{(\epsilon_0-1)(k_1+k_2)}.\qedhere
    \]
    \end{proof}

    Now it just remains to get the bound for the derivative $\frac{\partial^n}{\partial x^n}\beta$. We remark that as we are thinking of $\beta$ as a function of $x$, then we have that $\beta=g\inv\left(\frac{-2\pi y}{x}\right)$ where $g\inv$ is the local inverse. Then to calculate $\frac{\partial^n}{\partial x^n}\beta$, we may just use the chain rule, and then need to get bounds on the derivative of $g\inv$. We state the bound as follows. 

    \begin{prop}\label{lem: Bound on derivative d/dx beta}
    For $y\sim K$, $r_i\sim_l X$, $c\ll X^{2+\epsilon_0}K^{-1-\theta+\epsilon_1}$, and all $n\geq 1$, we have the following bound 
    \begin{equation}\label{eq: Bound on derivative d/dx beta}
    \frac{\partial^{n}}{\partial x^n}\beta\ll_l \frac{\beta}{x^n}.
    \end{equation}
\end{prop}

\begin{proof}
    As we have that 
    \[
    \frac{\partial^n}{\partial x^n}\beta = \frac{\partial^n}{\partial x^n}g\inv\left(\frac{-2\pi y}{x}\right),
    \]
    we can apply the chain rule to write $\frac{\partial^n}{\partial x^n}\beta$ as a linear combination of terms of the form
    \[
    (g\inv)^{(m)}\left(\frac{-2\pi y}{x}\right)\prod_{i=1}^m\frac{\partial^{a_i}}{\partial x^{a_i}}\left(\frac{-2\pi y}{x}\right),
    \]
    where $1\leq m\leq n$, $a_i\geq 1$, and $\sum_{i=1}^ma_i=n$. Now we have by \Cref{eq: Bound for (g-1)^(m)(-2piy/x)} that $(g\inv)^{(m)}\left(\frac{-2\pi y}{x}\right)\ll_l 1$. Furthermore, we have that $\frac{\partial^{a_i}}{\partial x^{a_i}}\left(\frac{-2\pi y}{x}\right)\sim \frac{y}{x^{1+a_i}}$. Thus, we have that 
    \[
    (g\inv)^{(m)}\left(\frac{-2\pi y}{x}\right)\prod_{i=1}^m\frac{\partial^{a_i}}{\partial x^{a_i}}\left(\frac{-2\pi y}{x}\right)\ll_l\prod_{i=1}^m\frac{y}{x^{1+a_i}}\ll_l\frac{y^m}{x^{m+n}}\sim_l\beta^{m}\frac{1}{x^n}.
    \]
    Now as $m\geq 1$, and $\beta\ll_l X^{\epsilon_0}K^{-\theta+\epsilon_1}$, the value is largest when $m=1$, and we conclude the bound.
\end{proof}

Now we just need to calculate $(g\inv)^{(m)}\left(\frac{-2\pi y}{x}\right)$. As we are taking the derivative of an inverse function. We remark that it is a classical exercise by using the chain rule on $g\inv(g(x))=x$ to show that $(g\inv)'(x)=\frac{1}{g'(g\inv(x))}$. However, as we wish to take a higher-derivative, we remark that we can prove a similar formula that holds for higher derivatives. Although as the higher derivatives of $x$ are $0$, there is a strict difference between the first derivative of $g\inv$ and higher derivatives of $g\inv$. We state the general formula for the derivative of an inverse here. 

\begin{prop}\label{lem: Derivative of an inverse function}
    For $n\geq 2$, we have the following formula 
    \begin{equation}\label{eq: Derivative of an inverse function}
    (g\inv)^{(n)}(x)=\frac{-1}{(g'(g\inv(x)))^n}\left(\sum_{\pi\in\Pi^*}(g\inv)^{\vert \pi\vert}(x)\prod_{B\in \pi}g^{\vert B\vert}(g\inv(x))\right)
    \end{equation}
        where we have the sum over $\Pi^*$ denotes the set of all partitions of $n$ excluding the partition of all $1$'s (i.e. $1+1+...+1=n$).
\end{prop}

\begin{proof}
    We use Fa\'a di Bruno's formula on $g\inv(g(x))=x$. Thus, we see that we will have that 
    \[
    \left(g\inv\circ g\right)^{(n)}(x)=\sum_{\pi \in \Pi}(g\inv)^{\vert \pi\vert}(g(x))\prod_{B\in \pi}g^{\vert B\vert}(x)=0.
    \]
    Now we remark that the term corresponding to the partition of all $1$'s is
    \[
    (g\inv)^{(n)}(g(x))(g'(x))^n.
    \]
    Isolating this term on  one side of the equality, we will have that 
    \[
    (g\inv)^{(n)}(g(x))=\frac{-1}{(g'(x))^n}\left(\sum_{\pi\in\Pi^*}(g\inv)^{\vert \pi\vert}(g(x))\prod_{B\in \pi}g^{\vert B\vert}(x)\right).
    \]
    In particular, by plugging in $g\inv(x)$ for $x$, we will have that 
    \[
    (g\inv)^{(n)}(x)=\frac{-1}{(g'(g\inv(x)))^n}\left(\sum_{\pi\in\Pi^*}(g\inv)^{\vert \pi\vert}(x)\prod_{B\in \pi}g^{\vert B\vert}(g\inv(x))\right),
    \]
    which is precisely what we wanted to prove. 
\end{proof}

\begin{prop}\label{lem: Bound for (g-1)^(m)(-2piy/x)}
    For $y\sim K$, $r_i\sim_l X$, $c\ll X^{2+\epsilon_0}K^{-1-\theta+\epsilon_1}$, and $m\geq 1$, we have that 
    \begin{equation}\label{eq: Bound for (g-1)^(m)(-2piy/x)}
    (g\inv)^{(m)}\left(\frac{-2\pi y}{x}\right)\ll_l 1.
    \end{equation}
\end{prop}

\begin{proof}
    We shall do this by inducting upon $m$. For the case of $m=1$, we use the formula for the derivative of the inverse of a function. In particular, we have that 
    \[
    \left(g\inv\right)'\left(\frac{-2\pi y}{x}\right)=\frac{1}{g'\left(g\inv\left(\frac{-2\pi y}{x}\right)\right)}=\frac{1}{g'(\beta)}=\frac{1}{-4\pi^2+8\pi^4\beta^2}.
    \]
    Now as $\beta\sim \frac{y}{x}\ll_l X^{\epsilon_0}K^{-\theta+\epsilon_1}$ which is small, we see that 
    \[
    (g\inv)'\left(\frac{-2\pi y}{x}\right)=\frac{1}{g'(\beta)}\ll_l 1.
    \]
    
    Now for our induction hypothesis, we shall assume that we have $(g\inv)^{(k)}\left(\frac{-2\pi y}{x}\right)\ll_l 1$ for each $1\leq k<m$. Furthermore, we saw that $\frac{1}{g'(\beta)}\sim 1$. Now let us show that the bound will hold for $m$. We remark that as $m\geq 2$, we may use our formula, \Cref{eq: Derivative of an inverse function}, for $(g\inv)^{(m)}$. In particular, we will have that 

    \begin{align*}
    (g\inv)^{(m)}\left(\frac{-2\pi y}{x}\right)&=\frac{-1}{(g'(g\inv(\frac{-2\pi y}{x})))^{m}}\left(\sum_{\sigma\in\Sigma^*}(g\inv)^{(\vert \sigma\vert)}\left(\frac{-2\pi y}{x}\right)\prod_{C\in \sigma}g^{\vert C\vert}\left(g\inv\left(\frac{-2\pi y}{x}\right)\right)\right)\\
    &=\frac{-1}{(g'(\beta))^{m}}\left(\sum_{\sigma\in\Sigma^*}(g\inv)^{(\vert \sigma\vert)}\left(\frac{-2\pi y}{x}\right)\prod_{C\in \sigma}g^{\vert C\vert}(\beta)\right),
    \end{align*}
    where $\Sigma^*$ are all partitions of $m$ that are not all $1$, $C\in \sigma$, are the parts of the partitions of $\sigma$. Thus, combining this all, we see that we can write $(g\inv)^{(m)}\left(\frac{-2\pi y}{x}\right)$ as a linear combination of terms of the form
    \[
    \frac{1}{(g'(\beta))^m}(g\inv)^{(k)}\left(\frac{-2\pi y}{x}\right)\prod_{i=1}^k g^{(a_i)}(\beta),
    \]
    where $1\leq k<m$, $a_i\geq 1$, and $\sum_{i=1}^ka_i=m$. Furthermore, as $g$ is a degree $3$ polynomial, we have that $g^{(l)}(\beta)=0$ if $l\geq 4$, and for other as $\beta\ll_l X^{\epsilon_0}K^{-\theta+\epsilon_1}\ll_l 1$, we will have that each of these $g^{(l)}(\beta)\ll_l 1$. Thus, as each of these terms is $\ll_l 1$, we conclude our result. 
\end{proof}

This concludes all of the bounds we needed to get the bounds for the derivative of $g_c$. We now shift focus to derive bounds for the derivative of $\phi$. In the next section we will apply the principle of stationary phase with this $\phi$ being a part of our phase function, so we must be a bit more precise in our lower order derivatives to assure that the Hessian is non-zero. We remark that as $\phi$ is composed of two parts namely $\alpha$ and $x-\frac{4\pi AR_1R_2}{c}$, we treat these separately.


\begin{prop}\label{lem: bound on derivative d/dr1dr2 alpha}
    For $y\sim K$, $r_i\sim_l X$, $c\ll X^{2+\epsilon_0}K^{-1-\theta+\epsilon_1}$, and all $k_1,k_2\geq 0$, we have the following bound 
    \begin{equation}\label{eq: bound on derivative d/dr1dr2 alpha}
        \frac{\partial^{k_1+k_2}}{\partial r_1^{k_1}\partial r_2^{k_2}}\alpha(r_1,r_2,y)\ll_l \frac{cK^2}{X^2}X^{(\epsilon_0-1)(k_1+k_2)}.
    \end{equation}
\end{prop}

\begin{proof}
    As we have that 
    \[
    \alpha(r_1,r_2,y) = \frac{x}{2\pi}\left(-2\pi^2\beta^2+\frac{2}{3}\pi^4\beta^4\right)+\beta y-\frac{\pi}{4},
    \]
    it will suffice to prove the bound on the two parts. 

    In particular, from using \Cref{eq: Bound on derivative d/dr1dr2 beta}, we will have that 
    \[
    \frac{\partial^{k_1+k_2}}{\partial r_1^{k_1}\partial r_2^{k_2}}\beta y\ll_l \beta X^{(\epsilon_0-1)(k_1+k_2)}y.
    \]
    Now as $\beta\sim \frac{y}{x}$, $y\sim K$, and $\frac{X^2}{c}\ll_lx\ll_l\frac{X^{2+\epsilon_0}}{c}$, we conclude the bound for the $\beta y$ part of $\alpha$. 

    Now for the other part of $\alpha$, we apply the product rule to see that this will be a linear combination of terms of the form 
    \[
    \frac{\partial^{a_1+a_2}}{\partial r_1^{a_1}\partial r_2^{a_2}}x \frac{\partial^{k_1-a_1+k_2-a_2}}{\partial r_1^{k_1-a_1}\partial r_2^{k_2-a_2}}\left(-2\pi^2\beta^2+\frac{8}{3}\pi^4\beta^4\right),
    \]
    where $0\leq a_i\leq k_i$.  Now using the bounds of \Cref{eq: Bound on derivative d/dr1dr2 x} and \Cref{eq: Bound on derivative d/dr1dr2 beta n}, we have that $\frac{\partial^{a_1+a_2}}{\partial r_1^{a_1}\partial r_2^{a_2}}x\ll_l \frac{1}{c}X^{2+\epsilon_0-a_1-a_2}$ and $\frac{\partial^{k_1-a_1+k_2-a_2}}{\partial r_1^{k_1-a_1}\partial r_2^{k_2-a_2}}\left(-2\pi^2\beta^2+\frac{8}{3}\pi^4\beta^4\right)\ll_l \beta^2X^{(\epsilon_0-1)(k_1-a_1+k_2-a_2)}$, we see that 
    \begin{align*}
        \frac{\partial^{k_1+k_2}}{\partial r_1^{k_1}\partial r_2^{k_2}}\left(\frac{x}{2\pi}\left(-2\pi^2\beta^2+\frac{8}{3}\pi^4\beta^4\right)\right)&\ll_l \frac{1}{c}X^{2+\epsilon_0-a_1-a_2}\beta^2X^{(\epsilon_0-1)(k_1-a_1+k_2-a_2)}\\
        &\ll_l\frac{\beta^2}{c}X^{2+\epsilon_0(k_1-a_1+k_2-a_2)-k_1-k_2}\\
        &\ll_l \frac{y^2}{cx^2}X^{2+(\epsilon_0-1)(k_1+k_2)}\\
        &\ll_l\left(\frac{cK^2}{X^4}\right)X^{2+(\epsilon_0-1)(k_1+k_2)}.
    \end{align*}
    Where in the third line we used that $\beta\sim \frac{y}{x}$, then in the fourth line we used $\frac{X^2}{c}\ll_lx$ and $y\sim K$.
\end{proof}

\begin{prop}
    For $y\sim K$, $r_i\sim_l X$, $D\ll X$, $c\ll X^{2+\epsilon_0}K^{-1-\theta+\epsilon_1}$, and all $k_1,k_2\geq 0$, we have the following bound
    \begin{equation}
        \frac{\partial^{k_1+k_2}}{\partial r_1^{k_1}\partial r_2^{k_2}}\left(x-2\pi c\inv\frac{R_1R_2}{2A}\right)\ll_{l,k_1,k_2}c\inv D^2 X^{(\epsilon_0-1)(k_1+k_2)}.
    \end{equation}
    In particular, if in addition $DX^{\epsilon_2}\ll c\ll X^{2+\epsilon_0}K^{-2-\theta+\epsilon}$, we will have that 
    \begin{equation}\label{eq: d/dr1dr2 x-c-1R1R2/2A bound}
        \frac{\partial^{k_1+k_2}}{\partial r_1^{k_1}\partial r_2^{k_2}}\left(x-2\pi c\inv\frac{R_1R_2}{2A}\right)\ll_{l,k_1,k_2}cK^2X^{-2+(\epsilon_0-1)(k_1+k_2)}.
    \end{equation}
\end{prop}

\begin{proof}
    When working out the main part of the exponential term of $\exp\left(\frac{i\Delta}{2c}\right)$, we deduced that 
    \begin{align*}
    x&=\frac{2\pi}{c}\frac{R_1R_2}{2A}\sqrt{\left\vert 1-\frac{D}{R_1^2}\right\vert}\sqrt{\left\vert 1- \frac{D}{R_2^2}\right\vert}\\
    &=2\pi c\inv \frac{R_1R_2}{2A}\sum_{n_1,n_2\geq 0}\binom{1/2}{n_1}\binom{1/2}{n_2}\left(\frac{-D}{R_1^2}\right)^{n_1}\left(\frac{-D}{R_2^2}\right)^{n_2},
    \end{align*}
    where $R_i=2Ar_i+B$. We see that $R_i\sim Ar_i\sim_l AX$, and as we have assumed that $A\ll X^{\epsilon_0}$. Furthermore as $D\ll X$ and $R_i\sim_l AX$, we will have that $\frac{D}{R_i^2}\ll_l X\inv$ will be small when we can Taylor expand. Now as $R_i$ is just a linear transformation applied to $r_i$, we will have that $2A\frac{\partial}{\partial R_i}=\frac{\partial}{\partial r_i}$. Now from our expansion it is clear that 
    \begin{equation}
    x=2\pi c\inv\left(\frac{R_1R_2}{2A}+O_l(D)\right)
    \end{equation}
    \begin{equation}\label{eq: Formula for x_ri}
    x_{r_i}=2A x_{R_i}=2A\left(2\pi c\inv\left(\frac{R_{3-i}}{2A}+O_l(D^2 X\inv)\right)\right)=2A\left(2\pi c\inv\left(r_{3-i}+\frac{B}{2A}+O_l(D^2 X\inv)\right)\right)
    \end{equation}
    \begin{equation}\label{eq: Formula for x_riri}
    x_{r_ir_i} = 4A^2x_{R_iR_i}=4A^2\left(2\pi c\inv O_l(D^2 X^{-2})\right)
    \end{equation}
    \begin{equation}\label{eq: Formula for x_r1r2}
    x_{r_1r_2}=4A^2 x_{R_1R_2}=4A^2\left(2\pi c\inv\left(1+O_l(D^2 X^{-2})\right)\right)
    \end{equation}
    \begin{equation}
    \frac{\partial^{k_1+k_2}}{\partial r_1^{k_1}\partial r_2^{k_2}}x=(2A)^{k_1+k_2}\frac{\partial^{k_1+k_2}}{\partial R_1^{k_1}\partial R_2^{k_2}}x=(2A)^{k_1+k_2}\left(2\pi c\inv O_{l,k_1,k_2}(D^2X^{-k_1-k_2})\right),
    \end{equation}
    where in the last equality, we require that $k_1+k_2\geq 3$. Now as we will have that the function $\frac{2\pi}{c}\left(\frac{R_1R_2}{2A}\right)$ will have the exact same lower order terms in the derivatives up to order $2$, and as $A\ll X^{\epsilon_0}$, we will have that the $(2A)^{k_1+k_2}\ll X^{\epsilon_0(k_1+k_2)}$. This allows us to conclude that for any $k_1$ and $k_2$ we have that 
    \[
    \frac{\partial^{k_1+k_2}}{\partial r_1^{k_1}\partial r_2^{k_2}}\left(x-\frac{2\pi }{c}\left(\frac{R_1R_2}{2A}\right)\right)\ll_{l,k_1,k_2} c\inv D^2X^{(\epsilon_0-1)(k_1+k_2)}.
    \]
    We remark, in the case where additionally $DX^{\epsilon_2}\ll c\ll X^2K^{-1-\theta+\epsilon_1}$, we will have that 
    \[
    c\inv D^2X^{(\epsilon_0-1)(k_1+k_2)}\ll cX^{(\epsilon_0-1)(k_1+k_2)-2\epsilon_2}\ll cK^2X^{-2+(\epsilon_0-1)(k_1+k_2)},
    \]
    allowing us to conclude the lemma. 
\end{proof}

\begin{prop}\label{lem: Bound on derivative d/dr1dr2 phi}
    For $y\sim K$, $r_i\sim_l X$, $D\ll X$ $DX^{\epsilon_2}\ll c\ll X^{2+\epsilon_0}K^{-1-\theta+\epsilon_1}$, and $k_1,k_2\geq 0$, we have the bound
    \begin{equation}\label{eq: Bound on derivative d/dr1dr2 phi}
        \frac{\partial^{k_1+k_2}}{\partial r_1^{k_2}\partial r_2^{k_2}}\phi(r_1,r_2,y)\ll_{l,k_1,k_2}cK^2X^{-2+(\epsilon_0-1)(k_1+k_2)}.
    \end{equation}
    Moreover, when $k_1+k_2=1$ or $2$, we have 
    \begin{equation}\label{eq: Bound for First Dervative of phi}
        \phi_{r_i}\sim_l \frac{cK^2}{X^{3-\epsilon_0}},
    \end{equation}
    \begin{equation}\label{eq: Bounds for Second Derivative and Hessians}
        \vert \phi_{r_ir_j}\vert\sim_l\frac{cK^2}{X^{4-2\epsilon_0}}, \text{ and }\vert\phi_{r_1r_1}\phi_{r_2r_2}-\phi_{r_1r_2}^2 \vert\gg_l \frac{c^2K^4}{X^{8-4\epsilon_0}}.
    \end{equation}
\end{prop}

\begin{proof}
    We remark that the statement of \Cref{eq: Bound on derivative d/dr1dr2 phi} follows from \Cref{eq: bound on derivative d/dr1dr2 alpha} and \Cref{eq: d/dr1dr2 x-c-1R1R2/2A bound}. Thus, we just need to prove the bounds on the derivative for orders $1$ and $2$. 

    If we think of $\alpha$ as a function of $x$, and we consider that $\beta\sim \frac{y}{x}$, we will have that there is some constant say $c_1$ such that 
    \[
    \alpha_x = -c_1\frac{y^2}{x^2}+O(K^4x^{-4})
    \]
    \[
    \alpha_{xx} = 2c_1\frac{y^2}{x^3}+O(K^4x^{-5}).
    \]

    Thus, applying the chain rule for $\alpha$ when taking derivatives with respect to $r_i$ and \Cref{eq: Formula for x_ri}, we see that
    \begin{align*}
    \phi_{r_i} &= \alpha_{x}x_{r_i}+\left(x-\frac{2\pi}{c}\frac{R_1R_2}{2A}\right)_{r_i}\\
    &=\left(-c_1\frac{y^2}{x^2}+O(K^4x^{-4})\right)x_{r_i} +O_l(c\inv D^2X^{\epsilon_0-1})\\
    &=-c_1\frac{y^2}{x^2}x_{r_i}+O_l(c^3 X^{-7+\epsilon_0}K^4)+O_l(c\inv D^2X^{\epsilon_0-1}),
    \end{align*}
    where in the last line we applied the bound for $x_{r_i}\ll_lc\inv X^{1+\epsilon_0}$, \Cref{eq: Bound on derivative d/dr1dr2 x}. We remark that the leading term is $\sim_l cK^2X^{-3+\epsilon_0}$ and the error term is $o(cK^2X^{-3+\epsilon_0})$ as $DX^{\epsilon_2}\ll c\ll X^{2+\epsilon_0}K^{-1-\theta+\epsilon_1}$. This gives us \Cref{eq: Bound for First Dervative of phi}.

    Now we see taking another derivative and using \Cref{eq: Formula for x_riri} and \Cref{eq: Formula for x_r1r2}, we have that
    \begin{align*}
        \phi_{r_ir_j}& = \alpha_{xx}x_{r_i}x_{r_j} + \alpha_x x_{r_ir_j} +\left(x-\frac{2\pi}{c}\frac{R_1R_2}{2A}\right)_{r_ir_j}\\
        &= 2c_1\frac{y^2}{x^3}x_{r_i}x_{r_j}+O(K^4x^{-5}x_{r_i}x_{r_j})-c_1\frac{y^2}{x^2}x_{r_ir_j}+O(K^4x^{-4}x_{r_ir_j})+O_l(c\inv D^2X^{(\epsilon_0-1)2})\\
        &=c_1\frac{y^2}{x^3}\left(2x_{r_i}x_{r_j}-x_{r_ir_j}x\right)+O_l(c^{-2}K^4x^{-5}X^{2+2\epsilon_0}+c\inv K^4x^{-4}X^{\epsilon_0}+c\inv D^2X^{(\epsilon_0-1)2})\\
        &=c_1\frac{y^2}{x^3}\left(\frac{2\pi}{c}\right)^2\left((2-\delta_{ij})R_{3-i}R_{3-j}+O_l(D^2)\right)+O_l(K^4c^3X^{-8+2\epsilon_0}+cX^{(\epsilon_0-\epsilon_2-1)2})\\
        &=c_1\frac{y^2}{x^3}\left(\frac{2\pi}{c}\right)^2\left((2-\delta_{ij})R_{3-i}R_{3-j}\right)+O_l(K^2D^2X^{-6}c+K^4c^3X^{-8+2\epsilon_0}+cX^{(\epsilon_0-\epsilon_2-1)2}).
    \end{align*}
    We see that the error term is $o(cK^2X^{-4+2\epsilon_0})$. Thus, we will have that $\vert \phi_{r_ir_j}\vert\sim_l \frac{cK^2}{X^{4-2\epsilon_0}}$. As for the final statement, we observe that the main term of the Hessian will be 
    \[
    \vert \phi_{r_1r_1}\phi_{r_2r_2}-\phi_{r_1r_2}^2\vert = \left(c_1\frac{y^2}{x^3}\left(\frac{2\pi}{c}\right)^2\right)^2\left((R_1)^2(R_2)^2-(2R_1R_2)^2\right)\sim_l \left(\frac{c^3K^2}{X^6c^2}\right)^2X^{4+4\epsilon_0}\sim_l\frac{c^2K^4}{X^{8-4\epsilon_0}}.
    \]

    From this we conclude the proposition. 
\end{proof}

\begin{prop}\label{lem: Bound on derivative d/dr1dr2 e(i phi)}
    For $y\sim K$, $r_i\sim_l X$, $D\ll X$, $DX^{\epsilon_2}\ll c\ll X^{2+\epsilon_0}K^{-1-\theta+\epsilon_1}$, and $k_1,k_2\geq 0$, we have the following bound 
    \begin{equation}\label{eq: Bound on derivative d/dr1dr2 e(i phi)}
        \frac{\partial^{k_1+k_2}}{\partial r_1^{k_2}\partial r_2^{k_2}}\exp(i\phi(r_1,r_2,y))\ll_{l,k_1,k_2} \left(\frac{cK^2}{X^{3-2\epsilon_0}}\right)^{k_1+k_2}.
    \end{equation}
\end{prop}

\begin{proof}
    We shall simply apply the chain rule to get that $\frac{\partial^{k_1+k_2}}{\partial r_1^{k_2}\partial r_2^{k_2}}\exp(i\phi(r_1,r_2))$ will be a linear combination of terms of the form
    \[
    \exp(i\phi(r_1,r_2,y))\prod_{i=1}^m\frac{\partial^{a_i+b_i}}{\partial r_1^{a_i}\partial r_2^{b_i}}\phi(r_1,r_1,y),
    \]
    where $1\leq m \leq k_1+k_2$, $a_i,b_i\geq0$, $a_i+b_i\geq 1$, $\sum_{i=1}^m a_i = k_1$, and $\sum_{i=1}^m b_i =k_2$. Now by applying \Cref{eq: Bound on derivative d/dr1dr2 phi}, we have that $\frac{\partial^{a_i+b_i}}{\partial r_1^{a_i}\partial r_2^{b_i}}\phi(r_1,r_1,y)\ll cK^2X^{-2+(\epsilon_0-1)(a_i+b_i)}$, so we see that 
    \[
    \prod_{i=1}^m\frac{\partial^{a_i+b_i}}{\partial r_1^{a_i}\partial r_2^{b_i}}\phi(r_1,r_1,y)\ll X^{(\epsilon_0-1)(k_1+k_2)}\left(\frac{cK^2}{X^2}\right)^m.
    \]
    Now the key observation is that as $c\ll X^{2+\epsilon_0}K^{-1-\theta+\epsilon}$ and $1\leq m\leq k_1+k_2$, we will have that this term is largest when $m=k_1+k_2$. This allows us to conclude the proposition. 
\end{proof}

\subsection{Proof of Proposition 2.25}\label{subsec: Proof of Lemma}

The goal of this section is to give a proof of \Cref{lem: Bound for Sum of fc}. We break this up into several parts. In particular, we begin by applying the Poisson summation formula, then we show that the non-zero terms have lots of decay. Then we bound the 

We shall apply Poisson summation to write
\[
\sum_{(r_1,r_2)\in\zz^2}f_c(r_1,r_2)e_c(-ur_1-vr_2)=\sum_{(j,k)\in\zz^2}B(j,k),
\]
where 
\begin{align*}
B(j,k)&=\iint f_c(r_1,r_2)e_c(-ur_1-vr_2)e(jr_1+kr_2)dr_1d_r\\
&=\iint f_c(r_1,r_2)e\left(\left(j-\frac{u}{c}\right)r_1+\left(k-\frac{v}{c}\right)r_2\right)dr_1dr_2.
\end{align*}

Now in the above, we may assume that $\vert u\vert,\vert v\vert \leq c/2$. The main contribution will come from the terms where $(j,k)\neq (0,0)$. Thus, we shall show that these terms are negligible. We state this as the following lemma. 

\begin{prop}\label{lem: Poisson Summation Non-Zero Terms Negligible}
    For $y\sim K$, $r_i\sim_l X$, $D\ll X$ and $DX^{\epsilon_2}\ll c\ll X^{2+\epsilon_0}K^{-1-\theta+\epsilon_1}$, we have that 
    \begin{equation*}\label{eq: Poisson Summation Non-Zero Terms Negligible}
        \sum_{\substack{j,k\\(j,k)\neq(0,0)}}B(j,k)\ll_{l,\theta}K^{-20}\norm{\psi}^2_{W^{m_0,\infty}},
    \end{equation*}
    where $m_0=m_0=\left\lceil\frac{100}{3\theta-1}\right\rceil$.
\end{prop}

\begin{proof}
    Without loss of generality, we shall give an upper bound on $B(j,k)$ in the case that $\max\{j,k\}=j>0$. By integration by parts $m$-times, where we always differentiate $f_c$ and integrate the exponential term, we have that
    \[
    \vert B(j,k)\vert = \left\vert\left(2\pi\left(j-\frac{u}{c}\right)\right)^{-m}\iint\left(\frac{\partial^m}{\partial r_1^m}f_c(y,r_1,r_2)\right)e\left(\left(j-\frac{u}{c}\right)r_1+\left(k-\frac{v}{c}\right)r_2\right)dr_1dr_2\right\vert. 
    \]
    Now as $f_c$ has support where $r_i\sim_l X$, we have that the integral will be supported in a region of size $X^2$. Furthermore, from \Cref{lem: Bound on derivative d/dr1dr2 fc}, we will have that 
    \[
    \iint\left(\frac{\partial^m}{\partial r_1^m}f_c(y,r_1,r_2)\right)e\left(\left(j-\frac{u}{c}\right)r_1+\left(k-\frac{v}{c}\right)r_2\right)dr_1dr_2\ll_{l,m} c^{1/2}X (K^{\frac{1-3\theta}{4}})^{m}\norm{\psi}^2_{W^{m,\infty}}.
    \]
    Now because for the bound $\vert u\vert,\vert v\vert\leq c/2$, we will have that 
    \[
    B(j,k)\ll_{l,m}(2\vert j\vert-1)^{-m}c^{1/2}XK^{-m(3\theta-1)/4}\norm{\psi}^2_{W^{m,\infty}}.
    \]
    Thus, for a general $(j,k)\neq(0,0)$, we will have that 
    \[
    B(j,k)\ll_{l,m}(2\max\{\vert j\vert,\vert k\vert\}-1)^{-m}c^{1/2}XK^{-m(3\theta-1)/4}\norm{\psi}^2_{W^{m,\infty}}.
    \]
    Now if we take $m=m_0=\lceil\frac{100}{3\theta-1}\rceil$, so that $m_0(3\theta-1)/4>25$. we will have that 
    \[
    \sum_{\substack{j,k\\(j,k)\neq(0,0)}}B(j,k)\ll_{l,\theta} c^{1/2}XK^{-25}\norm{\psi}^2_{W^{m_0,\infty}}\ll_{l,\theta}K^{-20}\norm{\psi}^2_{W^{m_0,\infty}},
    \]
    allowing us to conclude the lemma as $c\ll X^{2+\epsilon_0}K^{-1-\theta+\epsilon_1}$. 
\end{proof}

We now need to invoke the principle of stationary phase to bound the $B(0,0)$ contribution. We have that 
\[
B(0,0)=\iint f_c(r_1,r_2)e_c(-ur_1-vr_2)dr_1dr_2 = \iint g_c(r_1,r_2,y)\exp\left(i\phi(r_1,r_2,y)-\frac{2\pi i}{c}(ur_1+vr_2)\right)  dr_1dr_2.
\]
Thus, we see that our phase function is $\phi(r_1,r_2,y)-\frac{2\pi }{c}(ur_1+vr_2)$. Now from \Cref{lem: Bound on derivative d/dr1dr2 phi}, we will have that the points of stationary phase will occur when $\phi_{r_1}=\frac{2\pi}{c}u$ and $\phi_{r_2}=\frac{2\pi}{c}v$. In particular the points of stationary phase occur when 
\[
u\sim_l\frac{c^2K^2}{X^{3-\epsilon_0}}\text{ and }v\sim_l\frac{c^2K^2}{X^{3-\epsilon_0}}.
\]
Thus, we will show that for $u$ and $v$ not in this range, we may apply integration by parts to show that $B(0,0)$ is small. We formalize this as the following lemma. 

\begin{prop}\label{lem: B00 small range}
    For  $y\sim K$, $r_i\sim_l X$, $D\ll X$, and $DX^{\epsilon_2}\ll c\ll X^{2+\epsilon_0}K^{-1-\theta+\epsilon_1}$, we have that there is a constant $\eta$ depending on $l$ such that if either $u$ or $v$ is not in the range 
    \[
    \left[\eta\inv\frac{c^2K^2}{2\pi X^{3-\epsilon_0}},\eta\frac{c^2K^2}{2\pi X^{3-\epsilon_0}}\right],
    \]
    then we have that 
    \[
    B(0,0)\ll_{l,\epsilon,\theta}K^{-40}\norm{\psi}^2_{W^{n_0,\infty}},
    \]
    where $n_0=\lceil\frac{100}{\epsilon_2-\epsilon_0}\rceil$.
\end{prop}

\begin{proof}
    From \Cref{eq: Bound for First Dervative of phi} we will have that there is some constant $\eta_0$ depending on $l$ such that 
    \[
    \eta_0\inv\frac{cK^2}{X^{3-\epsilon_0}}<\phi_{r_i}<\eta_0\frac{cK^2}{X^{3-\epsilon_0}}.
    \]
    Now as $\supp\psi\subset(R/l,Rl)$, if we let $\eta=2\eta_0$, we can assume without loss of generality that $u$ is not in the range 
    \[
    \left[\eta\inv\frac{c^2K^2}{2\pi X^{3-\epsilon_0}},\eta\frac{c^2K^2}{2\pi X^{3-\epsilon_0}}\right].
    \]
    Now we shall do integration by parts. In particular, we define $I_n$ as the $n$-th time we perform integration by parts. Thus, we let $I_0=g_c$, and then inductively define
    \[
    I_n(r_1,r_2,y)=i\frac{\partial}{\partial r_1}\frac{I_{n-1}(r_1,r_2,y)}{\phi(r_1,r_2,y)_{r_1}-\frac{2\pi u}{c}}
    \]
    for $n\geq 1$, so that 
    \[
    B(0,0)=\iint I_n(r_1,r_2,y)e^{i\phi(r_1,r_2,y)-\frac{2\pi i}{c}(ur_1+vr_2)}dr_1dr_2.
    \]
    By the chain rule and induction, we see that $I_n$ is a linear combination of 
    \[
    \frac{\partial^{a_0}}{\partial r_1^{a_0}}g_c(r_1,r_2,y)\prod_{j=1}^n \frac{\partial^{a_j}}{\partial r_1^{a_j}}\frac{1}{\phi(r_1,r_2,y)_{r_1}-\frac{2\pi u}{c}},
    \]
    where each $0\leq a_j$ and $\sum_{j=0}^na_j=n$. Then by the chain rule, we see that for each $j$, we have that we have that 
    \[
    \frac{\partial^{a_j}}{\partial r_1^{a_j}}\frac{1}{\phi(r_1,r_2,y)_{r_1}-\frac{2\pi u}{c}}
    \]
    is a linear combination of 
    \[
    \frac{1}{(\phi(r_1,r_2,y)_{r_1}-\frac{2\pi u}{c})^{b_j+1}}\prod_{k=1}^{b_j}    \frac{\partial^{b_{jk}}}{\partial r_1^{b_{jk}}}\left(\phi(r_1,r_2,y)_{r_1}-\frac{2\pi u}{c}\right),
    \]
    where $1\leq b_j\leq a_j$, $b_{jk}\geq 1$, and  $\sum_{k=1}^{b_j}b_{jk} = a_j$. Now by \Cref{eq: Bound for First Dervative of phi}, we have that 
    \[
    \frac{1}{\phi(r_1,r_2,y )_{r_1}-\frac{2\pi u}{c} }\ll_l \frac{X^{3-\epsilon_0}}{cK^2},
    \]
    and by \Cref{eq: Bound on derivative d/dr1dr2 phi} we have that
    \[
    \frac{\partial^{b_{jk}}}{\partial r_1^{b_{jk}}}\left(\phi(r_1,r_2,y)_{r_1}-\frac{2\pi u}{c}\right) = \frac{\partial^{b_{jk}+1}}{\partial r_1^{b_{jk}+1}}\phi(r_1,r_2,y)\ll_{l}cK^2X^{-2+(\epsilon_0-1)(b_{jk}+1)}.
    \]
    Thus, putting this all together, we have that 
    \begin{align*}
    \frac{1}{(\phi(r_1,r_2)_{r_1}-\frac{2\pi u}{c})^{b_j+1}}\prod_{k=1}^{b_j}\frac{\partial^{b_{jk}}}{\partial r_1^{b_{jk}}}\left(\phi(r_1,r_2)_{r_1}-\frac{2\pi u}{c}\right)&\ll_{l,n}\left(\frac{X^{3-\epsilon_0}}{cK^2}\right)^{b_j+1}\prod_{k=1}^{b_j}cK^2X^{-2+(\epsilon_0-1)(b_{jk}+1)}\\
    &\ll_{l,n} c\inv K^{-2} X^{3-\epsilon_0+(\epsilon_0-1)a_j}.
    \end{align*}

    Thus, using \Cref{eq: Bound on derivative d/dr1dr2 gcr1r2}, we see that 
    \begin{align*}
        \frac{\partial^{a_0}}{\partial r_1^{a_0}}g_c(r_1,r_2)\prod_{j=1}^n\frac{\partial^{a_j}}{\partial r_1^{a_j}}\frac{1}{\phi(r_1,r_2)_{r_1}-\frac{2\pi u}{c}}&\ll_{l,n} c^{1/2}X^{-1+(2\epsilon_0-1)a_0}\norm{\psi}^2_{W^{n,\infty}}\prod_{j=1}^nc\inv K^{-2}X^{3-\epsilon_0+(\epsilon_0-1)a_j}\\
        &=c^{1/2}X\inv(c\inv X^{2+\epsilon_0}K^{-2})^n\norm{\psi}^2_{W^{n,\infty}}\\
        &\ll_{l,n} c^{1/2}X\inv (c\inv X^{\epsilon_0})^{n}\norm{\psi}^2_{W^{n,\infty}}.
    \end{align*}
    Thus, as the region where the function we are integrating over for $B(0,0)$ has size on the order of $X^2$, we have that
    \[
    B(0,0)\ll_{l,n} c^{1/2}X(c\inv X^{\epsilon_0})^{n}\norm{\psi}^2_{W^{n,\infty}}.
    \]
    Now as we have that $X^{\epsilon_2}\ll DX^{\epsilon_2}\ll c$, if we pick $n = n_0 = \lceil\frac{100}{\epsilon_2-\epsilon_0}\rceil$, so now the $n$ dependency for the constant becomes and $\epsilon_0$ and $\epsilon_2$ dependency which is thus, an $\epsilon$ and $\theta$ dependency. And we see that 
    \[
    B(0,0)\ll_{l,\epsilon,\theta} c^{1/2}X (X^{\epsilon_0-\epsilon_2})^{n}\norm{\psi}^2_{W^{n_0,\infty}} \ll (X^{2+\epsilon_0}K^{-1-\theta+\epsilon_1})^{1/2}XX^{-100}\norm{\psi}^2_{W^{n_0,\infty}}.
    \]
    Now we conclude the proposition, by observing that $1\ll DX^{\epsilon_2}\ll c\ll X^{2+\epsilon_0}K^{-1-\theta+\epsilon_1}$, so we will have that $K\ll K^{1+\theta-\epsilon_1}X^{-\epsilon_0} \ll X^2$. Thus, $X^{-100}\ll K^{-50}$, and then we can just trivially bound say $(X^{2+\epsilon_0}K^{-1-\theta+\epsilon_1})^{1/2}X\ll K^{10}$ to conclude the proposition. 
    \end{proof}

Now it remains to prove the last bound on $B(0,0)$ in the case where $u,v\sim_l\frac{c^2K^2}{X^3}$. To deal with this case, we will shall apply integration by parts once in the variable $r_1$ and once in the variable $r_2$ where we have differentiated the $f_c$ function and integrated the exponential term, that we will have that 
    \begin{align}
    \begin{split}\label{eq: B00 Integration By Parts Bound}
            B(0,0)&=\iint f_c(r_1,r_2)e_c(-ur_1-vr_2)dr_1dr_2\\
            &=\iint\int_0^{r_1}\int_0^{r_2}e^{i\phi(t_1,t_2)}e_c(-ut_1-vt_2)dt_1dt_2g_c(r_1,r_2)_{r_1r_2}dr_1dr_2\\
            &\ll\sup_{\substack{X/l\ll r_1\ll lX\\ X/l\ll r_2\ll lX}}\left\vert \int_0^{r_1}\int_0^{r_2}e^{i\phi(t_1,t_2)-\frac{2\pi i}{c}(ut_1+vt_2)}dt_1dt_2 \right\vert\iint\left\vert g_c(r_1,r_2)_{r_1,r_2}\right\vert dr_1dr_2.
    \end{split}
    \end{align}

Now by \Cref{lem: Bound on derivative d/dr1dr2 phi} the phase function of $\phi(t_1,t_2)-\frac{2\pi }{c}(ut_1+vt_2)$ satisfies the estimate in the following lemma, see \cite[Lemma $\delta$]{titchmarsh1934epstein} , with $\lambda = cX^{-4+2\epsilon_0}K^2$. In particular, this will give us that 
\begin{equation}\label{eq: Application of Titchmarsh}
\int_0^{r_1}\int_0^{r_2}e^{i\phi(t_1,t_2)-\frac{2\pi i}{c}(ut_1+vt_2)}dt_1dt_2\ll_lc\inv X^{4-2\epsilon_0}K^{-2}\log K
\end{equation}
holds uniformly for $X/l\ll r_1,r_2\ll lX$.

\begin{lemma}\label{lem: Titchmarsh}
    Let $f(t_1,t_2)$ be a real and algebraic function defined in a rectangle $D=[a,b]\times[c,d]\subset\rr^2$. Assume throughout $D$ that 
    \[
    \vert f_{t_it_i}\vert\sim \lambda\text{ for $i=1,2$, } \vert f_{t_1t_2}\vert\ll\lambda,\text{ and } \left\vert\frac{\partial(f_{t_1},f_{t_2})}{\partial(t_1,t_2)}\right\vert\gg\lambda^2;
    \]
    then
    \[
    \iint_De^{if(t_1,t_2)}dt_1dt_2\ll \frac{1+\vert\log(b-a)\vert+\vert\log(d-c)\vert+\vert\log\lambda\vert}{\lambda}.
    \]
\end{lemma}

\begin{prop}\label{lem: Bound for B00}
    For $y\sim K$, $r_i\sim X$,$D\ll X$, $DX^{\epsilon_2}\ll c\ll X^2K^{-1-\theta+\epsilon_1}$, and $\epsilon_2>0$, we have that 
    \[
    B(0,0)\ll_l c^{-1/2}X^{3+2\epsilon_0}K^{-2} \log K\norm{\psi}^2_{W^{2,\infty}}.
    \]
\end{prop}
\begin{proof}
    As $r_i\sim X$, and using the bound for $g_c(r_1,r_2)_{r_1,r_2}$, \Cref{eq: Bound on derivative d/dr1dr2 gcr1r2}, we will have that 
    \[
    \iint\left\vert g_c(r_1,r_2)_{r_1,r_2}\right\vert dr_1dr_2\ll c^{1/2}X^{-1+4\epsilon_0}\norm{\psi}^2_{W^{2,\infty}}.
    \]
    Combining this with the bound of \Cref{eq: B00 Integration By Parts Bound} and \Cref{eq: Application of Titchmarsh}, we conclude the lemma. 
\end{proof}

Now we have that \Cref{lem: Bound for Sum of fc} follows immediately from \Cref{lem: Poisson Summation Non-Zero Terms Negligible}, \Cref{lem: B00 small range}, and \Cref{lem: Bound for B00}.

\section{Improvement of Quantum Variance}\label{sec: Application}

The goal of this section is to work out the details of how \Cref{cor: Quantum Ergodicty} follows from \Cref{thm: Main Theorem}. The key idea is that for a general test function $g\in C_0^\infty(R(1)\backslash\hh)$, we can spectrally expand $g$ in terms of Maass forms $\Psi_i$. Then for a non-constant Maass form $\Psi$, we can use the work of Nelson \cite{nelson2022quadratic}; however, we must make Nelson's argument more quantitative for our application as we must keep track upon the dependence of the Laplace eigenvalue $\lambda_\Psi$. Nelson stated his main theorem as an implication of two conjectures, using the notation of this paper, we have the following two conjectures. 

\begin{conj}\label{conj: Nelson Quadratic Sums}
    Assume that $f\in H_k(R(1))$, then for $\psi\in C_0^\infty(0,\infty)$ and $q\in \qq[x]$ an irreducible quadratic integer valued polynomial, then we have that 
    \begin{equation*}
        \lim_{k\rightarrow\infty}\frac{\sum_{n\geq 1}\lambda_f(\vert q(n)\vert) \psi\left(\frac{n}{k}\right)}{kL(1,\Sym^2(f))}=0.
    \end{equation*}
\end{conj}
We might remark, that Nelson states his conjecture while also varying the polynomials $q$ and the test functions $\psi$ over the weight $k$. However, as most applications including ours, we have a fixed polynomial and test function we do not include these extra details here. Furthermore, we remark that in this paper we shall just denote $\Sym^2(f)$, but as we are working with $\GL(2)$ objects, this is the same as the adjoint representation $\ad(f)$, so in some of the references they refer to the adjoint representation, but for simplicity we will always write $\Sym^2(f)$.

\begin{conj}\label{conj: Nelson Equidistribution}
    For each bounded continuous function $\Psi: R(1)\backslash\hh\rightarrow\cc$, we have 
    \begin{equation*}\label{eq: Weak* Limit for QUE }
        \mu_f(\Psi) =\frac{\int_{R(1)\backslash\hh}\vert f(z)\vert^2\Psi(z)y^kd\mu(z)}{\int_{R(1)\backslash\hh}\vert f(z)\vert^2 y^kd\mu(z)}\xrightarrow{k\rightarrow\infty}\frac{\int_{R(1)\backslash\hh}\Psi(z)d\mu(z)}{\int_{R(1)\backslash\hh}d\mu(z)}= \Tilde{\nu}(\Psi)
    \end{equation*}
    where $d\mu(z)=\frac{dxdy}{y^2}$ is the invariant measure on the surface $R(1)\backslash\hh$.
\end{conj}
We note that this is precisely the statement of \Cref{conj: Mass Equidistribution} for the surface $R(1)\backslash\hh$.

\begin{thm}[Theorem A \cite{nelson2022quadratic}]\label{thm: Nelson's Implication}
     Conjecture \ref{conj: Nelson Quadratic Sums} implies  \Cref{conj: Nelson Equidistribution}
\end{thm}

By following Nelson's line of reasoning in his deduction of the two conjectures, but being a bit more careful with keeping our bounds quantitative, we show the following theorem. 

\begin{thm}\label{thm: Maass Form QUE}
    Let $\{\Psi\}_{i\geq 0}$ be an orthonormal basis of Maass form on $R(1)\backslash \hh$, so that $\Delta\Psi_i=\lambda_i\Psi_i$ where $0=\lambda_0<\lambda_1\leq \lambda_2\leq...$. Then for any $\epsilon>0$ and $1/3<\theta<1$, we have that there is some value $T>0$ which is independent of $i$ such that 
    \[
    \sum_{\substack{\vert k-K\vert<K^\theta}}\sum_{f\in H_k(R(1))}\left\vert \mu_f(\Psi_i)-\Tilde{\nu}(\Psi_i)\right\vert^2\ll_{\Acal,\epsilon, \theta}\lambda_i^T K^{\theta+\epsilon}.
    \]
    for example one may take $T=\max\left\{19+\frac{100}{3\theta-1},19+\frac{2000}{\epsilon}\right\}$.
\end{thm}

We remark that as we have normalized everything to be a probability measure. The theorem is trivially for the constant function, $\Psi_0$, so we may assume that we are working with a non-constant Maass form $\Psi_i$ with $i>0$. Furthermore, as we are working with an orthonormal basis, we will have that
\[
\Tilde{\nu}(\Psi_i)= \frac{1}{\Volume(R(1)\backslash\hh)}\int_{R(1)\backslash\hh}\Psi_i\frac{dxdy}{y^2}=0
\]
As we can view the integral as the inner product $\langle\Psi_0,\Psi_i\rangle=0$ by orthogonality. Thus, it suffices to prove the bound of 
\begin{equation}\label{eq: Quantum Variance Goal Dropping nu for Maass Form}
\sum_{\substack{\vert k-K\vert<K^\theta}}\sum_{f\in H_k(R(1))}\vert \mu_f(\Psi)\vert^2\ll_{\epsilon,\theta,\Acal} \lambda_i^A K^{\theta+\epsilon}
\end{equation}
for $\Psi$ a non-constant Maass form. 

We can now see how to deduce \Cref{cor: Quantum Ergodicty} from \Cref{thm: Maass Form QUE}, to do this we shall use a spectral decomposition.

\begin{proof}[Proof of \Cref{cor: Quantum Ergodicty}]
    Let $g\in C_0^\infty(R(1)\backslash\hh)$. We shall spectrally expand $g$ as an orthonormal basis of Maass forms,
    \[
    g = \sum_{i\geq 0} \langle g,\Psi_i\rangle\Psi_i.
    \]
    Now we shall simply apply Cauchy-Schwarz, and \Cref{thm: Maass Form QUE} as follows. 
    \begin{align*}
        \sum_{\substack{\vert k-K\vert<K^\theta}}\sum_{f\in H_k(R(1))}\vert \mu_f(g)-\Tilde{\nu}(g)\vert^2 & = \sum_{\substack{\vert k-K\vert<K^\theta}}\sum_{f\in H_k(R(1))}\left\vert\sum_{i\geq 0} \langle g,\Psi_i\rangle\left(\mu_f(\Psi_i)-\Tilde{\nu}(\Psi_i)\right)\right\vert^2\\
        &=\sum_{i,j\geq 0}\langle g,\Psi_i\rangle\overline{\langle g,\Psi_j\rangle}\\
        &\times \sum_{\substack{\vert k-K\vert<K^\theta}}\sum_{f\in H_k(R(1))}\left(\mu_f(\Psi_i)-\Tilde{\nu}(\Psi_i)\right)\overline{\left(\mu_f(\Psi_j)-\Tilde{\nu}(\Psi_j)\right)}\\
        &\leq \sum_{i,j\geq 0}\langle g,\Psi_i\rangle\overline{\langle g,\Psi_j\rangle}\left(\sum_{\substack{\vert k-K\vert<K^\theta}}\sum_{f\in H_k(R(1))}\vert\left(\mu_f(\Psi_i)-\Tilde{\nu}(\Psi_i)\right)\vert^2\right)^{1/2}\\
        &\times \left(\sum_{\substack{\vert k-K\vert<K^\theta}}\sum_{f\in H_k(R(1))}\vert\left(\mu_f(\Psi_j)-\Tilde{\nu}(\Psi_j)\right)\vert^2\right)^{1/2}\\
        &\ll_{\epsilon,\Acal} \sum_{i,j\geq 1}\langle g,\Psi_i\rangle\overline{\langle g,\Psi_j\rangle}(\lambda_i^{T}K^{\theta+\epsilon})^{1/2}(\lambda_j^TK^{\theta+\epsilon})^{1/2}\\
        &\ll_{\epsilon,\Acal}K^{\theta+\epsilon} \sum_{i,j\geq 0}\langle g,\Psi_i\rangle\overline{\langle g,\Psi_j\rangle}\lambda_i^{T/2}\lambda_j^{T/2}.
    \end{align*}
    Thus, we see that the result will follow if we can show that 
    \[
    \sum_{i\geq 1}\langle g,\Psi_i\rangle\lambda_i^T\ll_{\Acal} \norm{g}_{W^{T+2,2}(R(1)\backslash\hh)}.
    \]
    To see this, we observe using the self-adjointness of $\Delta$, the Cauchy-Schwarz inequality, and the Weyl law, we see that:
    \begin{align*}
    \sum_{i\geq 1}\langle g,\Psi_i\rangle\lambda_i^T&\leq \sum_{i\geq 1}\vert \langle \Delta^{T+2}g, \Psi_i\rangle \lambda_i^{-2}\vert\ll_{\Acal} \norm{\Delta^{T+2}g}_{L^2}\sum_{i\geq 1}i^{-2}\ll_{\Acal} \norm{g}_{W^{T+2,2}(R(1)\backslash\hh)}.\qedhere
    \end{align*}
\end{proof}

We now begin the proof of \Cref{thm: Maass Form QUE}. 

Now as $f$ is a modular form and $\Psi$ is a Maass form on $R(1)\backslash\hh$, we shall denote, respectively, $\pi^{\Acal}$ and $\sigma^{\Acal}$ the corresponding automorphic representations on $\textrm{P}\Acal^\times(\aa)$. Similarly, we shall let respectively $\pi$ and $\sigma$ be the Jacquet-Langlands lifts to cuspidal automorphic representations of $\PGL_2(\aa)$. Furthermore, we shall denote $\Psi'$ to be a Hecke-Maass form on $\Gamma_0(d_{\Acal})\backslash\hh$ that is a Jacquet-Langlands lift of $\Psi$. We will also take the Fourier expansion of $\Psi'$ to be given by
\begin{equation*}\label{eq: Psi' Fourier Expansion}
    \Psi'(x+iy) = \sum_{0\neq \ell}\frac{\rho(\vert \ell\vert)}{\vert\ell\vert^{1/2}}W_{\Psi}(\ell y)e(\ell x),
\end{equation*}
where the $W_{\Psi}$ is the corresponding Whittaker function. Furthermore, by the Jaquet-Langlands correspondence, we will have that $\lambda_{\Psi} = \lambda_{\Psi'}$. We shall also, let $\Lambda(s,...) = L_\infty(s,...)L(s,...)$ denote a completed $L$-function, where $L_\infty(s,...)$ denotes the Archimedean factor, and $L(s,...)$ the finite part of the $L$-function. 

The first step is to apply the Watson-triple product formula \cite{watson2002rankin} which we state as the following theorem.
\begin{thm}\label{thm: Watson Formula}
    We have that there is some constant $c\geq 0$ that is independent of $\Psi$ such that 
    \begin{equation*}\label{eq: Watson Formula}
        \vert\mu_f(\Psi)\vert^2 = c\frac{\Lambda(1/2,\pi\times\pi\times\sigma)}{\Lambda(1,\Sym^2\pi)^2\Lambda(1,\Sym^2\sigma)}.
    \end{equation*} 
\end{thm}

Now we apply the factorization \cite[Equation (2.3)]{nelson2022quadratic} to get that 
\begin{equation}\label{eq: Nelson 2.3}
\Lambda(s,\pi\times\pi\times\sigma)=\Lambda(s,\Sym^2 \pi\times \sigma)\Lambda(s,\sigma),
\end{equation}
and we see that this will vanish if $\Lambda(1/2,\sigma)=0$, and in such a case, we have that the root number $\epsilon(\sigma)=-1$. Thus, we may assume that $\epsilon(\sigma)=1$. Now we shall apply a base change using \cite[Proposition 3.1]{nelson2022quadratic}, we can find a non-trivial quadratic character $\chi_D$ of $\aa^\times/\qq^\times$ such that 
\begin{enumerate}
    \item $L(1/2,\sigma\otimes\chi_D)\neq 0$,
    \item The Archimedean local component $(\chi_D)_\infty$ is trivial, and
    \item For each prime $p$ that ramifies in $\Acal$, the local component $(\chi_D)_p$ is the nontrivial unramified quadratic character of $\qq_p^\times$.
\end{enumerate}
Furthermore, by \Cref{prop: appendix MDS}, we may choose $D\ll_{\Acal} \lambda_{\Psi}^5$.

We shall now define for the fixed $D$ above, the base change of the representation $\pi$ of $\PGL_2(\aa)$ to $\qq(\sqrt{D})$ the representation $\pi_D$, and we shall let $f_D$ be a Hilbert modular form on $\PGL_2\aa_{\qq(\sqrt{D})}$. Furthermore, we will denote $\lambda$ be the Hecke eigenvalues of $f$ and $\lambda_D$ the Hecke-eigenvalues of $f_D$. Furthermore, we remark that by \cite[Lemma 4.3]{nelson2022quadratic} if we let $x = \frac{n+\ell\sqrt{D}}{2}\in \Ocal_D$ and if $d_1$ (resp. $d_2$) denote the largest split (resp. inert) natural numbers dividing $x$, then 
\begin{equation}\label{eq: Nelson 4.3 eigenvalues multiply}
\lambda_D(x) = \lambda(d_1)\lambda_D(d_2)\lambda\left(\frac{\vert n^2-D\ell^2\vert}{4d_1d_2^2}\right).
\end{equation}
Now by \cite[Equation (5.3)]{nelson2022quadratic}, we have that 
\begin{equation}\label{eq: Nelson 5.3}
\Lambda(s,\asai(\pi_D)\times\sigma) = \Lambda(s,\Sym^2\pi\times\sigma)\Lambda(s,\sigma\otimes\chi_D).
\end{equation}
Combining \Cref{eq: Nelson 2.3} and \Cref{eq: Nelson 5.3}, we have that 
\begin{equation*}\label{eq: Nelson Combine 2.3 and 5.3}
\Lambda(s,\pi\times\pi\times\sigma)=\frac{\Lambda(s,\sigma)}{\Lambda(s,\sigma\otimes\chi_D)}\Lambda(s,\asai(\pi_D)\times\sigma).
\end{equation*}
Plugging this into Watson's formula gives us that 
\[
\vert \mu_f(\Psi)\vert^2 = c \frac{\Lambda(1/2,\sigma,)\Lambda(1/2,\asai(\pi_D)\times \sigma)}{\Lambda(1,\Sym^2 \pi)^2\Lambda(1,\Sym^2\sigma)\Lambda(1/2,\sigma\otimes\chi_D)}.
\]

Now, we have from \cite[Proposition 5.2]{nelson2022quadratic} that there is a constant $c_{\Acal}>0$, depending on $\Acal$, such that
\begin{equation*}\label{eq: Nelson 5.4}
\frac{\vert\int \res(f_D)\Psi'\vert^2}{\norm{f_D}^2\norm{\Psi'}^2} = c_{\Acal}\frac{\Lambda(1/2,\asai(\pi_D)\times\sigma)}{\Lambda(1,\Sym^2(\pi_D))}.
\end{equation*}
We remark that \cite{nelson2022quadratic} says that the constant is in a finite set depending upon both $\Acal$ and $D$; however, the dependence of the constant on $D$ will not affect the order of magnitude of the constant, we see this in \Cref{prop: Nelson 5.1 no D}. 

Now by the factorization\cite[Equation (4.6)]{nelson2022quadratic}, we have
\begin{equation*}\label{eq: Nelson eq 4.6}
L_\infty(s,\Sym^2\pi_D)=L_\infty(s,\Sym^2\pi)^2,
\end{equation*}
along with \cite[Lemma 4.4]{nelson2022quadratic} which tells us that there is a constant $c_{\Acal,D}$, depending only upon $\Acal$ and $D$, such that 
\begin{equation*}\label{eq: Nelson 4.13 lemma 4.4}
\norm{f_D}^2=c_{\Acal,D}L(1,\Sym^2(\pi_D)).
\end{equation*}
Putting this all together, we see that
\[
\vert \mu_f(\Psi)\vert^2 = \frac{c}{c_{\Acal}c_{\Acal,D}}\frac{\Lambda(1/2,\sigma)}{\norm{\Psi'}^2\Lambda(1,\Sym^2\sigma)\Lambda(1/2,\sigma\otimes\chi_D)}\left(\frac{\vert \int \res(f_D)\Psi'\vert}{L(1,\Sym^2\pi)}\right)^2.
\]
Now we remark that by \Cref{prop: Appendix all together bound} we will have that
\[
\frac{c}{c_{\Acal}c_{\Acal,D}}\frac{\Lambda(1/2,\sigma)}{\norm{\Psi'}^2\Lambda(1,\Sym^2\sigma)\Lambda(1/2,\sigma\otimes\chi_D)}\ll_{\Acal}\lambda_{\Psi}^4.
\]
Thus, in particular, we will have that 
\[
\vert \mu_f(\Psi)\vert^2 \ll_{\Acal} \lambda_{\Psi}^4\left(\frac{\vert \int \res(f_D)\Psi'\vert}{L(\ad\pi,1)}\right)^2.
\]
Thus, on the left hand side of \Cref{eq: Quantum Variance Goal Dropping nu for Maass Form}; it suffices to bound
\[
\sum_{\substack{\vert k-K\vert<K^\theta}}\sum_{f\in H_k(R(1))} \left(\frac{\vert \int \res(f_D)\Psi'\vert}{L(\ad\pi,1)}\right)^2\ll_{\epsilon,\theta} \lambda_i^{15+T}K^{\theta+\epsilon}.
\]
Now there is a small technicality which we must address. In particular, the above sum over the basis of $H_k(R(1))$, we observe that this is over a basis of Hecke forms for $R(1)$. However, in order to apply \Cref{thm: Main Theorem}, we need to sum over a basis of Hecke form $H_k(\Gamma_0(N))$. This is where we shall apply the Jacquet Langland's correspondence. Since we will have that $\pi$ is the corresponding automorphic representation for a form on $H_k(\Gamma_0(d_{\Acal}))$ where $d_{\Acal}$ is the reduced discriminant of $\Acal$. Thus, when we sum over a basis of $H_k(R(1))$ that is equivalent to summing over a basis for $H_k(\Gamma_0(d_{\Acal}))$. Furthermore, we will have that the Hecke data for $\pi^{\Acal}$ appearing in \Cref{eq: Nelson 4.3 eigenvalues multiply} will agree with the Hecke data for a corresponding form of $\pi$.


Now\cite[Equation (8.1)]{nelson2022quadratic} gives us a value for the integral. In particular, it tells us that for a choice of $0<\epsilon_0<\frac{\epsilon}{100}$, we will have that 
\begin{equation}\label{eq: Nelson 8.1}
    \frac{\int\res(f_D)\Psi'}{L(\ad\pi,1)}=\sum_{0\neq\vert \ell\vert<k^{\epsilon_0}}\hashsum_{a\in \zz/2\ell}\lambda(d_1(a))\lambda_D(d_2(a))S(\ell,a) + O(k^{-1+\epsilon_0}),
\end{equation}
where we have that $d_1(a)$ and $d_2(a)$ are defined to be the largest split and inert natural numbers dividing $a$, $\hashsum$ denotes that we are only summing over $a$'s that satisfy certain congruence conditions to assure that we only sum integer values for $\frac{\vert n^2-D\ell^2\vert}{4d_1d_2^2}$, and $S(\ell, a)$ is given by
\begin{equation*}\label{eq: Defn of S(ell,a)}
    S(\ell,a) = \frac{1}{kL(\ad\pi,1)}\sum_{n\equiv a(2\ell)}\lambda\left(\frac{\vert n^2-D\ell^2\vert}{4d_1d_2^2}\right)\psi_{\ell}\left(\frac{n}{k}\right),
\end{equation*}
where we have that $\psi_\ell\in C^\infty(0,\infty)$.

Furthermore, we should remark a little bit on the structure of $\psi_\ell$ (we note that this is the corresponding function which Nelson calls $f_\ell$, but we opt to call it $\psi$ to agree with the notation of our main theorem). In particular, we begin by noting that this function is not compactly supported which is a small technicality that we will deal with later on. However, this is another place where we need to be a bit more quantitative. In particular, we will have that this function $\psi_\ell$ depends upon the Maass form $\Psi$. In particular, we have that 
\begin{equation*}
    \psi_\ell(u):=\left(\frac{\rho(\vert \ell\vert)}{\vert\ell\vert^{1/2}}W_{\Psi}(\ell u)\right)\frac{2}{u}h\left(\frac{\sqrt{D}}{2\pi u}\right),
\end{equation*}
where $h\in C^\infty(0,\infty)$ is given by the rapidly decaying inverse Mellin transform
\begin{equation*}
    h(y) :=\frac{1}{2\pi i }\int_{(1+\delta)}(2s-1)2\xi(2s)\left(\sum_{d\mid d_{\Acal}}d^s\right)y^s ds.
\end{equation*}
In particular, we see from its definition that $\psi_\ell$ has a dependence upon the Maass form $\Psi$ through the inclusion of the Whittaker function $W_\Psi(\ell u)$. Now Nelson remarks that the function $\psi_\ell$ will satisfy the following bound on the Sobolev norm $S_N(\psi_\ell)\ll \vert \ell\vert^{-N}$ where we have that $S_N$ is the Sobolev norm for the invariant differential on $\rr^\times$, see \Cref{rem: Sobolev Norms for differnt Groups}. However, we would also like to see how the Sobolev norm depends upon the Maass form $\Psi$. In particular, we will have that $S_N(\psi_\ell)\ll t_\Psi^{N}\vert\ell\vert^{-N}$. This is easily seen as the Whittaker function will be a Bessel function $K_{it_\Psi}$, and by the Bessel differential equation, when taking a second derivative we will pick up at most a contribution of $t_\Psi^2$; thus, in taking $N$ derivatives, we will have at most a contribution of $t_\Psi^N$ coming from the spectral parameter.   

\begin{remark}\label{rem: Sobolev Norms for differnt Groups}
    We make a short remark here on the Sobolev norms. In \cite{nelson2022quadratic}, he defines the Sobolev norm as $S_N(\psi)$ with respect to the differential on the group $\rr^\times$. In particular, he has that if $\psi'(y)$ is the usual derivative, then $D\psi(y)=y\psi'(y)$ will be the invariant derivative, and then the Soblev norm is defined for $N\ge 0$ as 
    \begin{equation*}\label{eq: Nelson Sobolev Def}
        S_N(\psi) = \max_{j\leq N}\sup_{y\in\rr^\times}(y+1/y)^N \vert D^j\psi(y)\vert.
    \end{equation*}
    While this definition is not the one we gave in \Cref{eq: Sobolev Norm Def}, it will turn out for $\psi\in C_0^\infty(0,\infty)$, that $\norm{\psi}_{W^{N,\infty}}\sim S_N(\psi)$ where the implied constant depends only upon the support of $\psi$. To see this, we observe that by a simply repeated application of the chain rule, we have that there are some constants say $c_j$ such that 
    \[
    (D^n\psi)(y) = \sum_{j=1}^n c_jy^j\psi^{(j)}(y),
    \]
    then by applying the binomial theorem and distributing we will have that there are some constants $c_{i,j}$ such that 
    \[
    \left(\frac{1}{y}+y\right)^n(D^n\psi)(y) = \sum_{j=1}^n\sum_{i=-n}^nc_{i,j}y^{j+i}\psi^{(j)}(y).
    \]
    From this, we see that the two notions of the Sobolev norm are asymptotically equivalent up to a possible factor of $y^{i+j}$, but this will be a constant depending only upon the support of the function $\psi$.
\end{remark}

Now we see that squaring \Cref{eq: Nelson 8.1} yields:
\begin{align*}
\left(\frac{\int\res(f_D)\Psi'}{L(\ad\pi,1)}\right)^2 &\ll \left(\sum_{\substack{0\neq \vert \ell\vert<k^{\epsilon_0}}}\hashsum_{\substack{a\in\zz/\ell}}\lambda(d_1(a))\lambda(d_2(a))S(\ell,a)\right)^2\\
&+k^{-1+\epsilon_0}\sum_{0\neq \vert \ell\vert<k^{\epsilon_0}}\hashsum_{a\in\zz/\ell}\lambda(d_1(a))\lambda(d_2(a))S(\ell,a) + k^{-2+2\epsilon_0}.
\end{align*}
In particular, let us define $V_{main}$, $V_{cross}$, and $V_{error}$ to be the following three terms which if we bound will provide us a proof of \Cref{thm: Maass Form QUE}.
\begin{equation*}\label{eq: def of Vmain}
    V_{main}=\sum_{\substack{\vert k-K\vert<K^\theta}}\sum_{f\in H_k(\Gamma_0(d_{\Acal}))}\left(\sum_{\substack{0\neq \vert \ell\vert<k^{\epsilon_0}}}\hashsum_{\substack{a\in\zz/\ell}}\lambda(d_1(a))\lambda(d_2(a))S(\ell,a)\right)^2
\end{equation*}
\begin{equation*}\label{eq: def of Vcross}
    V_{cross}=\sum_{\substack{\vert k-K\vert<K^\theta}}\sum_{f\in H_k(\Gamma_0(d_{\Acal}))}k^{-1+\epsilon_0}\sum_{0\neq \vert \ell\vert<k^{\epsilon_0}}\hashsum_{a\in\zz/\ell}\lambda(d_1(a))\lambda(d_2(a))S(\ell,a)
\end{equation*}
\begin{equation*}\label{eq: def of Verror}
    V_{error}=\sum_{\substack{\vert k-K\vert<K^\theta}}\sum_{f\in H_k(\Gamma_0(d_{\Acal}))} k^{-2+2\epsilon_0}
\end{equation*}

Let us begin by bounding $V_{error}$, we can bound this trivially as the length over the sum over $H_k(\Gamma_0(d_{\Acal}))$ is $k$, and the length of the sum over $k$ is $K^\theta$, and as $k\sim K$, we see that 
\[
V_{error}\ll \sum_{\substack{\vert k-K\vert<K^\theta}} k^{-1+2\epsilon_0}\ll K^{-1+\theta+2\epsilon_0}.
\]

Now let us look bound $V_{cross}$, we remark that we may trivially bound the value of $S(\ell,a)$ by 
\[
S(\ell,a)\ll_{\epsilon_0} k^{\epsilon_0}. 
\]
Using, the above, along with the bound of $\lambda(d_1(a))\lambda(d_2(a))\leq 10 \vert \ell\vert^{10}$ given in \cite[p. 24]{nelson2022quadratic}. We will have that 
\begin{align*}
    V_{cross}&\ll_{\epsilon_0} \sum_{\substack{\vert k-K\vert<K^\theta}}\sum_{f\in H_k(\Gamma_0(d_{\Acal}))}k^{-1+2\epsilon_0}\sum_{0\neq \vert \ell\vert<k^{\epsilon_0}}\hashsum_{a\in\zz/\ell}\lambda(d_1(a))\lambda(d_2(a))\\
    &\ll_{\epsilon_0} \sum_{\substack{\vert k-K\vert<K^\theta}}\sum_{f\in H_k(\Gamma_0(d_{\Acal}))}k^{-1+2\epsilon_0}\sum_{0\neq \vert \ell\vert<k^{\epsilon_0}}\vert \ell\vert^{11}\\
     &\ll_{\epsilon_0} \sum_{\substack{\vert k-K\vert<K^\theta}}\sum_{f\in H_k(\Gamma_0(d_{\Acal}))}k^{-1+15\epsilon_0}\\
    &\ll_{\epsilon_0} \sum_{\substack{\vert k-K\vert<K^\theta}}k^{15\epsilon_0}\\
    &\ll_{\epsilon} K^{\theta+\epsilon},
\end{align*}
where in the last line we used that $\epsilon_0<\frac{\epsilon}{25}$.

Finally, it remains to bound $V_{main}$. We begin by applying Cauchy-Schwarz, so that we will have  
\begin{align*}
\left(\sum_{0\neq \vert\ell\vert<k^{\epsilon_0}}\hashsum_{a\in\zz/\ell} \lambda(d_1(a))\lambda(d_2(a))S(\ell,a)\right)^2&\leq\left(\sum_{0\neq \vert\ell\vert<k^{\epsilon_0}}\hashsum_{a\in\zz/\ell}S(\ell,a)^2\right)\\
&\times \left(\sum_{0\neq \vert\ell\vert<k^{\epsilon_0}}\hashsum_{a\in\zz/\ell}\lambda(d_1(a))^2\lambda(d_2(a))^2\right).
\end{align*}
Again using the bound that $\lambda(d_1(a))\lambda(d_2(a))\leq 10\vert \ell\vert^{10}$, we have that 
\[
\sum_{0\neq \vert\ell\vert<k^{\epsilon_0}}\hashsum_{a\in\zz/\ell}\lambda(d_1(a))^2\lambda(d_2(a))^2\ll\sum_{0\neq \vert\ell\vert<k^{\epsilon_0}}\hashsum_{a\in\zz/\ell} \vert \ell\vert^{20}\ll \sum_{0\neq \vert\ell\vert<k^{\epsilon_0}}\vert \ell\vert^{21}\ll k^{22\epsilon_0}\ll K^{22\epsilon_0}.
\]
Thus, we see that 
\[
V_{main}\ll K^{22\epsilon_0}\sum_{\substack{\vert k-K\vert<K^\theta}}\sum_{f\in H_k(\Gamma_0(d_{\Acal}))}\sum_{0\neq \vert\ell\vert<k^{\epsilon_0}}\hashsum_{a\in\zz/\ell}\left(\frac{1}{kL(\ad\pi,1)}\sum_{n\equiv a(2\ell)}\lambda\left(\frac{\vert n^2-D\ell^2\vert}{4d_1d_2^2}\right)\psi_{\ell}\left(\frac{n}{k}\right)\right)^2.
\]
Now we are almost ready to apply \Cref{thm: Main Theorem}, but we see that we have two $\frac{1}{kL(\ad\pi,1)}$, so we can pull out one of these copies which will be of size $K^{-1+\epsilon_0}$, so we now wish to bound the following
\[
V_{main}\ll_{\epsilon} K^{23\epsilon_0}\frac{1}{K}\sum_{\substack{\vert k-K\vert<K^\theta}}\sum_{f\in H_k(\Gamma_0(d_{\Acal}))}\sum_{0\neq \vert\ell\vert<k^{\epsilon_0}}\hashsum_{a\in\zz/\ell}\frac{1}{kL(\ad\pi,1)}\left(\sum_{n\equiv a(2\ell)}\lambda\left(\frac{\vert n^2-D\ell^2\vert}{4d_1d_2^2}\right)\psi_{\ell}\left(\frac{n}{k}\right)\right)^2.
\]
Now let us interchange the sums to get that 
\[
V_{main}\ll K^{23\epsilon_0}\sum_{0\neq \vert\ell\vert\ll K^{\epsilon_0}}\hashsum_{a\in\zz/\ell}\frac{1}{K}\sum_{\substack{\vert k-K\vert<K^\theta}}\sum_{f\in H_k(\Gamma_0(d_{\Acal}))}\frac{1}{kL(\ad\pi,1)}\left(\sum_{n\equiv a(2\ell)}\lambda\left(\frac{\vert n^2-D\ell^2\vert}{4d_1d_2^2}\right)\psi_{\ell}\left(\frac{n}{k}\right)\right)^2.
\]

Now we begin be writing the $n\equiv a\Mod{2\ell}$ as $n=2\ell n'+a$, then summing over all such $n'$, and letting $q(n) = \frac{\vert4\ell^2n^2+4a\ell n -D\ell^2\vert}{4d_1d_2^2}$ which has discriminant $\left(\frac{a\ell}{d_1d_2^2}\right)^2+\frac{D\ell^4}{(d_1d_2^2)^2}\ll DK^{4\epsilon_0}$, we apply \Cref{thm: Main Theorem} to get

\begin{align*}
    V_{main}&\ll K^{23\epsilon_0}\sum_{0\neq \vert\ell\vert\ll K^{\epsilon_0}}\hashsum_{a\in\zz/\ell}\left(\frac{1}{K}\sum_{\substack{\vert k-K\vert<K^\theta}}\sum_{f\in H_k(\Gamma_0(d_{\Acal}))}\frac{1}{kL(\ad\pi,1)}\left(\sum_{n\geq 1}\lambda\left(q(n)\right)\psi_{\ell}\left(\frac{2\ell n+a}{k}\right)\right)^2\right).
\end{align*}

Now there is a small issue as we will have that the $\psi_\ell\in C^\infty(0,\infty)$, in particular, it is not compactly supported; however, it decays very rapidly. In particular, using \cite[p.22]{nelson2022quadratic} we will have that for all $N$,  
\[
\psi_\ell\left(\frac{2\ell n+a}{k}\right)\ll \vert\ell\vert^{-N}\min\left(\frac{2\ell n+a}{k},\frac{k}{2\ell n+a}\right)^{N}.
\]
Thus, we will have that if $n\not\sim K$, then there is enough decay of $\psi_\ell$ to make those terms negligible. 
Thus, we can introduce another smooth compactly support function $\psi\left(\frac{n}{K}\right)$ (independent of $\ell$ and $\Psi_i$) such that $\norm{\psi}^2_{W^{T,\infty}}\ll_{\epsilon,\theta}1$ where say $\supp\psi\subset (1/l,l)$ where we may say that this $l$ is fixed and could depend upon say $\epsilon$ and $\theta$. Then we will have that $\psi\cdot\psi_\ell\in C_0^\infty(0,\infty)$, and by \Cref{rem: Sobolev Norms for differnt Groups}, we will have that $\norm{\psi\cdot\psi_\ell}^2_{W^{T,\infty}}\ll_{\epsilon,\theta} \lambda_\Psi^T\vert \ell\vert^{-T}$. Thus, by applying \Cref{thm: Main Theorem}, we will have that 
\begin{align*}
    V_{main}&\ll_{\epsilon,\theta} K^{23\epsilon_0}\sum_{0\neq \vert\ell\vert\ll K^{\epsilon_0}}\hashsum_{a\in\zz/\ell}\left((1+(DK^{4\epsilon_0})^{3/2})K^{\theta+\epsilon/2}\norm{\psi_\ell}^2_{W^{T,\infty}}\right)\\
    &\ll_{\Acal,\epsilon,\theta}\lambda_{\Psi}^{15}K^{\theta+\epsilon/2+ 29\epsilon_0}\sum_{0\neq \vert\ell\vert\ll K^{\epsilon_0}}\hashsum_{a\in\zz/\ell}\norm{\psi_\ell}^2_{W^{T,\infty}}\\
    &\ll_{\Acal,\epsilon,\theta}\lambda_{\Psi}^{15+T}K^{\theta+\epsilon/2+ 31\epsilon_0}.
\end{align*}
Where as we take an exponent of $\epsilon/2$ in the application of our theorem, we have that $T$ can be taken for example to be $T = \max\left\{\frac{100}{3\theta-1},\frac{2000}{\epsilon}\right\}$, and we have used the bound that $D\ll_{\Acal}\lambda_{\Psi}^5$. Then, as we have $\epsilon_0<\frac{\epsilon}{100}$, we conclude \Cref{thm: Maass Form QUE}, and hence the proof of \Cref{cor: Quantum Ergodicty}.

Now that we have finished the proof of the quantum variance, we shall now fill in the proof of \Cref{cor: Quantitative QUE}. In particular, we shall let 
\[
\Fcal (K) = \{ f\in H_k(R(1)\backslash\hh) : \vert k-K\vert<K^\theta\},
\]
in particular, we have that $\Fcal(K)\sim K^{1+\theta}$, and this is the family that we are averaging over in our results. Now we shall consider the exceptional set of elements where we have that effective mass equidistribution fails that is for $\delta>0$ we define
\[
S(K,\delta) = \left\{f\in \Fcal(K) : \vert \mu_f(g) -\Tilde{\nu}(g)\vert \gg_{\Acal,\epsilon,\theta} \frac{\norm{g}_{W^{T,2}}}{K^{\delta}}\right\}.
\]
Now the game to be played is that we want to find the largest $\delta$ such that $\vert S(K,\delta)\vert\ll \vert \Fcal(K)\vert$ as $K\rightarrow\infty$ so that we will have that almost all forms will satisfy effective mass equidistribution, and as the larger $\delta$ will imply a faster convergence rate. In particular, we see that we can bound the size of $S(K,\delta)$ as follows: 
\begin{align*}
    \vert S(K,\delta)\vert &= \sum_{f\in \Fcal(K)}\mathds{1}_{S(K,\delta)}(f)\\
    &\ll_{\epsilon,\Acal} \frac{K^{2\delta}}{\norm{g}_{W^{T,2}}^2}\sum_{f\in\Fcal(K)}\vert \mu_f(g)-\Tilde{\nu}(g)\vert^2\\
    &\ll_{\epsilon,\Acal}K^{\theta+\epsilon+2\delta}.
\end{align*}
Now if we have that $\frac{K^{\theta+\epsilon+2\delta}}{K^{1+\theta}}\rightarrow0$ as $K\rightarrow\infty$, then we will have that the exceptional set $S(K,\delta)$ will have measure zero, so then almost all forms will satisfy effective mass equidistribution with that given $\delta$, and we see that we can take $\delta<\frac{1-\epsilon}{2}$, which gives us \Cref{cor: Quantitative QUE}.



\appendix

\section{Some Bounds}\label{appendix: Bounds on Special Values}
The goal of this appendix is to state the dependence of the bounds that we have used in \Cref{sec: Application} on the Laplace eigenvalue of $\Psi'$. In particular, the majority of these propositions are just refinements of those in Nelson's paper where we keep track of the dependence upon $\lambda_{\Psi'}$. Thus, the majority of the statements just go through the references and proofs of Nelson, but tracks these dependencies. In addition to getting these bounds, we also answer a question about doing the quadratic base change that we did in \Cref{sec: Application}. In particular, Nelson needed a non-vanishing result for $L(1/2,\sigma\otimes\chi_D)$, and he had used the existence of a non-vanishing twist; however, we quantified a lower bound for central value of the twisted $L$-function as well as an upper bound on how large we can take $D$ to be to have such a twist to get such a lower bound in \Cref{prop: appendix MDS}. We opt to only sketch this proof as it is not the purpose of this paper, and the result is implicit. In particular, one could read off the bound for $D$ from say \cite[Theorem 1.14]{hoffstein2010first}. Although, we remark that their bound for the choice of $D$ is better, we opt to just pick a large number that is nice enough to avoid having too many $\epsilon$'s.


\begin{prop}\label{prop: appendix MDS}
    Let $\Psi'$ be a Maass form on $\Gamma_0(d_{\Acal})$ with spectral parameter $t_{\Psi'}$ and Laplace eigenvalue $\lambda_{\Psi'}$, then we can find some value $D$ such that $D\ll_{\Acal}\lambda_{\Psi'}^5$ and so that we have that 
    \begin{equation*}
    \lambda_{\Psi'}^{-\epsilon}\ll_{\Acal,\epsilon}L\left(\frac{1}{2},\Psi'\otimes \chi_D\right).
    \end{equation*}
    In particular, letting $\sigma$ be the corresponding automorphic representation, we have that with the choice of $D$ above we shall have that 
    \begin{equation*}
        \lambda_{\Psi'}^{-\epsilon} \ll_\epsilon \Lambda(1/2,\sigma\otimes\chi_D).
    \end{equation*}
\end{prop}

\begin{proof}[Proof Sketch]
    The existence of such a $D$ will come from studying a twisted first moment problem. In particular, if we can show that there is some constant $c$ such that 
    \begin{equation*}\label{eq: twisted first moment}
        \sum_{\substack{d \leq X}} L\left(\frac{1}{2},\Psi'\otimes \chi_d\right)= cL(1,\Sym^2 \Psi')X + \text{Error},
    \end{equation*}
    then if we choose $X$ large enough so that the Error term is negligible and the main term dominates, then there must be a choice of $d\leq X$ such that $L(1,\Sym^2\Psi')\ll L(1/2,\Psi'\otimes \chi_d)$, then from \cite{hoffstein1994coefficients} we will have that $\lambda_{\Psi'}^\epsilon\ll_\epsilon L(1,\Sym^2 \Psi')$, and we can conclude the result. 

    For simplicity, we shall make the assumption that we are working on level $1$, as the introduction of a higher level will just add more technicalities and will affect the bound by some constant that is only dependent upon the level. Now we shall use the theory of multiple Dirichlet series to tackle this problem. In particular, we start with the multiple Dirichlet series given by \cite[Equation (3.26)]{chinta2006multiple} where we shall always write $d=d_0d_1^2$ where $d_0$ is square-free and $m=m_0m_1^2$, then
    \begin{equation}\label{eq: Def of Zsw}
        Z(s,w) = \sum_{\substack{d\geq 1\\ (d,2)=1}} \frac{L(s,\Psi'\otimes \chi_{d_0})P(s,d_0d_1^2)}{(d_0d_1^2)^w} = \sum_{\substack{m\geq 1\\ (m,2)=1}} \frac{L(w,\chi_{m_0})Q(w,m_0m_1^2)}{(m_0m_1^2)^s}.
    \end{equation}
    Where in the above we have that $P(s,d_0d_1^2)$ and $Q(w,m_0m_1^2)$ are the correction polynomials which satisfy functional equations 
    \begin{equation*}\label{eq: fun eq. for correction poly}
        P(s,d_0d_1^2) = d_1^{2-4s}P(1-s,d_0d_1^2) \text{,  } Q(w,m_0m_1^2)=m_1^{1-2w}Q(1-w,m_0m_1^2).
    \end{equation*}
    Now we have that classically these $L$-functions satisfy the functional equations of the following forms:
    \begin{equation*}
        L(s,\Psi'\otimes \chi_{d_0}) \xrightarrow{\alpha}(d_0t_j)^{1-2s}L(1-s,\Psi'\otimes \chi_{d_0})
    \end{equation*}
    \begin{equation*}
        L(w,\chi_{m_0})\xrightarrow{\beta}m_0^{\frac{1}{2}-w}L(1-w,\chi_{m_0}).
    \end{equation*}
    Thus, when we apply these functional equations along with those for the correction polynomials, we see that $Z(s,w)$ will satisfy the functional equations of the form:
    \begin{equation*}
        Z(s,w) \xrightarrow{\alpha}\sum \frac{t_j^{1-2s}L(1-s,\Psi'\otimes \chi_{d_0})P(1-s,d_0d_1^2)}{(d_0d_1^2)^{w+2s-1}} = t_j^{1-2s}Z(1-s,w+2s-1)
    \end{equation*}
    \begin{equation*}
        Z(s,w)\xrightarrow{\beta}\sum \frac{L(1-w,\chi_{m_0})Q(1-w,m_0m_1^2)}{(m_0m_1^2)^{s+w-\frac{1}{2}}} = Z\left(s+w-\frac{1}{2},1-w\right).
    \end{equation*}
    Now we see that 
    \begin{align*}
        Z(s,w)&\xrightarrow{\alpha}t_j^{1-2s}Z(1-s,w+2s-1)
        \xrightarrow{\beta}t_j^{1-2s}Z\left(w+s-\frac{1}{2},2-2s-w\right)\\
        &\xrightarrow{\alpha}t_j^{3-4s-2w}Z\left(\frac{3}{2}-w-s,w\right)
        \xrightarrow{\beta}t_j^{3-4s-2w}Z\left(1-s,1-w\right).
    \end{align*}
    Now we shall compute an integral in two ways first we have that 
    \begin{align*}
        \frac{1}{2\pi i}\int_{(2)}Z(s,w)X^w\Gamma(w)dw & = \frac{1}{2\pi i}\int_{(2)}\sum_{\substack{d\geq 1\\ (d,2)=1}}  \frac{L(s,\Psi'\otimes \chi_{d_0})P(s,d_0d_1^2)}{(d_0d_1^2)^w}X^w\Gamma(w)dw\\
        &=\sum_{\substack{d\geq 1\\ (d,2)=1}} L(s,\Psi'\otimes \chi_{d_0})P(s,d_0d_1^2)e^{-d/X}.
    \end{align*}
    We see that this is just a smoothed first moment (weighted by the correction polynomial). Now on the other hand, we shall shift the line of integration to $-\epsilon$, and pick up a residue from the integrand at $w=0$ and $w=1$. When $w=0$, we have that the residue is coming from the Gamma function and will be $Z(s,0)$, and when $w=1$ we have that the residue will be $X\res_{w=1}Z(s,w)$. The more interesting residue will occur at $w=1$ as this is coming from $Z(s,w)$, in particular, we will have that using the sum over $m$ formula in \Cref{eq: Def of Zsw} that the residue at $w = 1$ is supported entirely by the terms where $m_0=1$ that is when $m$ is a square. In particular, it will turn out that the residue will be a constant multiple of the symmetric square at $2s$. Thus, we have that 
    \begin{align*}
    \frac{1}{2\pi i}\int_{(2)}Z(s,w)X^w\Gamma(w)dw & = \frac{1}{2\pi i}\int_{(-\epsilon)}Z(s,w)X^w\Gamma(w)dw+Z(s,0) + cL(2s,\Sym^2 \Psi')X.
    \end{align*}
    Now we will specialize to the value of $s=1/2$. We observe that by the functional equations for $Z(s,w)$, we have $Z(1/2,w)\xrightarrow{(\alpha\beta)^2} t_j^{1-2w}Z(1/2,1-w)$. In particular, as we are integrating when the real part is $-\epsilon$, we will get that this will contribute a factor of $t_j^{1-2\epsilon}$. Furthermore, the integral term will contribute to the error which will be of size $O(X^{1/2+\epsilon})$, one might for example compare with the error term for the modular form case as in  \cite{QuanliTwistedFirstMoment}. In particular, if we consider this formula for $X$ and $2X$ and then subtract the two terms, we deduce that 
    \begin{equation*}
        \sum_{\substack{d\geq 1\\ (d,2)=1}} L(1/2,\Psi'\times \chi_{d_0})P(1/2,d_0d_1^2)(e^{-d/2X} - e^{-d/X})= c L(1,\Sym^2 \Psi')X+ t_j^{1-2\epsilon}O_{\epsilon}(X^{1/2+\epsilon}).
    \end{equation*}
    In particular, we see that if we take for example $X \sim  t_{\Psi'}^{10}\sim \lambda_{\Psi'}^5$, then we will have that the main term of the sum dominates. Furthermore, this sum is a smoothed first moment weighted by the $P(1/2,d_0d_1^2)$ supported on the $d\in [X,2X]$. Thus, it must be the case that there will be some $d \ll X\ll \lambda_{\Psi'}^5$ where we will have that $L(1,\Sym^2\Psi')\ll L(1/2, \Psi'\times \chi_d)$. In particular, using \cite{hoffstein1994coefficients} to lower bound the value of the symmetric square, we are done. 
\end{proof}

\begin{prop}\label{prop: Nelson 5.1 no D}
    We have that there is a constant $c_{\Acal,D}>0$ such that 
    \begin{equation*}
        \frac{\vert \res(f_D)\Psi'\vert^2}{\norm{f_D}^2\norm{\Psi'}^2} =c_{\Acal,D} \frac{\Lambda(\asai(\pi_D)\times \sigma,1/2)}{\Lambda(\sigma\otimes\chi_D,1)}.
    \end{equation*}
    However, there are constants $c_{\Acal}'$ and $c_{\Acal}''$ such that 
    \[
    c_{\Acal}'\leq c_{\Acal,D} \leq c_{\Acal}'',
    \]
    that is we may choose the constant so it's order of magnitude depends only upon $\Acal$.
\end{prop}

\begin{proof}
    \cite[Proposition 5.1]{nelson2022quadratic} gives a proof of the equality with a constant $c_{\Acal,D}$ which will be given by a product over the primes of terms $I_p$ where $I_p=1$ for all primes that do not ramify in $\Acal$. Furthermore, \cite[Proposition 4.8 (2)]{MR4033818} tells us that there are two options for $I_p$ when $p$ is not ramified in $\Acal$, in particular
    \[
    I_p = \begin{cases}
        2p\inv (1+p\inv)\inv(1+p^{-2})\inv\\
        p\inv 
    \end{cases}
    \]
    depending on the behavior of $p$ in the extension $\qq(\sqrt{D})$. However, we conclude the proposition by picking $c_{\Acal}'$ to just be the product of the smaller possibilities of the $I_p$ and $c_{\Acal}''$ to be the product of the larger ones. 
\end{proof}

\begin{prop}\label{prop: appendix L2 fD}
    We have that
    \begin{equation*}
        \norm{f_D}^2 = c_{\Acal,D}L(\Sym^2(\pi_D),1),
    \end{equation*}
    where we have that 
    \[
    c_{\Acal,D}\sim_{\Acal}\frac{\xi_{\qq(\sqrt{D})}(2)D^{1/2}}{\xi^*_{\qq(\sqrt{D})}(1)},
    \]
    where $\xi$ is the completed Dedekind zeta function associated to the field $\qq(\sqrt{D})$.
\end{prop}
\begin{proof}
    The existence of $c_{\Acal,D}$ is precisely \cite[Proposition 4.4]{nelson2022quadratic}; however, Nelson cites \cite[Lemma 2.2.3]{michel2010subconvexity}. Looking at this reference, we get that the entire dependency upon $D$ is given by the quotient that we specified. 
\end{proof}

\begin{prop}\label{prop: Appendix all together bound}
    We have that in our context, we have the bound 
    \begin{equation*}
        \frac{c}{c_{\Acal}c_{\Acal,D}}\frac{\Lambda(\sigma,1/2)}{\norm{\Psi'}^2\Lambda(\ad\sigma,1)\Lambda(\sigma\otimes\chi_D,1/2)}\ll_{\Acal} \lambda_{\Psi'}^4.
    \end{equation*}
\end{prop}

\begin{proof}
    We shall get this by just putting together all of our bounds.
    We see from \Cref{prop: appendix MDS}, will give us the lower bound of $\lambda_{\Psi'}^{-1}\ll_{\Acal} \Lambda(1/2,\sigma\otimes \chi_D)$ (we remark that we may choose the value of $\epsilon = 1$. Furthermore, as $\Psi'$ is the Jacquet-Langlands correspondent to $\Psi$ which was a part of an orthonormal basis, another application of the Hoffstein-Lockhart bound \cite{hoffstein1994coefficients}, we will have the lower bound of 
    \[
    \lambda_{\Psi'}^{-1}\ll_{\Acal}\norm{\Psi'}^2\Lambda(\Sym^2\sigma,1).
    \]
    Now we observe that as we are dividing by $c_{\Acal,D}$, and we have the dependence on $D$ given in \Cref{prop: appendix L2 fD}, when looking for an upper bound, we will just have a contribution of $\xi^*_{\qq(\sqrt{D})}(1)$ which is on the order of $L(1,\chi_D)\ll \log D\ll \log \lambda_{\Psi'}\ll \lambda_{\Psi'}$. Finally, we can trivially apply the convexity bound to say that 
    \[
    \Lambda(\sigma,1/2)\ll_{\Acal}\lambda_{\Psi'}.
    \]
    Now combining all of these bounds we conclude the proposition. 
\end{proof}

\bibliography{Ref}
\bibliographystyle{alpha}
\end{document}